\documentclass[12pt]{amsart}
\pdfoutput=1

\newif\ifdraft

\draftfalse

\usepackage[T1]{fontenc}
\usepackage[english]{babel}

\ifdraft
\usepackage[left=2in,top=1in,right=2in,bottom=1in,letterpaper]{geometry}
\else
\usepackage[left=1in,top=1in,right=1in,bottom=1in,letterpaper]{geometry}
\fi

\usepackage{amscd, amssymb, amsthm, amsmath}
\usepackage{pdfsync}
\usepackage{tikz-cd}
\usepackage{graphicx}
\usepackage{color}
\usepackage{libertine}

\usepackage{url}
\usepackage{verbatim} 
\usepackage{pinlabel}

\ifdraft
\usepackage{todonotes}
\newcommand{\compat}[1]{\textcolor{purple}{\textrm{{\bf Pat:} #1}}}
\newcommand{\compav}[1]{\textcolor{blue}{\textrm{{\bf Pavel:} #1}}}
\else
\usepackage[disable]{todonotes}
\newcommand{\compat}[1]{\ignorespaces}%
\newcommand{\compav}[1]{\ignorespaces}%
\fi

\newcommand{\notepat}[1]{\todo[color=yellow]{{\bf Pat:} #1}}

\ifx\pdfoutput\undefined
\usepackage[pdftex,bookmarks]{hyperref}
\else
\usepackage[pdftex,colorlinks=true,linkcolor=blue,urlcolor=blue,citecolor=violet]{hyperref}
\fi
\usepackage[nameinlink]{cleveref}

\renewcommand{\UrlFont}{\ttfamily\small}

\ifdraft

\usepackage[notref,notcite]{showkeys}
\fi

\theoremstyle{plain}
	\newtheorem{thm}{Theorem}[section]
	\newtheorem*{thm*}{Theorem}
	\newtheorem{lemma}[thm]{Lemma}
	\newtheorem*{lemma*}{Lemma}
	\newtheorem{prop}[thm]{Proposition}
	\newtheorem{cor}[thm]{Corollary}
\theoremstyle{definition}
	\newtheorem*{defn}{Definition}
	\newtheorem{remark}[thm]{Remark}

\usepackage{color}
\definecolor{dred}{rgb}{.65, 0, 0.15}

\def\cA{\mathcal{A}}\def\cB{\mathcal{B}}\def\cE{\mathcal{E}}\def\cG{\mathcal{G}}\def\cH{\mathcal{H}}\def\cO{\mathcal{O}}\def\cX{\mathcal{X}}

  \def\CC{\mathbb{C}}     \def\HH{\mathbb{H}}      \def\NN{\mathbb{N}}   
\def\QQ{\mathbb{Q}} \def\RR{\mathbb{R}} \def\SS{\mathbb{S}} \def\TT{\mathbb{T}}      \def\ZZ{\mathbb{Z}}

\def\GL{\mathrm{GL}} \def\SL{\mathrm{SL}}  
\def\SO{\mathrm{SO}}

\def\<{\langle} \def\>{\rangle}

\def\Aff{\mathrm{Aff}}
\def\Aut{\mathrm{Aut}}
\def\Dev{\mathrm{Dev}}
\def\Dehn{\mathrm{\Phi}}
\def\dist{\mathrm{dist}} 
\def\dir{\mathbf{dir}} 
\def\dim{\mathrm{dim}}

\def\GL{\mathrm{GL}} 
\def\Homeo{\mathrm{Homeo}}
\def\PGL{\mathrm{PGL}}

\def\Isom{\mathrm{Isom}} 
 
\def\localdegree{{\mathit{loc\,deg}}}

\def\proj{\mathrm{proj}}%
 
\def\Span{\mathrm{span}}%
\def\SL{\mathrm{SL}} 
\def\PSL{\mathrm{PSL}}

\def\f{\varphi}

\def\bG{\mathbf{G}}

\def\bH{\mathbf{H}}
\def\bHs{\mathbf{H}^\circ}
\def\bN{\mathbf{N}}

\def\bS{\mathbf{S}}
\def\bSs{\mathbf{S}^\circ}
\def\hbSs{\hat{\mathbf{S}}^\circ}
\def\bC{\mathbf{C}}
\def\tbS{\tilde{\mathbf{S}}}

\def\bT{\mathbf{T}}
\def\bTs{\mathbf{T}^\circ}
\def\tbT{\tilde{\mathbf{T}}}

\def\bU{\mathbf{U}}
\def\bUs{\mathbf{U}^\circ}

\def\bHt{{\widetilde{\mathbf{H}}}}
\def\bHh{{\widehat{\mathbf{H}}}}
\def\bNt{\widetilde{\mathbf{N}}}
\def\bNh{\widehat{\mathbf{N}}}
\def\disp{\mathbf{disp}}
\def\holhom{\mathrm{hol}}

\def\hyp{\psi}

\newcommand{\bone}{{\boldsymbol 1}}
\newcommand{\0}{{\boldsymbol 0}}
\def\be{\mathbf{e}}
\def\bp{\mathbf{p}}
\def\bu{\mathbf{u}}
\def\bv{\mathbf{v}}
\def\bw{\mathbf{w}}
\def\bx{\mathbf{x}}

\def\bs{\backslash}%

\def\authom{\Psi}
\def\matrixhom{M}

\title{The Necker Cube Surface}

\author{W. Patrick Hooper}
\address{Dept. of Mathematics, City College of New York and CUNY Graduate Center, New York, NY, USA}
\email{whooper@ccny.cuny.edu}
\urladdr{\url{http://wphooper.com}} 

\author{Pavel Javornik}
\urladdr{\url{https://sites.google.com/site/pjavornikmath/}}

\date{\today}                                           

\begin{document}
\begin{abstract}
We study geodesics on the Necker cube surface, $\bN$, an infinite periodic Euclidean cone surface that is homeomorphic to the plane and is tiled by squares meeting three or six to a vertex. We ask: When does a geodesic on the surface close? When does a geodesic drift away periodically? We show that both questions can be answered only using knowledge about the initial direction of a geodesic. Further, there is a natural projection from $\bN$ to the plane, and we show that regions related to simple closed geodesics tile the plane periodically. We also describe the full affine symmetry group of the half-translation cover and use this to study dynamical properties of the geodesic flow on $\bN$. We prove results related to recurrence, ergodicity, and divergence rates.
\end{abstract}
\maketitle

\section{Introduction}

\begin{figure}[b]
\includegraphics[width=\textwidth]{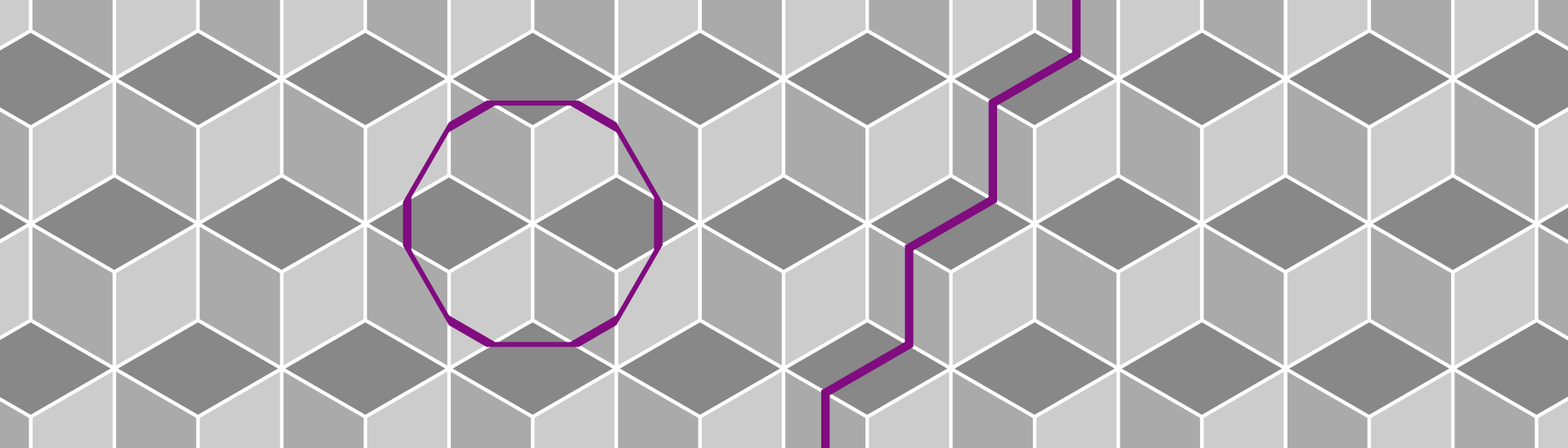}
\caption{The surface $\bN$ rendered in isometric projection together with the simplest closed and drift-periodic geodesics}
\label{fig:NCS}
\end{figure}

The {\em Necker cube surface}, $\bN$, is a periodic surface built from unit squares in $3$-space, see \Cref{fig:NCS}. We consider the unit speed geodesic flow $F:\RR \times T^1 \bN \to T^1 \bN$ . A {\em geodesic} is a trajectory of $F$, i.e., $\gamma(t)=F^t(\bv_0)$ for some $\bv_0 \in T^1 \bN$ representing the initial point and direction. Geodesics are not defined through the cone singularities on the surface, and a geodesic hitting one in forward or backward time is called {\em singular}.
A geodesic is {\em periodic} (or {\em closed}) if there is an $\ell \neq 0$ such that $\gamma(\ell+t)=\gamma(t)$ for all $t \in \RR$. A geodesic is {\em drift-periodic} if there is an $\ell \neq 0$ and a nonzero vector $\bx \in \RR^3$ such that $\gamma(\ell+t)=\bx  + \gamma(t)$ for all $t \in \RR$.

We normalize the Necker cube surface so that its square faces are parallel to the coordinate planes of $\RR^3$. Geodesics follow line segments within the faces. Suppose $\gamma(t)=F^t(\bv_0)$ is a geodesic and $t_0 \in \RR$ is such that $\gamma(t_0)$ is based at the interior of a face and whose direction is given by the unit vector $(x_0, y_0, z_0)$. Then the {\em slope} of $\gamma$ at $\gamma(t_0)$ is one of the three values $\frac{y_0}{x_0}$, $\frac{z_0}{y_0}$, or $\frac{x_0}{z_0}$, where we utilize the ratio involving the coordinates whose directions span the face.

\begin{thm}
\label{thm:main}
Let $m$ be the slope of a nonsingular geodesic $\gamma$ as it passes through the interior of a face. Then,
\begin{enumerate}
\item The geodesic $\gamma$ is periodic if and only if $m=\frac{p}{q}$ for two odd integers $p$ and $q$. Moreover, assuming these statements are true, the length of $\gamma$ is $6\sqrt{p^2+q^2}$.
\item The geodesic $\gamma$ is drift-periodic if and only if $m=\infty$ or $m=\frac{p}{q}$ for two relatively prime integers $p$ and $q$, one of which is even.
\end{enumerate}
\end{thm}

\begin{figure}[bth]
	\includegraphics[height=4in]{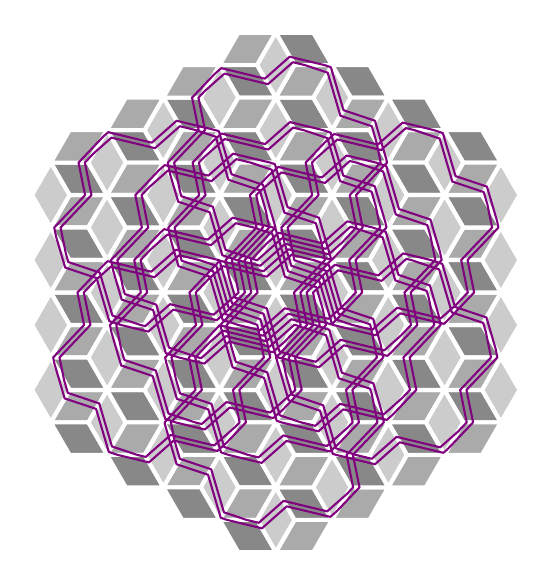}\includegraphics[height=4in]{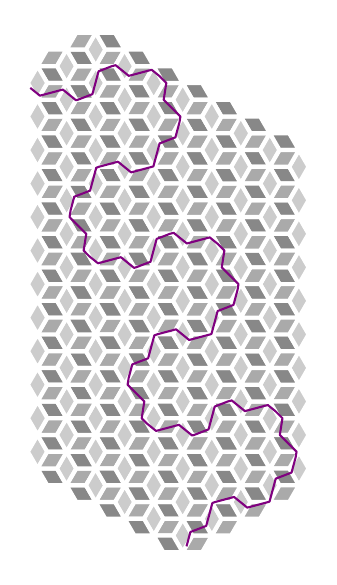}
	\caption{A periodic geodesic of slope $\frac{71}{23}$ (left) and a drift-periodic geodesic of slope $\frac{4}{11}$ (right)}
	\label{fig:N}
\end{figure}

\subsection*{Notion of direction}
It is a fundamental observation that the geodesic flow on $\bN$ splits into invariant subspaces. Given a unit tangent vector $\bv \in T^1 \bN$, which is based at a point in the interior of a square of $\bN$, we may take an orientation-preserving isometry from the containing square to the usual unit square $[0,1] \times [0,1] \subset \RR^2$. The derivative of this isometry carries $\bv$ to a vector $\bv'$ in the unit circle $\SS^1 \subset \RR^2$. This orientation-preserving isometry is unique up to postcomposition with a rotation of the unit square, so we define the {\em direction} of $\bv$ to be the equivalence class of $\bv'$ in $\SS^1/C_4$, where $C_4$ is the rotation group of order four. This notion of direction extends to a continuous map 
\begin{equation}
\label{eq:direction map}
\dir:T^1 \bN \to \SS^1/C_4,
\end{equation}
that is invariant under the unit speed geodesic flow $F$. 

Note that the collection of slopes in a direction $[(x,y)]$ is $\{\frac{y}{x}, -\frac{x}{y}\}$, so slope is not an invariant of geodesic flow. Nonetheless, rationality of the slope is an invariant of direction, as is the property of slope being the ratio of two odd integers.

\begin{figure}[bth]
	\includegraphics[width=\textwidth]{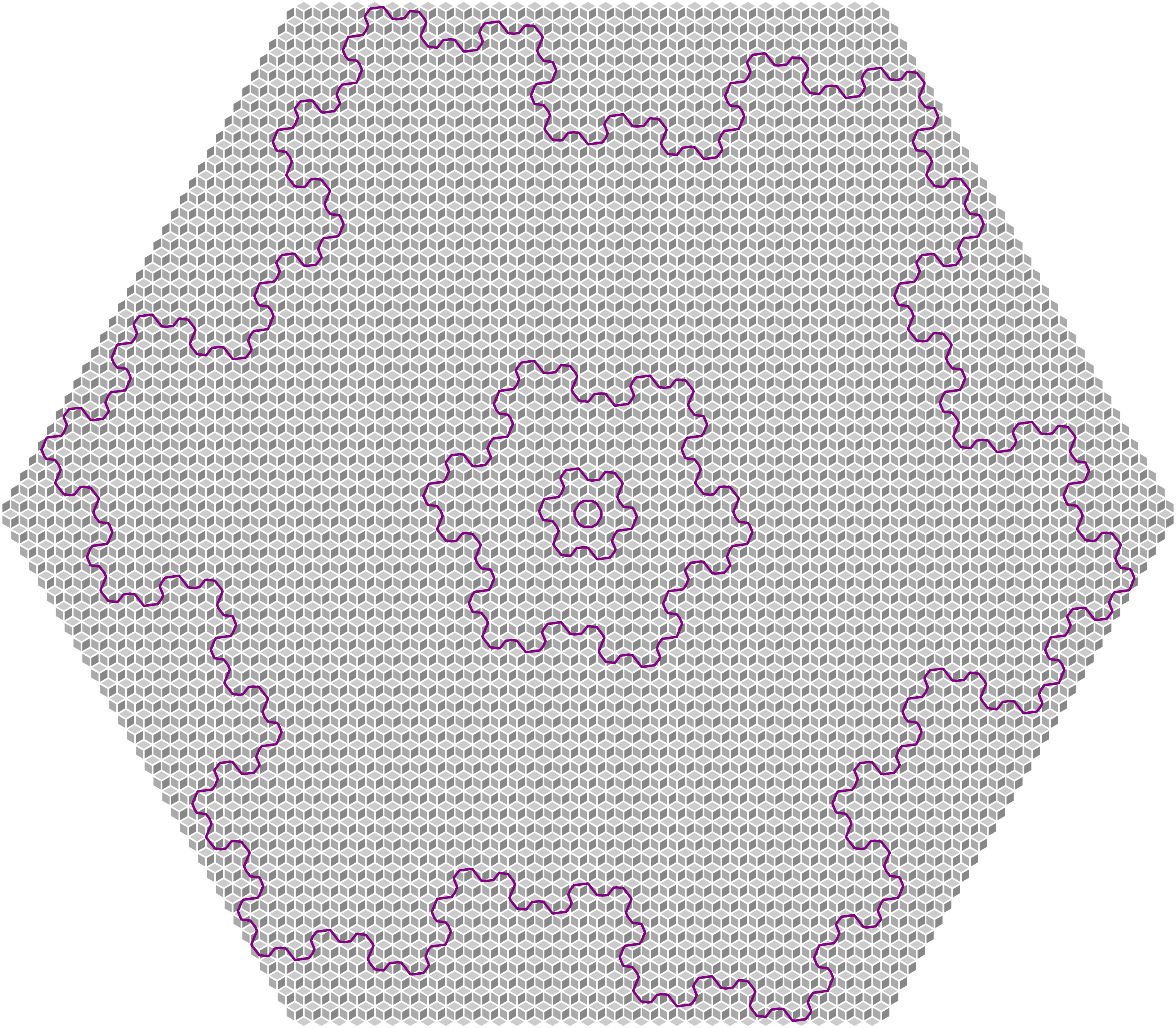}
	\caption{From the inside to the outside, simple closed geodesics of slope $\frac{1}{1}$, $\frac{3}{5}$, $\frac{13}{21}$, and $\frac{55}{89}$ (the first four ratios of consecutive odd Fibonacci numbers).}
	\label{fig:sequence}
\end{figure}

\subsection*{Closed geodesics and tilings}
Figures \ref{fig:NCS}, \ref{fig:N}, and \ref{fig:sequence} illustrate periodic trajectories. Observe the periodic trajectories all have rotational symmetry. This generally holds:

\begin{thm}[Sixfold symmetry]
\label{thm:order six}
If $\gamma:\RR \to \bN$ is a periodic geodesic with least period $\ell$, then there is an order six isometry $\xi:\RR^3 \to \RR^3$ preserving
$\bN$ such that $\xi \circ \gamma(t)=\gamma(t+\frac{\ell}{6})$ for all $t \in \RR$.
\end{thm}

It is a standard observation that any simple closed geodesic extends to a maximal open cylinder $C$ foliated by simple closed geodesics in the same homotopy class. We show (\Cref{simple six}) that each component of the boundary $\partial C$ is also a simple closed curve. By the Jordan curve theorem, each component bounds a disk in $\bN$. We name the two boundary components $\partial_{\textit{out}} C$ and $\partial_{\textit{in}} C$, and name the two disks they bound $D_{\textit{out}}(C)$ and $D_{\textit{in}}(C)$, respectively, following the convention that $\partial_{\textit{in}} C$ is contained in the interior of $D_{\textit{out}}(C)$. The two boundaries $\partial_{\textit{out}} C$ and $\partial_{\textit{in}} C$ have additional symmetry enabling the following result:

\begin{thm}[Tiling theorem]
\label{thm:tiling}
Let $[\bu] \in \SS^1/C_4$ be a direction whose slope is the ratio of two odd integers.
Let $\ast$ denote either $\textit{in}$ or $\textit{out}$.
If there is a simple closed geodesic on $\bN$ with direction $[\bu]$, then there is a collection ${\mathcal C}$ of maximal cylinders with direction $[\bu]$ such that the collection of disks $\{D_\ast(C):~C \in {\mathcal C}\}$ have pairwise disjoint interiors and cover $\bN$ periodically. The symmetry group of this collection contains the wallpaper group {\bf p6} in the notation of Coexter \cite{CM1972} (equivalently, $632$ in orbifold notation \cite{Conway}) and is precisely this group unless $[\bu]$ consists of vectors of slope $\pm 1$.
\end{thm}

The images of these regions project to a tiling of the plane. See \Cref{fig:tiling example} for an example. We did not investigate which directions give rise to simple geodesics. Experimentally it seems ratios of consecutive odd Fibonacci numbers satisfy this property. More generally, it seems that ratios of odd integers whose continued fraction expansion contains only ones and twos satisfies this property.
\Cref{fig:sequence} seems to suggest a limiting fractal tile appearing in a rescaled limit.
Experimentally this seems to occur for other sequences of directions as well. We hope to further investigate these ideas in the future. 

\begin{figure}[tbh]
	\includegraphics[width=\textwidth]{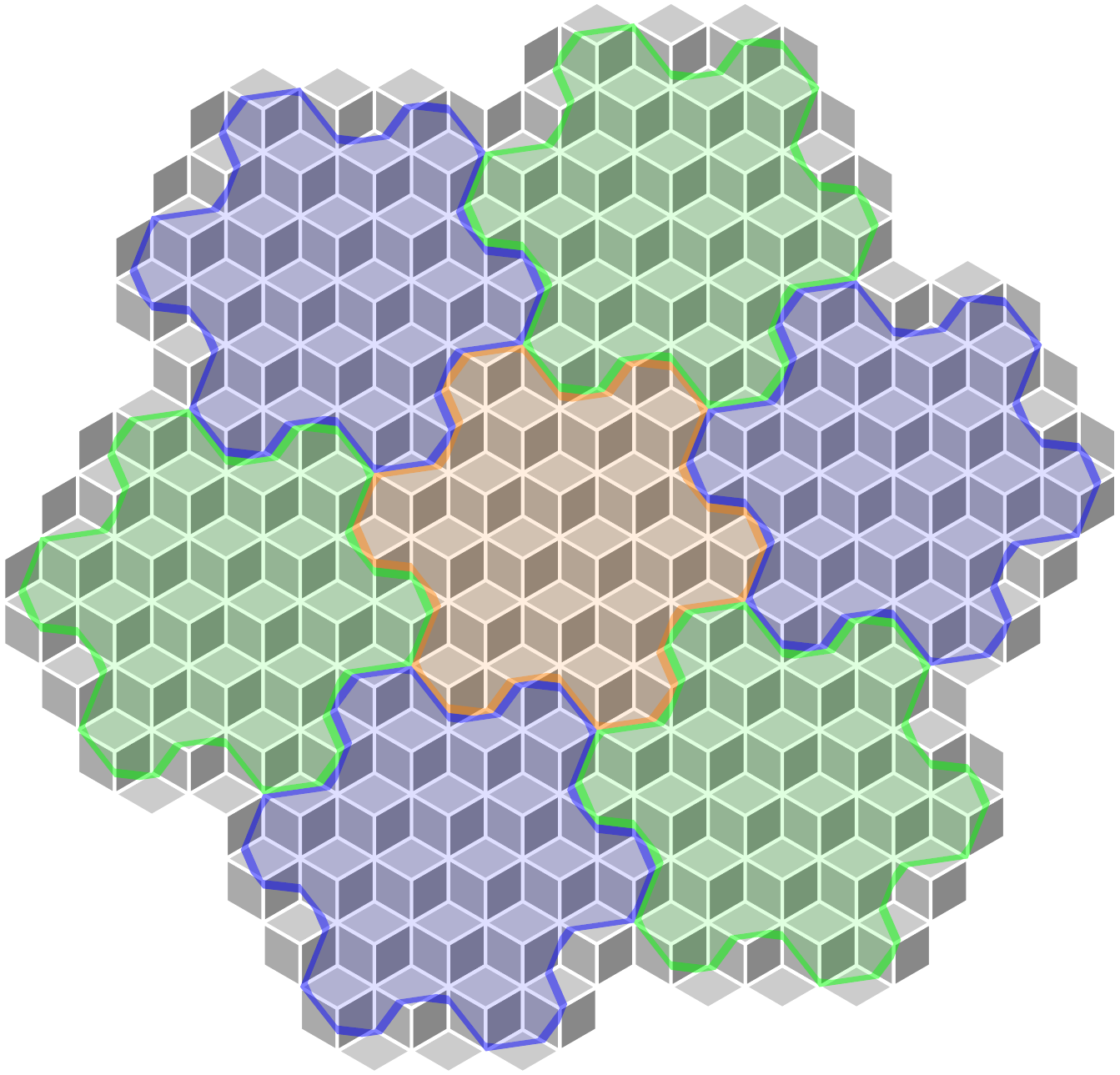}
	\caption{Part of a tiling whose tiles are bounded by outer boundaries of cylinders with slope $\frac{3}{5}$. Cylinders are shown in darker colors.}
	\label{fig:tiling example}
\end{figure}

\subsection*{Affine symmetry}
Much of the paper focuses on the affine symmetry group of the minimal half-translation cover $\bNt$ of $\bN$. This is the double cover whose branch locus is the set of singularities where three squares come together. \Cref{thm:main} guarantees that in any slope $m \in \QQ$ that is the ratio of two odd integers, $\bNt$ decomposes into isometric cylinders of slope $m$. Therefore, there is an affine symmetry $\varphi_m:\bNt \to \bNt$ that performs a right Dehn twist in each such cylinder. We show:

\begin{thm}
\label{thm:symmetries}
Any orientation-preserving affine symmetry of $\bNt$ is the composition of an isometry of $\bNt$ and an element of the group 
$$\langle \varphi_m:~\text{$m$ is the ratio of two odd integers}\rangle.$$
\end{thm}

The group of derivatives of orientation-preserving affine automorphisms of $\bNt$ forms its {\em Veech group}, $V(\bNt) \subset \PSL(2,\RR)$. The Veech group of $\bNt$ is therefore determined by \Cref{thm:symmetries} and knowledge of the isometry group of $\bNt$. We show $V(\bNt)$ is an infinitely-generated Fuchsian group of the first-kind (i.e., it has a dense limit set). We also describe the topology and geometry of the Teichm\"uller disk $\HH/V(\bNt)$, where $V(\bNt)$ is acting on the upper half-plane $\HH$ by M\"obius transformations. For more detail, see \Cref{sect:affine group}.

\subsection*{Dynamical results}
The affine group of a surface is often used to renormalize the geodesic flow. We use this and related ideas to establish dynamical results about the associated geodesic flow on $\bN$.
Recall the direction map $\dir:T^1\bN \to \SS^1/C_4$ of \eqref{eq:direction map} is an invariant of the geodesic flow. Fixing an equivalence class $[\bu] \in \SS^1/C_4$, we let
$F^t_{[\bu]}$ denote the restriction of the geodesic flow to the set $T^1_{[\bu]}\bN=\dir^{-1}([\bu])$. Each subset $T^1_{[\bu]}\bN$ is naturally homeomorphic to the minimal translation surface cover of $\bN$ (see \Cref{sect:setup}), and so come equipped with natural metrics and measures. Each flow $F^t_{[\bu]}$ is defined for all time almost everywhere. The spaces $\SS^1$ and $\SS^1/C_4$ have natural finite Lebesgue measures coming from angle coordinates on the circle; $\bu=(\cos \theta, \sin \theta)$.

Recall that if $F^t:X \to X$ is a flow in a metric space with an invariant measure $m$, a point $x \in X$ is {\em positive recurrent} (resp. {\em negatively recurrent}) if for every $\epsilon>0$ and each $T>0$, there is a $t>T$ (resp. $t<-T$) such that $d\big(x,F^t(x)\big)<\epsilon$. We say the flow is {\em recurrent} with respect to $m$ if $m$-almost every point is both positively and negatively recurrent. We show our directional flows tend to be recurrent, tend to have a specific divergence rate, and are sometimes ergodic. 

\begin{thm}[Recurrence]
\label{thm:recurrence}
For almost every $[\bu] \in \SS^1/C_4$, the flow $F_{[\bu]}$ is recurrent.
\end{thm}

\begin{remark}
Experimental evidence suggests that not all directions of irrational slope are recurrent. For example, trajectories whose slope is the limit of the ratio of consecutive Fibonacci numbers seem to be nonrecurrent and diverge quicker than \Cref{thm:divergence} suggests. Paths tend to look like limits of the curves in \Cref{fig:sequence}.
\end{remark}


\begin{thm}[Ergodicity]
\label{thm:ergodicity}
The set of $[\bu] \in \SS^1/C_4$ for which $F_{[\bu]}$ is ergodic is dense and has Hausdorff dimension larger than $\frac{1}{2}$.
\end{thm}


\begin{thm}[Divergence rate]
\label{thm:divergence}
For almost every $[\bu] \in \SS^1/C_4$, for every $\bv \in T^1_{[\bu]} \bN$ with an infinite forward $F_{[\bu]}$-trajectory, we have
$$\limsup_{T \to +\infty} \frac{\log d\big(F^T_{[\bu]}(\bv),\bv\big)}{\log T}=\frac{2}{3}$$
where $d$ denotes the Euclidean distance in $\RR^3$ between the basepoints of the tangent vectors.
\end{thm}

\subsection{The Necker cube surface in the Arts and Culture}
The Necker cube surface depicted in \Cref{fig:NCS} is a variant of a family of famous illusions attributed to Swiss crystallographer Louis A. Necker \cite{LN32}. However, the pattern has been used as a tiling pattern (called the {\em rhombille tiling}) for floors since the ancient Greeks \cite[255]{dunbabin1999mosaics}. The surface received a brief mention in Popular Science Monthly in 1899, in an article exhibiting a variety of optical illusions \cite{JJ99}. This periodic pattern can be spotted in numerous works by M.C. Escher, and is the game board in the cabinet arcade classic, Q*bert. The illusion of cubes that may alternately be seen as popping in or out of the page captured the interest of people throughout history.

\subsection{Related work}

Recently there has been significant interest in the periodic billiard tables and similar systems. Indeed, this work was in part motivated by \cite{DH18} which considered tiling billiards on the trihexagonal tiling. (The rhombille tiling is dual to the trihexagonal tiling.) There are few other papers studying the geometry of periodic polygonal surfaces from a similar point of view; see \cite{AL21} and \cite{ORSS}.

Many of these systems can be understood by passing to a periodic translation (or half-translation) surface cover. In \cite{DH18}, it was realized that for the trihexagonal tiling, the tiling billiard system is governed by an infinite $\ZZ^2$ periodic translation surface. The surface obtained from the trihexagonal tiling is highly symmetric: It has the {\em lattice property}, i.e., the surface is stabilized by the linear action of a co-finite area subgroup of $\SL(2,\RR)$.

Our understanding of $\bN$ is similarly powered by working with a natural half-translation surface cover, but this cover is not stabilized by a lattice in $\SL(2,\RR)$. This  suggests that to obtain a more general understanding of geodesics of irrational slope on $\bN$, techniques used to study periodic translation and half-translation surfaces which do not have the lattice property should be applied. The most well-studied similar example is the periodic Wind-tree Model, which is a $\ZZ^2$-periodic billiard table obtained by inserting a periodic array of rectangular obstacles in the plane.
This system is known to be recurrent \cite{HLT11, AH17}, periodic orbits are understood \cite{HLT11},  numerous divergent trajectories have been found \cite{D13}, most directions are non-ergodic \cite{FU14}, and diffusion rates are understood \cite{HLT11,DHL14}. Other work has studied dynamics in periodic patterns of Eaton lens \cite{FS14, FSU15, A17, FS18}.

The dynamics of geodesics on $\bN$ are similar to the Wind-tree Model with parameters $a=b=1/2$ \cite{HLT11}. The translation covers of these two are both $\ZZ^2$-covers of the same square-tiled surface. But, there are some differences. For example, the question if a geodesic on this Wind-tree model closes cannot be decided based on the direction alone, but depends on the initial position as well. Our \Cref{thm:main} guarantees that in every direction in which there is a periodic geodesic, the flow is {\em completely periodic}: every nonsingular geodesic is closed.

Our proof of the recurrence result (\Cref{thm:recurrence}) is very close to the recurrence argument for Wind-tree models given in \cite{HLT11}. Recurrence turns out to follow from the existence of a large number directions with cylinder decompositions. We show that in typical directions trajectories tend to wrap arbitrarily tightly around some such cylinders. We further develop these ideas from \cite{HLT11} by connecting them to hyperbolic geometry in \Cref{sect:recurrence}.

The Necker cube surface falls cleanly within the class of surfaces studied in \cite{H15}, which uses twistable cylinder decompositions to produce directions in which locally finite ergodic invariant measures of geodesic flows can be classified. Because there are twistable
cylinder decompositions in a dense set of directions, and the nilpotent group $\ZZ^2$ acts cocompactly on $\bN$ by isometry, ergodicity in some directions follows directly from results in \cite{H15}. Ergodic directions can also be found via the approach introduced by Hubert and Weiss for $\ZZ$-covers \cite{HW13}. We do not bother to work out the details, but see \cite{A17} and \cite[Appendix A]{DH18} for details on the extension of the Hubert-Weiss ideas to $\ZZ^d$-covers.

Our proof of the divergence rate in \Cref{thm:divergence} makes heavy use of work already done by Delecroix, Hubert, and Leli\`evre when investigating divergence rates on Wind-tree models \cite{DHL14}. In brief, the minimal translation cover of $\bN$ covers a translation surface in the locus of translation surfaces studied in \cite{DHL14}. Their deviation rate arguments depend on the Lyapunov exponents within this locus, and deviation rates in $\bN$ are governed by the same Lyapunov exponents.

We are not aware of tiling-type theorems in similar systems, nor are we aware of similar exact calculations of non-lattice Veech groups.

%

\subsection{Future work and related ideas}
The Necker cube surface is a member of a three-parameter family of step surfaces that can be parameterized by three positive values, each dictating the length, width, and height of the rectangular tiling. See the left side of \Cref{fig:future}. Presumably some dynamical aspects of these surfaces can be understood via similar methods. It seems like ideas similar to those used to study the Ehrenfest wind-tree model are likely to apply.

Another generalization is obtained by allowing the step-sizes in each of the coordinate directions to be nonconstant. If the step sizes in each direction are summable, then you get a finite area surface which presumably makes it easier for trajectories to diverge. We wonder if the Hausdorff dimension of the nondivergent trajectories can be bounded here. If the step sizes are random or quasi-periodic, we wonder if there are still recurrent directions. See the right side of \Cref{fig:future} for a depiction of a quasiperiodic example. Surfaces of this type can be drawn using the app at \cite{Hooper_Necker_App}.

\begin{figure}[htb]
\begin{minipage}{5.7in}
  \centering
  \raisebox{-0.5\height}{\includegraphics[height=1.25in]{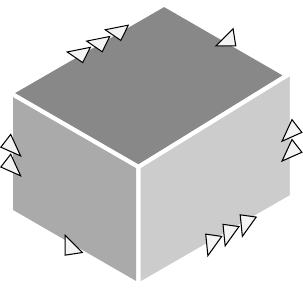}}
  \hspace*{.2in}
  \raisebox{-0.5\height}{\includegraphics[height=2.5in]{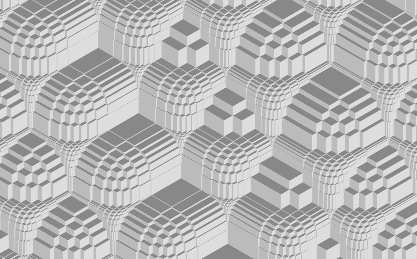}}
\end{minipage}
	\caption{Left: A variant of the Necker cube surface can be formed by taking the universal cover of this torus, built from three rectangles. Right: A variant of the Necker cube surface where step sizes in the $i$-th coordinate 
	have length of the form $-\log(a_i n+b_i \mod{1})$ for $n \in \NN$ and constants $a_i, b_i \in (0,1)$.}
	\label{fig:future}
\end{figure}

\section{Background on Euclidean cone surfaces}

\subsection{Euclidean cone surfaces}
We will give a standard geometric structure definition of Euclidean cone surface following notions of $(G,X)$-structures in \cite[Chapter 3]{Thurston}.

A {\em Euclidean cone surface} is a connected oriented surface $S$ together with a closed discrete subset $\Sigma \subset X$ and a maximal atlas of charts from $S \smallsetminus \Sigma$ to the plane such that the transition functions are locally given by elements of $\Isom_+(\RR^2)$, the group of orientation-preserving isometries of $\RR^2$. We can then pullback the metric under the charts to obtain a path metric on $S \smallsetminus \Sigma$. The space $S \smallsetminus \Sigma$ is locally isometric to the plane. To be a cone surface, we insist that the metric on $S \smallsetminus \Sigma$ can be extended to all of $S$ and the extension to each point in $\Sigma$ makes that point into a cone singularity. For this paper, we will also insist that this metric is complete.
See \cite[\S 2]{Troyanov07} for more detail on Euclidean cone surfaces.
We'll use bold letters such as $\bS$ to denote a Euclidean cone surface.

We can obtain more structured surfaces by restricting transition functions to a subgroup of $\Isom_+(\RR^2)$. We'll call this subgroup the {\em structure group} of the surface. If we insist that transition functions lie in the {\em translation group} $G_1 \subset \Isom_+(\RR^2)$ consisting of of translations of $\RR^2$, we get a {\em translation surface}. If we insist that transition functions lie in the {\em half-translation group} $G_2$ consisting of translations and rotations by $180^\circ$, we get a {\em half-translation surface}. 
If we insist that transition functions lie in the {\em quarter-translation group} $G_4$ consisting of translations and rotations by multiples of $90^\circ$, we get a {\em quarter-translation surface}. 

\begin{remark}
The data determining these geometric structures is equivalent to the choice of a Riemann surface paired with a holomorphic 1-form, a quadratic differential, or a $4$-differential in the respective cases of translation, half-translation, and quarter-translation surface. See \cite[\S 2.3]{AAH} for a description of geometric structures induced by $k$-differentials.
\end{remark}

\subsection{Square-tiled surfaces}
In this paper, a {\em square-tiled surface} is an oriented connected surface formed by gluing together unit squares edge-to-edge by isometries along the edges. Such a surface is locally isometric to the plane at all points other than the identified vertices. If $p$ is a point formed by identifying vertices of squares, its {\em valence} is the number of vertices being identified to form $p$. The cone angle at $p$ is $\frac{\pi}{2}$ times the valence. In particular, if the valence of $p$ is not equal to four, then $p$ is a singularity. We insist that the singular set of a square-tiled surface only consist of identified vertices, but we allow points of valence four to be excluded from the singular set.

\begin{remark}
It is possible for infinitely many vertices of a square to be identified, forming a point without a neighborhood homeomorphic to an open set in the plane. This point becomes an infinite cone singularity. Technically, this case does not lead to a surface, because an infinite cone singularity does not have a neighborhood homeomorphic to an open subset of the plane. This situation will not arise in the later sections of this paper, but universal branched covers (defined below) have such points.
\end{remark}

A square-tiled surface is naturally a Euclidean cone surface with structure group given by $\Isom_+(\ZZ^2)$, the group of orientation-preserving isometries of $\RR^2$ that preserve $\ZZ^2$. Since $\Isom_+(\ZZ^2) \subset G_4$, a square-tiled surface is also a quarter-translation surface. In this paper we predominantly think of squares as quarter-translation surfaces.

\subsection{Universal branched covers}
Let $\bS$ be a connected Euclidean cone surface with singular set $\Sigma$. Let
$\bSs$ denote $\bS \smallsetminus \Sigma$. Let $\bUs$ denote the universal cover of $\bSs$.
Then $\bUs$ inherits an atlas of charts with transition functions in $\Isom_+(\RR^2)$ (or a subgroup) by pulling back the charts of $\bSs$ with simply connected domains under the covering map $\bUs \to \bSs$. We can extend this atlas to a unique maximal atlas, and, using these charts, we obtain a metric on $\bUs$.

The {\em universal branched cover} of $\bS$ is the metric completion $\bU$ of $\bUs$. Since the covering map $\bUs \to \bSs$ is $1$-Lipschitz, and since $\bS$ is complete, there is a unique continuous extension $\pi_\bS: \bU \to \bS$. The complement of $\bUs$ in $\bU$ consists of the preimages of cone singularities. These points are the {\em singularities} on $\bU$. The local geometry near a singularity of $\bU$ is that of an infinite cone singularity. See \cite[\S 2.2]{Troyanov07} for another viewpoint on this.

Observe that deck transformations of the covering $\bUs \to \bSs$ extend continuously to isometries $\bU \to \bU$.
We call the collection of these extensions the {\em deck group}, $\Delta_\bS=\Delta(\bU, \bS)$, of the universal branched cover $\bU \to \bS$. (Note that a deck transformation $\delta:\bU \to \bU$ may fix points in $\bU \smallsetminus \bUs$.)

\subsection{Developing maps} 
A {\em developing map} for a connected Euclidean cone surface $\bS$ is a continuous map 
$$\Dev:\bU \to \RR^2$$
which is compatible with the charts defining the structure on $\bUs$. That is, we insist that if $\phi:U \to \RR^2$ is a chart with $U\subset \bUs$ and $p \in U$, then there should be an open $V$ with $p \in V \subset U$ such that the restrictions of $\Dev$ and $\phi$ differ by postcomposition by an element of the structure group. Developing maps can be obtained by analytically continuing a single chart to obtain a $1$-Lipschitz map $\bUs \to \RR^2$, and then extending to the metric completion $\bU$. Any two developing maps differ by postcomposition by an element of the structure group.
For more details on the developing map see \cite[\S 2.3]{Troyanov07} or \cite[\S 3.4]{Thurston}.

\subsection{The holonomy homomorphism}
Let $\bS$ be a connected Euclidean cone surface with singular set $\Sigma$ and structure group $G \subset \Isom_+(\RR^2)$. Fix a developing map $\Dev:\bU \to \RR^2$. Observe that if $\delta: \bU \to \bU$ is a deck transformation, then $\Dev \circ \delta$ is another developing map. It follows that there is a $\holhom(\delta) \in G$ such that 
\begin{equation}
\label{eq:holonomy}
\holhom(\delta) \circ \Dev = \Dev \circ \delta.
\end{equation}
The map $\holhom:\Delta_\bS \to G$ is a group homomorphism called the {\em holonomy homomorphism}. Note that $\holhom$ depends on the choice of $\Dev$. A change in the choice of $\Dev$, will result in $\holhom$ changing by post-conjugation by an element of $G$.

Given a basepoint $p_0 \in \bSs$ and a preimage $\tilde p_0 \in \bUs$, covering space theory gives a natural isomorphism between the fundamental group $\pi_1(\bSs, p_0)$ with the deck group $\Delta_\bS$. Namely, the relative homotopy class of $\gamma:[0,1] \to \bSs$ is identified with the deck transformation $\delta:\bU \to \bU$ that carries $\tilde p_0$ to the endpoint $\tilde \gamma(1)$ of a lift $\tilde \gamma:[0,1] \to \bU$ such that $\tilde \gamma(0)=\tilde p_0$. Thus we define the {\em holonomy} around $\gamma$ to be $\holhom(\gamma)=\holhom(\delta)$.

We will generally be interested in extracting a holonomy invariant from a free homotopy class of loops in $\bSs$ rather than elements of $\pi_1(\bSs, p_0)$ as above. Recall that free homotopy classes of loops are associated to conjugacy classes in $\pi_1(\bSs, p_0)$. Further observe that the conjugacy classes are fibers of the Jacobian derivative
$J:\Isom_+(\RR^2) \to \SO(2),$
taking values in the rotation group $\SO(2) \cong \RR/2\pi \ZZ$. The {\em rotational holonomy} of a loop $\eta:\RR/\ZZ \to \bSs$ is
$$\holhom_\circ(\eta) = J \circ \holhom(\eta_0),$$
where $\eta_0$ is a representative of the free homotopy class of $\eta$ such that $\eta_0(0)$ is the basepoint of $\bSs$. From the remarks above, rotational holonomy is invariant under free homotopy in $\bSs$. Since $\SO(2)$ is abelian, this induces a group homomorphism $(\holhom_\circ)_\ast:H_1(\bSs; \ZZ) \to \SO(2)$.

\subsection{Natural maps}
\label{sect:natural maps}
Let $\bS$ and $\bT$ be Euclidean cone surface with structure groups $G_S$ and $G_T$, respectively. Let $H$ be a subgroup of the affine group $\Aff(\RR^2)$ that contains both $G_S$ and $G_T$. We'll say that a map $f:\bS \to \bT$ is an {\em $H$-mapping} if $f(\bSs) \subset \bTs$, $f(\bS \setminus \bSs) \subset \bT \setminus \bTs$, and in local coordinate charts taken near points $p \in \bSs$ and $f(p) \in \bTs$, the map behaves like an element of $H$.
This condition can be understood using any choice of developing maps $\Dev_\bS:\bU_\bS \to \RR^2$ and $\Dev_\bT:\bU_\bT \to \RR^2$: The map $f$ is an $H$-mapping if and only if for one (equivalently, any) lift $F:\bU_\bS \to \bU_\bT$, there is a $h \in H$ such that 
\begin{equation}
\label{eq:qts mapping}
\Dev_\bT \circ F = h \cdot \Dev_\bS.
\end{equation}
(Such a lift can be obtained by first lifting the restriction of $f$ to a map $\bSs \to \bTs$, and then continuously extending.) There is a natural homomorphism between the deck groups $F_\ast:\Delta_\bS \to \Delta_\bT$ such that for any $\delta_\bS \in \Delta_\bS$, 
$$\delta_\bT \circ F=F \circ \delta_\bS, \quad \text{where} \quad \delta_\bT = F_\ast(\delta_\bS).$$
Then by \eqref{eq:holonomy} and \eqref{eq:qts mapping}, for holonomy maps defined using $\Dev_\bS$ and $\Dev_\bT$, we have 
\begin{equation}
\holhom_\bT(\delta_\bT)=h \cdot \holhom_\bS(\delta_\bS) \cdot h^{-1}
\quad \text{where $h$ is from \eqref{eq:qts mapping}.}
\label{eq:holonomy and maps}
\end{equation}

An {\em $H$-isomorphism} $\bS \to \bT$ is an $H$-mapping that is also a homeomorphism. An {\em $H$-automorphism} of $\bS$ is an $H$-isomorphism $\bS \to \bS$. The collection of all $H$-automorphisms forms a group, denoted $\Aut_H(\bS)$. 


An $H$-mapping that is also a covering map will be called an {\em $H$-covering map}. We allow such maps to have a branching locus in the singular set of the domain. If $f:\bS \to \bT$ is an $H$-covering and $\Sigma_\bS \subset \bS$ denotes the set of singularities, the {\em local degree function}
$$\localdegree_f:\Sigma_\bS \to \ZZ_{>0}$$
sends $x \in \Sigma_\bS$ to the smallest $n \geq 1$ such that $f$ has a restriction to a neighborhood of $x$ with degree $n$; the branching locus is the set where $\localdegree_f>1$.

Suppose $\bS$ and $\bT$ have structure group $H$. If $H \subset G$ is normal and $f:\bS \to \bT$ is a $G$-mapping, its {\em derivative} $D(f) \in G/H$ is the element $gH \in G/H$ such that
\begin{equation}
\label{eq:derivative}
\Dev_T \circ F=g \cdot \Dev_\bS
\end{equation}
for some choice of developing maps and some lift $F:\bU_\bS \to \bU_\bT$ as above. The coset $gH$ is independent of choices. Note that $f$ is also an $H$-mapping if and only if $D(f)$ is trivial in $G/H$. If $f:\bS_1 \to \bS_2$ and $g:\bS_2 \to \bS_3$ are $G$-mappings between $H$-structures, we have 
$$D(g \circ f)=D(g)D(f).$$
In particular, for any $\bS$ with structure group $H$ which is normal in $G$, the derivative
$D:\Aut_G(\bS) \to G/H$ is a group homomorphism and the following sequence is exact:
$$1 \to \Aut_H(\bS) \to \Aut_G(\bS) \xrightarrow{D} G/H \to 1.$$
The subgroup $D\big(\Aut_G(\bS)\big) \subset G/H$ is known as the {\em Veech group} of $\bS$.



\subsection{Veech groups of square-tiled surfaces}

Let $\bS$ be a square-tiled translation surface or half-translation surface. Let $\bU$ be its universal branched cover. Fix a developing map $\Dev: \bU \to \RR^2$ sending vertices of squares to points in $\ZZ^2$. Let $\Lambda \subset \ZZ^2$ be the smallest subgroup containing differences between vectors in $\Dev(\Sigma_\bU) \subset \ZZ^2$ where $\Sigma_\bU \subset \bU$ is the singular set. We'll call this the {\em group of singular differences} of $\bS$. Note that because developing maps differ by an element of $G_1$ or $G_2$, the group of singular differences is independent of the choice of developing map.

Gutkin and Judge studied Veech groups of finite square-tiled surfaces in \cite{GJ00}. The following is a weak form of one of their results that holds for infinite square-tiled surfaces.

\begin{prop}
\label{SL2Z}
Suppose $\bS$ is a square-tiled translation (respectively, half-translation) surface.
If the group of singular differences $\Lambda$ is a sublattice of $\ZZ^2$ (i.e. is isomorphic as a group to $\ZZ^2$), then the Veech group of $\bS$ is the contained in the subgroup of $\GL(2,\QQ)$ (respectively, $\GL(2,\QQ)/\langle -I \rangle$) stabilizing $\Lambda$.
\end{prop}

In the cases we care about, $\Lambda=\ZZ^2$ and so the Veech group is contained in $\GL(2,\ZZ)$ or $\GL(2,\ZZ)/\langle -I \rangle$.

\begin{proof}
Let $\phi:\bS \to \bS$ be an affine automorphism, and let $\tilde \phi:\bU \to \bU$ be a lift. By \eqref{eq:qts mapping}, there is an $h \in \Aff(\RR^2)$ such that $h \cdot \Dev=\Dev \circ \tilde \phi$.  We have that $D\phi$ is the image of $h$ in $\GL(2,\RR)$ or $\GL(2,\RR)/\langle -I \rangle$. 
Since $\tilde \phi(\Sigma_\bU)=\Sigma_\bU$, we have $D\phi(\Lambda)=\Lambda.$ Thus Veech group elements stabilize $\Lambda$ as claimed.
Since $\Lambda$ is a lattice, we can choose linearly independent $\bv_1, \bv_2 \in \Lambda$. Then $D\phi$ is determined by the images $D\phi(\bv_1)$ and $D\phi(\bv_2)$. Since all four vectors lie in $\ZZ^2$, $D\phi$ lies in $\GL(2,\QQ)$ or $\PGL(2,\QQ)$.
\end{proof}

\subsection{Minimal covers}
Let $\bS$ be a connected Euclidean cone surface with structure group $G$. Let $H \subset G$ be a normal subgroup. We will say $\bS$ is {\em genuine with respect to $H$} if the restriction of the homomorphism $G \to G/H$ to $\holhom(\Delta_\bS) \to G/H$ is surjective.

\begin{remark}
If $\bS$ is not genuine with respect to $H$, then one can construct a subgroup $G'$ such that $H \subset G' \subsetneq G$ and a Euclidean cone surface $\bS'$ with structure group $G'$ such that $\bS'$ is isometric to $\bS$. So, it would be more natural to consider $\bS$ to have the structure group $G'$. Here $G' \subset G$ can be taken to be the preimage of the image of $\holhom(\Delta_\bS)$ in $G/H$.
\end{remark}

Let $\Gamma_\bS = \holhom^{-1}(H),$ i.e., $\Gamma_\bS \subset \Delta_\bS$ consists of the deck transformations $\bU \to \bU$ with holonomy in $H$. Note that $\Gamma_\bS$ is independent of the choice of developing map, because $H$ is normal in $G$.

The quotient $\tbS=\bU/\Gamma_\bS$ is naturally a Euclidean cone surface whose charts are obtained as composition of a lift of an open subset of $\tbS$ to $\bU$ with a fixed developing map $\Dev:\bU \to \RR^2$. By construction, $\tbS$ has structure group $H$.
The natural map $\tilde \pi_\bS: \tbS \to \bS$ is a $G$-covering map. Following \cite{GJ00}, we call $\tbS$ the {\em minimal $H$-cover} of $\bS$. The cover is clearly a normal cover with deck group isomorphic to $G/H$. In fact, we have:

\begin{lemma}
\label{deck group and derivative}
Let $\bS$ is a Euclidean cone surface with structure group $G$. Let $H \subset G$ be normal. Suppose that $\bS$ is genuine with respect to $H$ and let $\tbS$ be the minimal $H$-cover of $\bS$. Then the restriction of the derivative $D:\Aut_G(\tbS) \to G/H$ to the deck group $\Delta(\tbS, \bS)$ is an isomorphism of groups.
\end{lemma}
\begin{proof}
The universal branched cover $\bU$ of $\bS$ also serves as such for $\tbS$, and the same developing map can be used. As $\tbS \to \bS$ is normal, each deck transformation $\delta \in \Delta(\bU, \bS)$ descends to a deck transformation $\delta' \in \Delta(\tbS, \bS)$. 
Comparing the developing map versions of the definition of holonomy in \eqref{eq:holonomy} and of the derivative in \eqref{eq:derivative}, we see that 
$D(\delta')=\holhom(\delta)H \in G/H$. Since $\bS$ is genuine, the restriction of $D$ to $\Delta(\tbS, \bS)$ is surjective. By definition of $\Gamma_\bS$, we have $\Gamma_\bS$ is the kernel of $D$, but covering space theory also tells us that $\Delta(\tbS, \bS)$ is $\Delta_\bS/\Gamma_\bS$ so the restriction of $D$ to $\Delta(\tbS, \bS)$ is an isomorphism as claimed.
\end{proof}

\begin{thm}
\label{minimal cover}
Let $H$ be a normal subgroup of $G \subset \Isom_+(\RR^2)$. Suppose $\bS$ and $\bT$ are Euclidean cone surfaces with structure group $G$ and $H$, respectively. Suppose also that $\bS$ is genuine with respect to $H$. Let $\pi_\bS:\tbS \to \bS$ be a minimal $H$-covering. Then if $f:\bT \to \bS$ is a $G$-mapping, then there is a unique $H$-mapping $\tilde f:\bT \to \tbS$ such that $f=\pi_\bS \circ \tilde f$. In particular, any two minimal $H$-coverings of $\bS$ are $H$-isomorphic.
\end{thm}
\begin{proof}
By standard covering space theory, there is a lift of $f$ to a map $\tilde f_1: \bT \to \tbS$. Because it is a lift, $\tilde f_1$ is a $G$-mapping. By \Cref{deck group and derivative}, we can find a deck transformation $\tilde \delta \in \Delta(\tbS, \bS)$ such that $D(\tilde f_1)=D(\tilde \delta)$. Then $\tilde f = \tilde \delta^{-1} \circ \tilde f_1$ is an $H$-mapping as desired. Uniqueness follows from the fact that there is exactly one deck transformation $\tilde \delta$ satisfying $D(\tilde f_1)=D(\tilde \delta)$ by \Cref{deck group and derivative}.

If we have two minimal $H$-coverings $\tbS_1$ and $\tbS_2$, then we get unique $H$-mappings $f=\tbS_1 \to \tbS_2$ and $g=\tbS_2 \to \tbS_1$. The composition $g \circ f$ gives an $H$-mapping $\tbS_1\to \tbS_1$ that descends to the identity, and therefore $g \circ f$ is the identity by uniqueness. Since $f \circ g$ is also the identity, both $f$ and $g$ are $H$-isomorphisms.
\end{proof}

\begin{remark}
The hypotheses of normality and genuine are necessary in \Cref{minimal cover}, despite their under emphasis in the statement of the closest result in \cite{GJ00}, their Theorem 3.6. Without these hypotheses, the two results above break. As an example relevant to this paper, the cube $\bC$ can be considered to be a Euclidean cone surface with structure group $G=\Isom_+(\ZZ^2)$. If $H=\Isom_+\big((2\ZZ)^2\big)$, then the minimal $H$-cover (as defined above) is a torus $\tilde \bC$ obtained as the double branched cover branched over four vertices. (The four vertices are colored the same in a graph-theoretic $2$-coloring of vertices. As there are two choices, the vertices branched over depend on the choice of the developing map.) The Necker cube surface $\bN$ is the universal cover of $\tilde \bC$.
\end{remark}

The following is an important special case: \compat{The tildes in $\tilde \pi$ seem superfluous below.}

\begin{cor}
\label{commute-branch-cover}
Suppose $H$ is a normal subgroup of $G \subset \Isom_+(\RR^2)$. Suppose $\bS$ and $\bT$ are 
Euclidean cone surfaces with structure group $G$ and that both are genuine with respect to $H$. Let $\tilde \pi_\bS:\tbS \to \bS$ and $\tilde \pi_\bT:\tbT \to \bT$ be their respective minimal $H$-coverings. Then if $f:\bS \to \bT$ is a $G$-covering map, there is a unique $H$-mapping $\tilde f:\tbS \to \tbT$ such that the following diagram commutes:
\begin{equation}
\label{eq:commutative}
\begin{tikzcd}
	\tbS \arrow[r,"\tilde f",dashed] \arrow[d,"\tilde \pi_\bS"]
		& \tbT \arrow[d,"\tilde \pi_\bT"] \\
	 \bS \arrow[r,"f"]
	& \bT.
\end{tikzcd}
\end{equation}
Furthermore, the map $\tilde f$ is an $H$-covering map. If in addition $H$ is finite index in $G$, then for all singular $\tilde x \in \tbS$, we have
$$\localdegree_f\big(\tilde \pi_\bS(\tilde x)\big)=1 \quad \text{implies} \quad \localdegree_{\tilde f}(\tilde x)=1.$$
\end{cor}
\begin{proof}
We apply \Cref{minimal cover} to the composition $f \circ \tilde \pi_\bS:\tbS \to \bT$. It gives a unique $H$-mapping $\tilde f:\tbS \to \tbT$ such that $f \circ \tilde \pi_\bS = \tilde \pi_\bT \circ \tilde f$ as desired. It follows from this identity and the definition of covering map that $\tilde f$ is a covering map because both $f \circ \tilde \pi_\bS$ and $\tilde \pi_\bT$ are. Thus $\tilde f$ is an $H$-covering map.

Now suppose that $H$ is finite index in $G$.
Let $x \in \bS$ be such that $\localdegree_f(x)=1$ and let $\tilde x \in \tilde \pi_\bS^{-1}(x)$. 
We will show $\localdegree_{\tilde f}(\tilde x)=1$. Let $y=f(x)$ and $\tilde y=\tilde f(\tilde x)$. 
Observe that from 
\eqref{eq:commutative}, we know
\begin{equation}
\label{eq:commutative local degrees}
\localdegree_{f \circ \tilde \pi_\bS}(\tilde x)=\localdegree_{\tilde \pi_\bS}\big(\tilde x) \cdot \localdegree_f\big(x)=\localdegree_{\tilde f}\big(\tilde x) \cdot \localdegree_{\tilde \pi_\bT}\big(\tilde y).
\end{equation}
Thus if $\localdegree_{\tilde \pi_\bS}\big(\tilde x)=1$, then all the other local degrees in the equation above must also be one.
Assume to the contrary that $\localdegree_{\tilde \pi_\bS}\big(\tilde x)=k>1$.
Then we can find a simply connected neighborhood $\tilde U$ of $\tilde x$ whose only singular point is $\tilde x$
such that the restriction of $\tilde \pi_\bS$ to $\tilde U$ is a branched cover of a simply connected neighborhood $U$ of $x$.
Let $\gamma:\RR/\ZZ \to U$ be a simple loop enclosing $x$.
Recall that $\holhom_\circ(\gamma)$ is a conjugacy class of $G$, say the conjugacy class of $g_1 \in G$.
By definition of the minimal $H$-cover, $k$ is the smallest positive integer such that $g_1^k \in H$. 
Then by \eqref{eq:holonomy and maps}, we see that $\holhom_\circ(f \circ \gamma)$ and $\holhom_\circ(\gamma)$ are equal as conjugacy classes.
Since $\tilde U$ has only one singular point $\tilde x$, the only possible singular point in $\tilde f(\tilde U)$ is $\tilde y$.
Since $f \circ \gamma$ is a simple curve in $f(U)$ enclosing $\tilde y$, by the definition of minimal $H$-cover we must also have
$\localdegree_{\tilde \pi_\bT}(\tilde y)=k$. Therefore by \eqref{eq:commutative local degrees} we can conclude that
$\localdegree_{\tilde f}\big(\tilde x)=1$.
\end{proof}

The following explains that in the situation above, with the additional assumption that the covering $\bS \to \bT$ is normal (in the sense that the deck group acts transitively on fibers of nonsingular points), the deck groups of the coverings $\bS \to \bT$ and $\tbS \to \tbT$ are isomorphic.
\begin{lemma}
	\label{deck groups}
Suppose $H$ is a normal subgroup of $G \subset \Isom_+(\RR^2)$. Suppose $\bS$ and $\bT$ are Euclidean cone surfaces with structure group $G$
and that $f:\bS \to \bT$ is a normal $G$-covering map. Suppose further that $\bS$ and $\bT$ are genuine with respect to $H$ and let $\tilde \pi_\bS:\tbS \to \bS$ and $\tilde \pi_\bT:\tbT \to \bT$ be their respective minimal $H$-coverings. Let $\tilde f:\tbS \to \tbT$ be the $H$-covering map from \Cref{commute-branch-cover}. Then $\tilde f$ is also a normal covering, and there is an isomorphism between the deck groups $\iota:\Delta(\bS,\bT) \to \Delta(\tbS,\tbT)$ such that for every $\delta \in \Delta(\bS,\bT)$, the following diagram commutes:
\begin{equation}
\label{eq:deck groups}
	\begin{tikzcd}
		\tbS \arrow[d,"\tilde \pi_\bS"] \arrow[r,"\iota(\delta)"] & \tbS \arrow[d,"\tilde \pi_\bS"] \\
		\bS \arrow[r,"\delta"]
		& \bS.
	\end{tikzcd}
\end{equation}
\end{lemma}
\begin{proof}
Let $\bU_\bS$ and $\bU_\bT$ be the universal branched covers of $\bT$ and $\bS$. Let $\Delta_\bS$ and $\Delta_\bT$ denote the respective deck groups.
Fix developing maps $\Dev_\bS$ and $\Dev_\bT$.
Let $\holhom_\bT:\Delta_\bT \to G$ be the holonomy homomorphism corresponding to $\Dev_\bT$.
Let $\holhom'_\bT$ denote the composition of $\holhom_\bT$ with the natural homomorphism $G \to G/H$. Similarly define $\holhom'_\bS$.
By definition of the minimal $H$-cover, the deck groups of the coverings $\bU_\bT \to \tbT$
and $\bU_\bS \to \tbS$ are given by
\begin{equation}
\label{eq:deck U}
\Delta(\bU_\bT, \tbT)=\ker \holhom'_\bT \subset \Delta_\bT
\quad \text{and} \quad
\Delta(\bU_\bS, \tbS)=\ker \holhom'_\bS \subset \Delta_\bS.
\end{equation}

The $G$-covering $f:\bS \to \bT$ lifts to a $G$-isomorphism $F:\bU_\bS \to \bU_\bT$ between their universal branched covers. Since $\bS$ covers $\bT$, the map $F$ induces a homomorphism $F_\ast:\Delta_\bS \to \Delta_\bT$; see \Cref{sect:natural maps}. As in \eqref{eq:holonomy and maps}, we have 
$$\holhom_T \circ F_\ast(\delta) = g_0 \cdot \holhom_S(\delta) \cdot g_0^{-1}.$$
for all $\delta \in \Delta_\bS$. It follows then that $F_\ast$ carries $\Delta(\bU_\bS, \tbS)$ into $\Delta(\bU_\bT, \tbT)$, giving us the following commutative diagram:
\begin{equation}
\label{eq:snake diagram}
	\begin{tikzcd}
	\Delta(\bU_\bS, \tbS) \arrow[r,hookrightarrow] \arrow[d,"F_\ast"] & 
	\Delta_\bS \arrow[d,"F_\ast"] \arrow[r,two heads,"\holhom_\bS'"] &
	G/H \arrow[d,"c_0"]
	\\
	\Delta(\bU_\bT, \tbT) \arrow[r,,two heads,hookrightarrow] & 
	\Delta_\bT \arrow[r,two heads,"\holhom_\bT'"] &
	G/H 
	\end{tikzcd}
\end{equation}
where $c_0$ denotes the action of conjugation by $g_0$ on $G/H$.

The horizontal rows of \eqref{eq:snake diagram} are exact, by \eqref{eq:deck U} and the assumption the surfaces $\bS$ and $\bT$ are genuine with respect to $H$. We will see that the vertical arrows have the property that images are normal in the codomain. This is clearly true for $c_0$, which is an automorphism of $G/H$. It is true for $F_\ast:\Delta_\bS \to \Delta_\bT$ because $\bS$ is a normal cover of $\bT$. 
To prove the same for the leftmost arrow, observe by commutativity that for any $\delta_\bS \in \Delta_\bS$, the conditions
 $\delta_\bS \in \ker \holhom'_\bS$ and
the condition that $F_\ast(\delta_\bS) \in \ker \holhom'_\bT$
are equivalent. Therefore, $F_\ast\big(\Delta(\bU_\bS, \tbS)\big)=F_\ast(\Delta_\bS \cap \ker \holhom'_\bS)$.
Thus,
$$F_\ast\big(\Delta(\bU_\bS, \tbS)\big)=F_\ast(\Delta_\bS) \cap \ker \holhom'_\bT=
F_\ast(\Delta_\bS) \cap \Delta(\bU_\bT, \tbT),
$$
which is normal in $\Delta(\bU_\bT, \tbT)$ because $F_\ast(\Delta_\bS)$ is normal in $\Delta_\bT$. By an application of the Snake Lemma in the category of non-abelian groups (see for instance \cite{DGJL}), since $c_0$ is an isomorphism we get the exact sequence
$$1 \to \Delta(\bU_\bT, \tbT)/F_\ast\big(\Delta(\bU_\bS, \tbS)\big) \xrightarrow{\phi} \Delta_\bT/F_\ast(\Delta_\bS) \to 1.$$
Therefore, $\phi$ is an isomorphism of groups. Further $\phi$ is the map induced by the inclusions, i.e., for $\delta_\bT \in \Delta(\bU_\bT, \tbT)$, 
$$\phi:\delta_\bT F_\ast\big(\Delta(\bU_\bS, \tbS)\big) \mapsto 
\delta_\bT F_\ast(\Delta_\bS),$$
which is well-defined since $\Delta(\bU_\bS, \tbS) \subset \Delta_\bS$.
\notepat{I learned about the Snake Lemma in the category of groups from Wikipedia: {\url{https://en.wikipedia.org/wiki/Snake_lemma\#In_the_category_of_groups}}, but there was no reference!}

Now note that the deck group of $f:\bS \to \bT$ is given by $\Delta(\bS,\bT)=\Delta_\bT/F_\ast(\Delta_\bS)$, where the deck group action is induced by the action of $\delta_\bT \in \Delta_\bT$ on $S$ given by 
\begin{equation}
\label{eq:deck action}
s \mapsto \tilde \pi_\bS \circ F^{-1} \circ \delta_\bT \circ F(\tilde s)
\end{equation}
where $\tilde s \in \bU_\bS$ is a lift of $s \in S$. Similarly, we have
$$\Delta(\tbS,\tbT)=\Delta(\bU_\bT, \tbT)/F_\ast\big(\Delta(\bU_\bS, \tbS)\big),$$
and the deck group action on $\tbS$ is induced by the same action \eqref{eq:deck action} but with $\delta_\bT \in \Delta(\bU_\bT, \tbT) \subset \Delta_\bT$, the lift taken from $\tbS$ to $\bU_\bS$, and the projection given by the covering $\bU_\bS \to \tbS$. Thus, $\iota=\phi^{-1}$ is the desired identification between the deck groups and 
\eqref{eq:deck groups} commutes because a $\delta_\bT \in \Delta(\bU_\bT, \tbT)$ acts on $\bS$ and $\tbS$ in compatible ways via \eqref{eq:deck action} and the corresponding equation for $\Delta(\tbS,\tbT)$.
\end{proof}

\subsection{Group actions}

Let $H \subset \Isom_+(\RR^2)$ and let $G \subset \Homeo_+(\RR^2)$ be a subgroup such that $H$ is a normal subgroup of $G$. Let $\bS$ be a Euclidean cone surface with structure group $H$. Let $g \in G$. We can form a new Euclidean cone surface with structure group $G$ by postcomposing each chart of $\bS$ with the affine map $g:\RR^2 \to \RR^2$. Denote the resulting surface $g(\bS)$, which has the same underlying topological surface but a new atlas. 

The identity map gives a map $\varphi_g: \bS \to g(\bS)$ with derivative $gH$. This map is a $G$-isomorphism, but typically $\bS$ and $g(\bS)$ will not be $H$-isomorphic. However if $h \in H$, then $\varphi_h:\bS \to h(\bS)$ is an $H$-isomorphism. Thus, the $G$-action on Euclidean cone surfaces with structure group $H$ (up to isomorphism) descends to an action of the quotient group $G/H$.

The subgroup $V(\bS)=D\big(\Aut_G(\bS)\big) \subset G/H$ is known as the Veech group of the surface. We have $g H \in V(\bS)$ if and only if $\bS$ and $g(\bS)$ are $H$-isomorphic. In other words, the $H$-isomorphism classes in the $G$-orbit of $\bS$ are parameterized by the cosets in $(G/H)/V(\bS)$.

In this paper $H$ will be either the translation group $G_1$ or the half-translation group $G_2$, and $G=\Aff_{+1}(\RR^2)$, the group of orientation-preserving and area-preserving affine maps of the plane. Thus $G/H$ is either 
$$\Aff_{+1}(\RR^2)/G_1\cong \SL(2,\RR) \quad \text{or} \quad \Aff_{+1}(\RR^2)/G_2\cong \PSL(2,\RR).$$

\subsection{Paths}
Let $J:\Aff_+(\RR^2) \to \GL(2,\RR)$ be the usual Jacobian derivative from multivariable calculus. This is a group homomorphism whose kernel is $G_1$, the group of translations. 

Let $\bS$ be a Euclidean cone surface with structure group $G \subset \Isom_+(\RR^2)$ and let $\Dev_\bS:\bU \to \RR^2$ be a choice of developing map. 
The {\em displacement} of a path $\gamma:[a,b] \to \bS$ is the element of the quotient space $J(G) \bs \RR^2$ obtained from a lift $\tilde \gamma:[a,b] \to \bU$ by the formula 
$$\disp(\gamma) = J(G) \cdot [\Dev_\bS \circ \tilde \gamma(b) - \Dev_\bS \circ \tilde \gamma(a)].$$
The value $\disp(\gamma)$ is independent of the choices of $\Dev_\bS$ and $\tilde \gamma$. The displacement of $\gamma$ is also invariant under homotopies within $\bSs$ relative to the endpoints.

On a translation surface, $J(G)=J(G_1)$ is trivial, so displacement is a vector. For a half-translation surface, $J(G)=J(G_2) = \{\pm I\}$ so displacement takes values in $\{\pm I\} \bs \RR^2$. For a quarter-translation surface, $J(G)=J(G_4)=C_4$, so displacement takes values in the quotient of $\RR^2$ by $C_4$.

\begin{prop}
\label{affinemap-derivative}
Let $G$ be a normal subgroup of $\Aff_+(\RR^2)$ (e.g., $G \in \{G_1, G_2\}$).
Suppose both $\bS$ and $\bT$ are Euclidean cone surfaces with structure group $G$, and let $\gamma:[a,b] \to \bS$ be a path. 
Then for any $\Aff_+(\RR^2)$-mapping $\phi:\bS \to \bT$ with derivative $D \phi \in \Aff_+(\RR^2)/G$ we have
$\disp(\phi \circ \gamma) = J(D \phi) \cdot \disp(\gamma)$.
\end{prop}
\begin{proof}
Let $\tilde \gamma:[a,b] \to \bU_\bS$ be a lift of $\gamma$, and $\tilde \phi:\bU_\bS \to \bU_\bT$ be a lift of $\phi$.
Then $\tilde \phi \circ \tilde \gamma$ is a lift of $\phi \circ \gamma$. Let $\alpha \in \Aff_+(\RR^2)$ be such that
$\Dev_\bT \circ \tilde \phi = \alpha \circ \Dev_\bS$. We have $D(\phi)=\alpha G \in \Aff_+(\RR^2)/G$; see \eqref{eq:derivative}.
Then we have
$$\Dev_\bT \circ \tilde \phi \circ \tilde \gamma(b) - \Dev_\bS \circ \tilde \phi \circ \tilde \gamma(a)=
J(\alpha)[\Dev_\bS \circ \tilde \gamma(b) - \Dev_\bS \circ \tilde \gamma(a)].$$
Thus $\disp(\phi \circ \gamma) = J(G) \cdot J(\alpha)[\Dev_\bS \circ \tilde \gamma(b) - \Dev_\bS \circ \tilde \gamma(a)].$
Since $J$ is a group homomorphism and $\alpha$ normalizes $G$, we see that $J(G) \cdot J(\alpha)=J(\alpha) \cdot J(G)$. Thus,
$\disp(\phi \circ \gamma) = J(D \phi) \cdot \disp(\gamma)$ as claimed.
\end{proof}

\subsection{Geodesics}
\label{sect:geodesics}
A {\em geodesic} in a cone surface $\bS$ is a path $\gamma:I \to \bS$, where $I \subset \RR$ is an interval with interior $I^\circ$, such that 
$\gamma|_{I^\circ}$ is a nonconstant affine map in local coordinates. Letting $\tilde \gamma:I \to \tilde \bU$ be a lift to the universal branched cover, this locally affine condition is equivalent to the composition $\Dev \circ \tilde \gamma:I \to \RR^2$ being a nonconstant affine map.

Assuming $I$ is a closed and bounded interval, the {\em length} of $\gamma$ is the same as the length of $\disp(\gamma)$, denoted $\|\disp(\gamma)\|$. 
Let $\SS^1$ be the collection of equivalence classes of nonzero vectors in $\RR^2$ up to the action of positive scalars. (Each equivalence class then has a unique representative on the unit circle.)
The {\em direction} $\dir(\gamma)$ of a geodesic $\gamma$ is the image of $\disp(\gamma)$ in $\SS^1$. That is, $\dir(\gamma) \in \SS^1/J(G)$. This slightly generalizes the notion of direction introduced in the introduction where in a square-tiled surface, where we have $G=\Isom_+(\ZZ^2)$ and $J(G)=C_4$, the group consisting of rotations of the unit circle by multiples of $90^\circ$.

By \eqref{eq:qts mapping}, we see that if $\phi:\bS \to \bT$ is an $\Aff_+(\RR^2)$-mapping between cone surfaces, and $\gamma:I \to \bS$ is a geodesic, then $\phi \circ \gamma$ is also a geodesic. \Cref{affinemap-derivative} tells us how the displacement $\disp(\phi \circ \gamma)$ relates to $\disp(\gamma)$. The group $\GL(2,\RR)$ acts on $\SS^1$, and we have
$$\dir(\phi \circ \gamma)=J\big(D(\phi)\big) \cdot \dir(\gamma)$$
when the hypotheses of \Cref{affinemap-derivative} are satisfied. It is similarly easy to figure out the length of $\phi \circ \gamma$
is $c$ times the length of $\gamma$, where $c$ is the length of the image of a unit vector in the direction of $\gamma$ under the action of $D(\phi) \in \GL(2,\RR)$.

A {\em saddle connection} is a geodesic $\gamma:[a,b] \to \bS$ such that $\gamma(a)$ and $\gamma(b)$ are singular points.

A {\em bi-infinite geodesic} is a geodesic $\gamma:\RR \to \bSs$. 
A bi-infinite geodesic is {\em periodic} if there is a $p>0$ such that $\gamma(t+p)=\gamma(t)$ for all $t \in \RR$.  A {\em period} for the geodesic is an interval of the form $[a,a+p]$. A {\em closed geodesic} is the restriction of a periodic geodesic to a period. The {\em length} of a periodic geodesic is the same as the length of a corresponding closed geodesic built with minimal period.

A bi-infinite geodesic is {\em drift-periodic} if it is not periodic and there is a $G$-automorphism $\phi:\bS \to \bS$ and a $p>0$ such that $\gamma(t+p)=\phi \circ \gamma(t)$ for all $t \in \RR$.
A {\em period} is an interval $[a,a+p]$ as before.

\subsection{Periodic geodesics on square-tiled surfaces} 
Let $\bS$ be a square-tiled surface. Then $\bS$ can be considered to have structure group $\Isom_+(\ZZ^2)$. And we use this in our definition of drift-periodicity. Observe that if $\gamma:\RR \to \bSs$ is a periodic or drift-periodic geodesic, and $\tilde \gamma:I \to \bU$ is a lift then there is a $G$-automorphism $\tilde \phi: \bU \to \bU$ and a $p$ such that $\tilde \gamma(t+p)=\tilde \phi \circ \tilde \gamma(t)$ for all $t \in \RR$. Then for any developing map, there is a $g \in \Isom_+(\ZZ^2)$ such that $\Dev \circ \tilde \gamma(t+p)=\Dev \circ \tilde \gamma(t)$ for all $t$. Since $g$ translates along the line $\Dev \circ \tilde \gamma(\RR)$, we know $g$ is a translation by some $(m,n) \in \ZZ$. We conclude that:

\begin{prop}
\label{periodic implies rational}
Suppose $\gamma:\RR \to \bSs$ is a unit speed periodic or drift-periodic geodesic. Then there is a nonzero $(m,n) \in \ZZ^2$ such that for any period $[a,p]$, we have 
$\disp(\gamma|_{[a,a+p]})=C_4 \cdot (m,n).$
In particular, the length of a period of $\gamma$ is $\sqrt{m^2+n^2}$ and $\dir(\gamma)$ has rational slope (concretely, 
$\dir(\gamma)=\{\frac{\pm(m,n)}{\|(m,n)\|}, \frac{\pm(-n,m)}{\|(m,n)\|}\} \in \SS^1/C_4$).
\end{prop}

The following is a partial converse to the result above. If $I \subset \RR$ is an interval and $\gamma:I \to \bS$ is a geodesic, then we call $\gamma$ {\em maximal} if there is no geodesic extension of $\gamma$ to a larger interval. On a square-tiled surface, a maximal geodesic is either bi-infinite or intersects a singularity.

\begin{prop}
\label{rational implies}
Let $\bS$ be a square-tiled surface and $\gamma:I \to \bSs$ be a geodesic of rational slope. 
\begin{enumerate}
\item If $\bS$ is built from finitely many squares, then $\gamma$ can be extended to either a periodic geodesic or a saddle connection. 
\item If there is a group $H$ of $\Isom_+(\ZZ^2)$-mappings $\bS \to \bS$ that send squares to squares such that there are only finitely orbits of squares under $H$, then $\gamma$ can be extended to a periodic geodesic, a drift-periodic geodesic, or a saddle connection.
\end{enumerate}
\end{prop}
\begin{proof}
Statement (1) in the case of $\bS=\TT^2$ the square torus is a standard fact (see for instance \cite[Chapter 2]{Tabachnikov}). Consider the quotient ${\mathbf Q}$ obtained by $\TT^2$ modulo a $90^\circ$ rotation fixing the center of the square as depicted in \Cref{fig:quotient}. The surface $Q$ is the $\Isom_+(\ZZ^2)$ surface $\RR^2/\Isom_+(\ZZ^2)$. It has three singularities: one cone singularity with cone angle $\pi$ (corresponding to the identified midpoints of edges of the square) and two cone singularities with cone angle $\frac{\pi}{2}$ (corresponding to the center of the square and the identified vertices of the square). Since it is a quotient quarter-translation surface, $\mathbf Q$ also satisfies (1). Observe that every square-tiled surface built from finitely many squares is a finite branched cover of $\mathbf Q$ via the map that sends a square to the square making up $\TT^2$ followed by the quotient map $\TT^2 \to \mathbf Q$ as above. Since statement (1) holds for $\mathbf Q$, it also holds for every finite branched cover. This proves (1).

\begin{figure}[htb]
\centering
\includegraphics[width=5in]{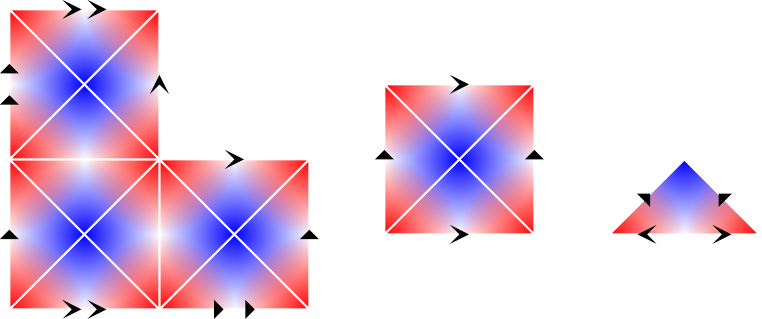}
\caption{From left to right: A square-tiled surface $\bH$, the square torus $\TT^2$, and its quotient $\mathbf Q$ from the proof of \Cref{rational implies}. Colors indicate the covering maps $\bH \to \mathbf Q$ and $\TT^2 \to \mathbf Q$.}
\label{fig:quotient}
\end{figure}

To see (2), observe that $\bS$ is still a branched cover of $\mathbf Q$. Since the group $H$ sends squares to squares, $\bS/H$ also covers $\mathbf Q$ and is a finite cover by hypothesis. Therefore $\bS/H$ satisfies (1). Given $\gamma$, let $\hat \gamma$ be the image in $\bS/H$. If $\hat \gamma$ extends to a saddle connection in $\bS/H$ then so must $\gamma$.
On the other hand if $\hat \gamma$ is a periodic geodesic on $\bS/H$ extends to a periodic geodesic, then we can extend $\gamma$ to one period of a lift and we can find an $h \in H$ such that this lift closes up modulo $h$. (That is, $h$ sends the starting point of the lift to the ending point and the starting tangent vector to the ending tangent vector.) The orbit of the lift under the subgroup $\langle h \rangle \subset H$ is the full extension of $\gamma$, so $\gamma$ extends to a periodic geodesic if $h$ is finite order and extends to a drift-periodic geodesic if $h$ is infinite order.
\end{proof}

\subsection{Algebraic intersection number} Let $\bS$ be a closed, oriented surface. We denote its integral homology by $H_1(\bS;\ZZ)$. \emph{Algebraic intersection} is a nondegenerate, symmetric bilinear form on $H_1(\bS;\ZZ):$
$$i:H_1(\bS;\ZZ)\times H_1(\bS;\ZZ)\rightarrow \ZZ$$
defined as the sum of signed intersections ($\pm 1$) over all intersection points of the two homology classes relative to the fixed orientation. Because of this definition, for any orientation-preserving homeomorphism $h:\bS \to \bS$, we have
$$i\big(h(\alpha), h(\beta)\big)= i(\alpha, \beta) \quad \text{for all $\alpha, \beta \in H_1(\bS;\ZZ)$.}$$

When $\bS$ has genus $g$, an ordered collection of classes in $H_1(\bS; \ZZ)$, $(\alpha_1, \ldots, \alpha_g; \beta_1, \ldots, \beta_g)$ is said to be a {\em standard symplectic basis} for $H_1(\bS; \QQ)$ if $i(\alpha_i, \alpha_j)=i(\beta_i, \beta_j)=0$ and $i(\alpha_i, \beta_j)=\delta_{ij}$ for all $i$ and $j$. It is a standard observation that any collection satisfying these conditions generates $H_1(\bS; \ZZ)$.

\subsection{Cylinders, cylinder decompositions, and multi-twists}
\label{sect:cylinders}
A {\em Euclidean cylinder} is a space $C_{I,\ell}=I\times \RR/\ell\ZZ$ where $I \subset \RR$ is an interval, with its natural local Euclidean structure. We call $\ell>0$ the {\em circumference}, and the length of $I$ is the {\em width} of the cylinder. Such a cylinder is foliated by the closed geodesics of the form $x \mapsto (x, t +\ell\ZZ)$ where $x$ is fixed and $t \in \RR/\ell \ZZ$ is the parameter. A {\em core curve} is a loop that is homotopic to such a geodesic. Assuming $I=[0,w]$, {\em left Dehn twist} of $C_{w,\ell}$ is the homeomorphism
$$\Dehn:C_{I,\ell} \to C_{I,\ell}; \quad (x,y+\ell\ZZ) \mapsto (x,\tfrac{\ell}{w}x + y + \ell \ZZ).$$
This map fixes all points in the boundary $\partial C_{I,\ell}$ and is an $\Aff_{+1}(\RR^2)$-isomorphism of the interior $C_{I,\ell}$ with its natural translation structure. The derivative of $\Dehn$ is determined by the {\em inverse modulus} 
$\ell/w$: $$D(\Dehn) = \begin{pmatrix}
1 & 0 \\
\frac{\ell}{w} & 1 \\
\end{pmatrix} \in \SL(2,\RR).$$
Furthermore, this map is normalized by the orientation-preserving isometry group of $C_{w,\ell}$. The {\em right Dehn twist} of $C_{w,\ell}$ is $\Dehn^{-1}$.

A \emph{cylinder} in a Euclidean cone surface $\bS$ is a continuous $\iota:C_{I,\ell} \to \bS$ whose restriction to the interior $C_{I, \ell}$ is an orientation-preserving isometric immersion (i.e., an $\Isom_+(\RR^2)$-mapping). By parallel transport, one can extend a closed geodesic to a cylinder:

\begin{prop}
Let $\bS$ be a Euclidean cone surface and $\gamma:\RR/\ZZ \to \bSs$ be a closed geodesic. Then, there is a cylinder $\iota:C_{I,\ell} \to \bS$ and an embedding $\epsilon:\RR/\ZZ \to C_{I,\ell}$ such that $\gamma=\iota \circ \epsilon$.
\end{prop}

We say a cylinder $f_1:C_1 \to \bS$ is {\em contained} in a cylinder $f_2:C_2 \to \bS$ if there is an isometric immersion $g:C_1 \to C_2$ such that $f_1=f_2 \circ g$. (Note that $g$ might be a covering map, so this does not completely coincide with the usual notion of containment.) Two cylinders should be considered the same if they are contained in each other. A cylinder $f:C \to \bS$ is {\em maximal} if whenever it is contained in another cylinder, that cylinder is the same.

If $\bS$ is a half-translation surface or a translation surface, then each cylinder has a well-defined {\em slope} given by the directions of the geodesic core curves. Thus, cylinders cannot self-intersect (because intersecting parallel geodesics coincide). In other words, every cylinder covers factors through an injective cylinder map. Taken up to precomposition with an isometry $C_{w,\ell} \to C_{w,\ell}$, injective cylinder maps $\iota:C_{w,\ell} \to \bSs$ can be identified with their images $\iota(C_{w,\ell})$ and so it is natural to think of a cylinder as an open subset of $\bSs$. Since the left Dehn twist of a cylinder fixes all points on its boundary, it can be extended to a homeomorphism of the whole surface $\bS$, acting trivially on the complement. The action of the left twist on homology is given by 
\begin{equation}
\label{eq:Dehn action}
\Dehn_\ast:\alpha \mapsto \alpha - i(\gamma, \alpha) \gamma,
\end{equation}
where $\gamma$ represents the homology class of a core curve (with either orientation). The formula for the action of the right twist is the same, but with $-$ changed to $+$.

A {\em cylinder decomposition} of a half-translation or translation surface $\bS$ is a collection of pairwise disjoint, parallel open cylinders such that the union of the closures of the cylinders is the whole surface. If all the cylinders in a decomposition have the same inverse modulus $m$, then there is an $f \in \Aff_{+1}(\bS)$ such that:
\begin{enumerate}
\item For every cylinder $\iota:C_{w, \ell} \to \bSs$ in the decomposition, we have $f \circ \iota = \Dehn_{w,\ell} \circ \iota$.
\item Every point of $\bS$ not contained in a cylinder is fixed by $f$.
\end{enumerate}
We call $f$ the {\em left multi-twist} in the cylinder decomposition. The derivative of $f$ is 
$$R \cdot \begin{pmatrix}
1 & 0 \\
m & 1 \\
\end{pmatrix} \cdot R^{-1}.$$
where $R$ is a rotation sending vertical lines to lines whose slope is the same as the cylinders in the decomposition. The {\em right multi-twist} is $f^{-1}$. These ideas go back to Thurston \cite{T88}. For more background see \cite{HubertSchmidt}.

\section{The Necker cube surface and its isometries}
\label{sect:isom}
Let $\bone=(1,1,1) \in \RR^3$. For $i \in \{-1,0,1\}$, let $\Lambda_i=\{\bv \in \ZZ^3:~\bv \cdot \bone = i\}$, where $\cdot$ denotes the usual dot product. The {\em Necker cube surface} $\bN$ is the union of all unit squares whose vertices lie in $\Lambda_{-1} \cup \Lambda_0 \cup \Lambda_1$. We'll say such a square is a {\em constituent square}.

Let $S_3 \subset O(3)$ denote the group of $3 \times 3$ permutation matrices, and let $\pm S_3=S_3 \cup \{-P:~P \in S_3\}$. 

From our definition of $\bN$ it follows that any isometry $f:\RR^3 \to \RR^3$ that preserves $\ZZ^3$ and the plane $\bone^\perp$ (consisting of vectors perpendicular to $\bone$), and whose derivative satisfies $Df(\bone)=\pm \bone$ stabilizes $\bN$.  From this it can be seen that any $f$ in the following group of isometries of $\RR^3$ must preserve $\bN$:
\begin{equation}
\label{eq:isometry}
\Isom(\RR^3, \bN) = \big\{\bx \mapsto P \bx + \bv
\text{ : $P \in \pm S_3$ and $\bv \in \Lambda_0$.}\big\}.
\end{equation}
The notation suggests that $\Isom(\RR^3, \bN)$ is the group of isometries of $\RR^3$ that preserve $\bN$; This is the case but we don't prove it until \Cref{isometries}.
Given $f \in \Isom(\RR^3, \bN)$, the matrix $P$ can be recovered by taking the derivative; $Df=P$.

Observe that $S_0 = [-1,0] \times \{0\} \times [0,1]$ is a constituent square. Moreover:

\begin{prop}
\label{isometries and squares}
The set of constituent squares is $\big\{f(S_0):~f \in \Isom(\RR^3, \bN)\big\}.$
\end{prop}

To prove this, we need to understand the squares making up $\bN$. For $k \in \{0,1,2\}$, let $\be_k$ denote the corresponding standard basis vector in $\RR^3$.

\begin{lemma}
\label{lem:squares}
A unit square $S$ is a constituent square if and only if there is a $\bv_0 \in \Lambda_{-1}$ and a choice of distinct $k,\ell \in \{1,2,3\}$ such that the vertices of $S$ in some cyclic order are given by $\bv_0$, $\bv_1=\bv_0+\be_k \in \Lambda_0$, $\bv_2=\bv_0+\be_k+\be_\ell \in \Lambda_1$, and $\bv_3=\bv_0+\be_\ell \in \Lambda_0$.
\end{lemma}
\begin{proof}
Certainly given any $\bv_0 \in \Lambda_{-1}$ and a choice of $k<\ell$, the vertices listed determine a constituent square. Consider the converse. Let $S$ be a constituent squares. Observe that for any $i \in \{-1, 0, 1\}$ the minimal distance between points in $\Lambda_i$ is $\sqrt{2}$. Also the minimal distance between a point in $\Lambda_{-1}$ and a point in $\Lambda_1$ is $\sqrt{2}$. It follows that for any two adjacent vertices of $S$, one must lie in $\Lambda_0$ and the other must lie in $\Lambda_{-1} \cup \Lambda_1$. It follows that two opposite vertices lie in $\Lambda_0$, and the other two opposite vertices lie in $\Lambda_{-1} \cup \Lambda_1$. It can't be that both opposite vertices lie in $\Lambda_1$, because then the diagonals could not cross (since the diagonals would then lie in parallel planes). The same argument works for $\Lambda_{-1}$, so one vertex, say $\bv_0$, lies in $\Lambda_{-1}$ and the opposite vertex $\bv_2$ lies in $\Lambda_1$. Since $S$ is a unit square and $\bv_1, \bv_3 \in \Lambda_0$, it must be that $\bv_1 = \bv_0+\be_k$ and $\bv_3=\bv+\be_\ell$ for some $k,\ell \in \{0,1,2\}$. Since $\bv_1 \neq \bv_3$, we have $k \neq \ell$. Since $S$ must be a square, we have $\bv_2=\bv_0+\be_k + \be_\ell$. 
\end{proof}

\begin{proof}[Proof of \Cref{isometries and squares}]
Above it was noted $S_0 \subset \bN$. By definition of $\Isom(\RR^3, \bN)$, the orbit of $S_0$ under $\Isom(\RR^3, \bN)$ is contained in $\bN$. Conversely, if $S$ is a constituent square, we know $S$ is as described in 
\Cref{lem:squares} in terms of vertices and $k$ and $\ell$. Observe that $S=f(S_0)$ where $f(\bx)=P\bx+\bv_1$ where $\bv_1$ is as in the lemma, and $P$ is the permutation matrix such that $P(\be_3)=\be_\ell$ and $P(\be_1)=\be_k$.
\end{proof}

As a second consequence of \Cref{lem:squares}, we see:

\begin{cor}
\label{cor:squares meeting}
If $\bv \in \Lambda_{-1} \cup \Lambda_1$, there are three constituent squares with vertex $\bv$. If $\bv \in \Lambda_0$, then there are six constituent squares with vertex $\bv$.
\end{cor}
\begin{proof}
For $\bv \in \Lambda_{-1}$, this follows from \Cref{lem:squares} because there are three pairs $\{k,\ell\}$ of elements of $\{0,1,2\}$. This follows for vertices in $\Lambda_1$ because there are elements of $\Isom(\RR^3, \bN)$ that swap $\Lambda_1$ and $\Lambda_{-1}$, such as $\bx \mapsto -\bx$. For a vertex $\bv \in \Lambda_0$, we see that squares sharing this vertex are in bijection with ordered pairs $(k, \ell)$ of distinct elements of $\{0,1,2\}$, because the other vertices have the form $\bv + \be_k$, $\bv-\be_\ell$ and $\bv+\be_k-\be_\ell$. There are six such ordered pairs.
\end{proof}

The isometric projection used in images throughout the paper is orthogonal projection $\proj:\RR^3 \to \bone^\perp$. Observe that any $f \in \Isom(\RR^3, \bN)$ commutes with $\proj$, i.e., 
\begin{equation}
\label{eq:proj commute}
f \circ \proj(\bx)=\proj \circ f(\bx) \quad \text{for all $\bx \in \RR^3$,}
\end{equation}
because $f$ is an isometry preserving $\bone^\perp$. Also:
\begin{prop}
\label{prop:rhombille tiling}
The collection of all $\proj(S)$ taken over constituent squares $S$ of $\bN$ is the rhombille tiling, as depicted in \Cref{fig:NCS}. The restriction $\proj|_{\bN}:\bN \to \bone^\perp$ is a homeomorphism.
\end{prop}
\begin{proof}
Observe that the image of our favorite square $S_0$ is the rhombus $\proj(S_0)$ with vertices $\proj(\0)=\0$,
$$
\proj\begin{pmatrix} -1 \\ 0 \\ 0 \end{pmatrix}=\begin{pmatrix}
\frac{-2}{3} \\ \frac{1}{3} \\ \frac{1}{3} \end{pmatrix}, \quad
\proj\begin{pmatrix} -1 \\ 0 \\ 1 \end{pmatrix} = \begin{pmatrix} -1 \\ 0 \\ 1 \end{pmatrix},
\quad \text{and} \quad
\proj\begin{pmatrix} 0 \\ 0 \\ 1 \end{pmatrix}=\begin{pmatrix}
\frac{-1}{3} \\ \frac{-1}{3} \\ \frac{2}{3} \end{pmatrix}.
$$
At $\0$ and the opposite vertex $(-1,0,1)$, the interior angle of $\proj(S_0)$ is $\frac{\pi}{3}$, while at the other vertices the interior angle is $\frac{2\pi}{3}$. Observe that if $P \in S_3$ is a permutation matrix of order three, then the map in $\Isom(\RR^3, \bN)$ given by $f(\bx)= -P \bx$ acts on $\bone^\perp$ as a rotation of order six fixing $\0$, and thus the six tiles $\proj \circ f^i(S_0)$ fit around the origin as six tiles meeting in the rhombille tiling. From \Cref{lem:squares}, we know that every constituent square has two opposite vertices in $\Lambda_0$. Choose two of the six rhombi that share an edge, and let $\bw_1, \bw_2 \in \Lambda_0$ be the vertices in these rhombi that lie opposite $\0$. Note that $\Lambda_0=\langle \bw_1, \bw_2\rangle$. As these vertices lie in $\Lambda_0$, translations by these vectors lie in $\Isom(\RR^3, \bN)$. Observe that the group of translations by elements of $\langle \bw_1, \bw_2\rangle$ applied to the six rhombi with vertex $\0$ produces the rhombille tiling.

To see the remaining statements hold, note that the restriction of $\proj$ to $\Lambda_{-1} \cup \Lambda_0 \cup \Lambda_1$ is injective. We claim that the tiles produced above give all images $\proj(S)$ of constituent squares. This is clear because each such square has a vertex in $\Lambda_0$. By construction, every element of $\Lambda_0$ has six tiles sharing it as a vertex, and there are only six squares with this vertex by \Cref{cor:squares meeting}. 

To see that $\proj|_\bN$ is a homeomorphism, observe that each rhombus in the tiling has a unique lift to $\bN$, because $\proj$ restricts to an injective map on $\Lambda_{-1} \cup \Lambda_0 \cup \Lambda_1$. It also follows that lifts of rhombi sharing an edge are squares sharing an edge. The union of these lifts defines a continuous inverse to $\proj|_\bN$.
\end{proof}

\begin{thm}
\label{isometries}
For each isometry $g$ of the plane $\bone^\perp$ that preserves the rhombille tiling of \Cref{prop:rhombille tiling}, there is a unique $f \in \Isom(\RR^3, \bN)$ such that $g=f|_{\bone^\perp}$. For each isometry $h:\bN \to \bN$, there is a unique $f \in \Isom(\RR^3, \bN)$ such that $h=f|_\bN$.
\end{thm}
\begin{proof}
To see the first statement, observe that restrictions to $\bone^\perp$ of maps of the form $\bx \mapsto P\bx$ taken over $P \in \pm S_3$ form the dihedral group of symmetries of the rhombille tiling fixing the origin. This finite group together with translations by elements of $\Lambda_0$ form the symmetry group of the rhombille tiling. Thus, every symmetry of the tiling is realized by restriction of an element of $\Isom(\RR^3, \bN)$. To see uniqueness, note that there are only two extensions of an isometry of $\bone^\perp$ to an isometry of $\RR^3$; they differ by reflection in $\bone^\perp$. Since $\Isom(\RR^3, \bN)$ is a group,
if an extension were not unique then reflection in $\bone^\perp$ would be in the group. But this reflection is $\bx \mapsto \bx - \frac{2}{3} (\bx \cdot \bone) \bone$ which does not preserve $\ZZ^3$ and so is not in $\Isom(\RR^3, \bN)$.

To see the second statement, let $h:\bN \to \bN$ be an isometry. By \Cref{cor:squares meeting}, $h(\Lambda_0)=\Lambda_0$ and $h(\Lambda_{-1} \cup \Lambda_1)=\Lambda_{-1} \cup \Lambda_1$. The only paths of length one in $\bN$ joining a point in $\Lambda_0$ to a point in $\Lambda_{-1} \cup \Lambda_1$ are the edges of squares. Thus, $h$ sends constituent squares to constituent squares. By \Cref{prop:rhombille tiling}, there is an induced map defined on $\bone^\perp$ given by $\bar h = \proj|_{\bN} \circ h \circ (\proj|_{\bN})^{-1}$ which sends rhombi to rhombi.
Noting that the restriction of $\proj$ to a constituent square is affine, we see that $\bar h$ is affine when restricted to rhombi in the tiling. The map $\bar h$ preserves $\Lambda_0$. Therefore the restriction of $\bar h$ to a rhombus $R$ must carry the vertices with interior angle $\frac{\pi}{3}$ to vertices of the image rhombus $\bar h(R)$ with the same interior angles. But observe that there is an isometry $R \to \bar h(R)$ that moves vertices in the same way. Thus, $\bar h$ acts by isometry on rhombi. But for two rhombi sharing an edge, the isometry applied to each must be the same, so by induction we see that $\bar h$ is an isometry of $\bone^\perp$ that preserves the rhombille tiling. We conclude that $\bar h$ is the restriction of an element $f \in \Isom(\RR^3, \bN)$ to $\bone^\perp$. Since $f$ commutes with $\proj$, we see that $h=f|_{\bN}$. This $f$ is unique, because $\bN$ has four non-coplanar points.
\end{proof}

\section{Canonical Half-Translation Covers of the Necker Cube Surface and the Half-Cube}

We will now introduce some surfaces which are naturally related to the Necker cube surface $\bN$. Our later arguments will involve some computations in homology, so we will also establish some notation for homology classes. \compat{UPDATE this paragraph.}

\subsection{The half-cube}
\label{sect:half-cube}
The {\em half-cube surface} $\bH$ is the square-tiled surface shown in \Cref{fig:halfcube}. 
It is the quotient $\bN$ modulo the subgroup of translations in $\Isom(\RR^3, \bN)$.
It is built out of three unit squares. Opposite sides of the resulting 3-dimensional hexagon are glued together by translation. 

\begin{figure}[htb]
\labellist
\small\hair 2pt
 \pinlabel {$\beta_0^\ast$} [ ] at 104 105
 \pinlabel {$\beta_1^\ast$} [ ] at 5 75
 \pinlabel {$\beta_2^\ast$} [ ] at 71 4
\endlabellist
\centering
\includegraphics[width=1.5in]{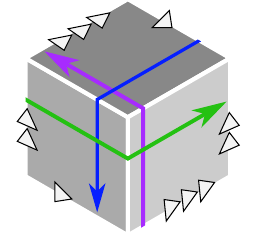}
\caption{The half-cube surface $\bH$ is made of three squares. Edges are identified by translation. Three curves are depicted.}
\label{fig:halfcube}
\end{figure}

\notepat{I'm not a fan of the notation $\beta^\ast_j$, because we don't use anything similar elsewhere. But, I left it because it seems like a pain to change it...}
Topologically, $\bH$ is a torus. Three oriented curves $\beta_0^\ast$, $\beta_1^\ast$, and $\beta_2^\ast$ are depicted in \Cref{fig:halfcube}. Any two of their homology classes generate $H_1(\bH;\ZZ) \cong \ZZ^2$.

\compat{It is unclear to me that this is the best place for these comments, which seem more about $\bN$ than $\bH$. I'm leaving them for now.}
The universal cover of $\bH$ is the Necker cube surface $\bN$. Thus a loop lifts from $\bH$ to $\bN$ if and only if it is null-homologous, or equivalently if it has trivial algebraic intersection with two of the curves $\beta^\ast_j$. 
As we take the faces of $\bN$ to be parallel to the coordinate axes, the deck transformations of the covering $\bN \to \bH$ are realized by translations of $\ZZ^3$. Covering space theory associates each deck transformation with an element of $\pi_1(\bH)$. The translation associated to $\alpha \in \pi_1(\bH)$ has components given by the intersection numbers $i(\alpha, \beta_j^\ast)$ for $j \in \{0, 1, 2\}.$

\begin{prop}
\label{aut H}
The $G_4$-automorphism group of $\bH$ is isomorphic to the cyclic group of order six, and is generated by the automorphism $\varphi_{\bH}$ depicted in \Cref{fig:descendrotation}. The induced action of $\varphi_\bH$ on homology sends $\beta_i^\ast$ to $-\beta_{i+1}^\ast$ with subscript addition in $\ZZ/3\ZZ$.
\end{prop}

\begin{figure}
\includegraphics[width=2in]{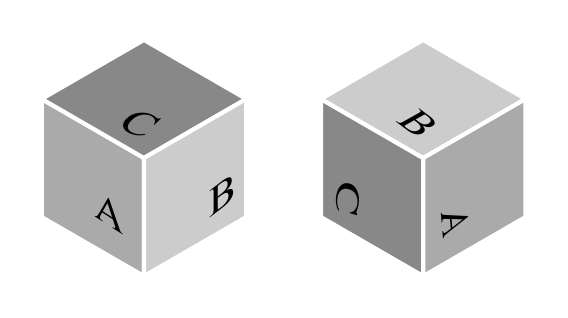}
\caption{The automorphism $\varphi_{\bH}:\bH\to\bH$. }
\label{fig:descendrotation}
\end{figure}

\begin{proof}
\Cref{fig:descendrotation} depicts $\varphi_\bH$. The reader should be able to recognize from this that $\varphi_\bH$ is a $G_4$-automorphism, that it is order six, and that it acts as claimed on homology. To see this is the full $G_4$-automorphism group note that
an isometry  of $\bH$ must preserves the paths of length one joining the singularities. The collection of such paths form a graph: the union of boundaries of the squares. Each square has two opposite vertices of valence three and the other vertices represent the same singularity with valence six. An automorphism must respect the valence of vertices, and must be orientation-preserving. So, the stabilizer of a square is at most order two. If the automorphism group were any larger than order six, the stabilizer of a square would have to be larger than two.
\end{proof}

\subsection{The half-translation surface cover of \texorpdfstring{$\bH$}{H}}
\label{sect:cover of H}
The surface $\bH$ is a quarter-translation surface. We'll denote its canonical half-translation surface cover $\bHt$ and denote the corresponding covering map by $\pi_\bH: \bHt \to \bH$.
The surface $\bHt$ is depicted in \Cref{fig:homologycover}. The surface has genus two and four singularities each with cone angle $3 \pi$.

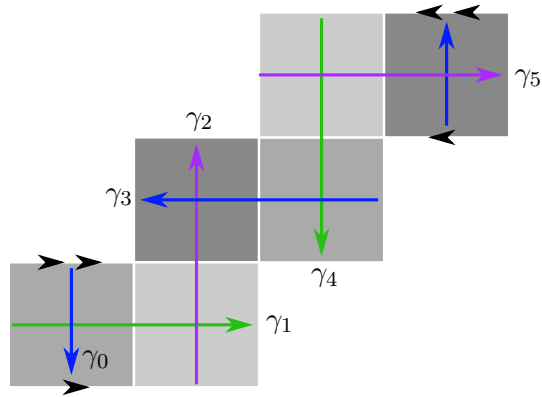
\begin{figure}[htb]
\labellist
\small\hair 2pt
 \pinlabel {$\gamma_0$} [ ] at 63 18
 \pinlabel {$\gamma_1$} [ ] at 157 36
 \pinlabel {$\gamma_2$} [ ] at 116 140
 \pinlabel {$\gamma_3$} [ ] at 75 101
 \pinlabel {$\gamma_4$} [ ] at 180 60
 \pinlabel {$\gamma_5$} [ ] at 284 162
\endlabellist
\centering
\includegraphics[width=3in]{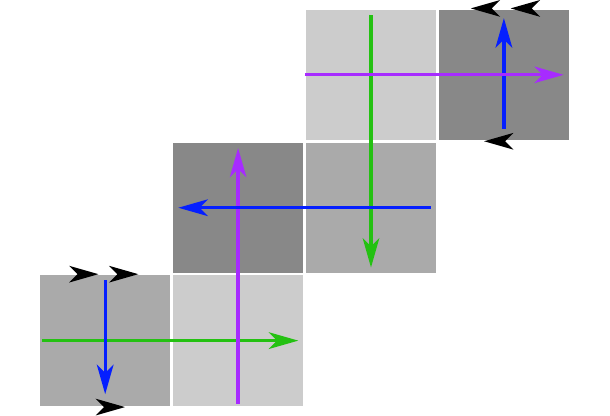}
\caption{The half-translation surface $\bHt$ is shown with curves $\gamma_0, \ldots, \gamma_5$. Darts indicate edges glued by $180^\circ$ rotations. Each unmarked edge is glued to another edge by a translation by a vector perpendicular to the edge.}
\label{fig:homologycover}
\end{figure}

The curves $\gamma_0, \ldots, \gamma_5$ shown in \Cref{fig:homologycover}
are lifts of the curves $\beta_0^\ast, \beta_1^\ast, \beta_2^\ast$ in $\bH$. We've chosen names for these curves such that $\gamma_j$ is a lift of $\beta_i^\ast$ if $j \equiv i \pmod{3}$. We have chosen our lifts in order to ensure that the algebraic intersection numbers satisfy
\begin{equation}
\label{eq:gamma intersect}
i(\gamma_j, \gamma_{j+1})=1
\quad \text{for all $j$}
\end{equation}
with subscript addition taken modulo six. Other pairs of curves do not intersect. In particular then, the quadruple of curves $(\gamma_0, \gamma_3; \gamma_1, \gamma_4)$ form a standard symplectic basis for $H_1(\bHt; \QQ) \cong \QQ^4$.
Based on this fact and the intersection number formula, it can be shown that
\begin{equation}
\label{eq:gamma id}
\gamma_0+\gamma_2+\gamma_4=0 
\quad \text{and} \quad
\gamma_1+\gamma_3+\gamma_5=0
\quad \text{in $H_1(\bHt; \ZZ)$}.
\end{equation}

By virtue of $\bHt$ being the minimal $G_2$-cover of $\bH$, each $G_4$ automorphism of $\bH$ lifts to a $G_2$-automorphism of $\bHt$; see \Cref{commute-branch-cover}. We define $\tilde \varphi_\bH$ to be the $G_2$-automorphism of $\bHt$ obtained by lifting $\varphi_\bH$. The surface $\bHt$ also has a $G_4$-automorphism that is a deck transformation of the covering $\bHt \to \bH$. Perhaps surprisingly, $\bHt$ has additional symmetry:

\begin{prop}
\label{G4 auto}
The $G_4$-automorphism group of $\bHt$ is order $24$, and any orientation-preserving isometry between two squares of $\bHt$ can be extended to a $G_4$-automorphism. We have $(\tilde \varphi_\bH)_\ast(\gamma_i)=-\gamma_{i+4}$.
In general, for each $G_4$-automorphism $\psi:\bHt \to \bHt$ there is a $k \in \ZZ/6\ZZ$ and two signs $s,t \in \{\pm 1\}$ such that
$$\psi_\ast(\gamma_i)=s (-1)^{\frac{(1-t)i}{2}} \gamma_{k + t i} \quad \text{for all $i \in \ZZ/6\ZZ$.}$$
Conversely, every such action on $\{\gamma_i\}$ is realized by a $G_4$-automorphism.
\end{prop}

\begin{figure}[htb]
\centering
\includegraphics[width=\textwidth]{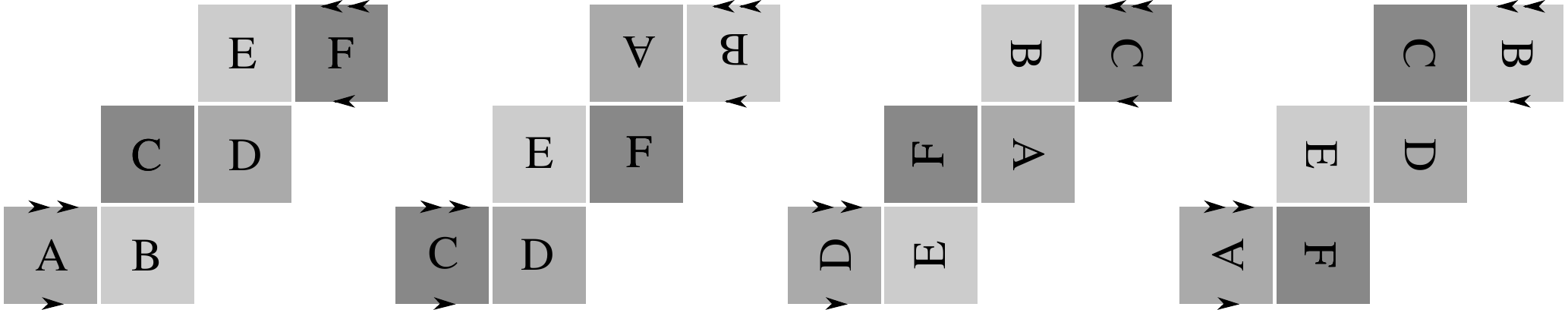}
\caption{The surface $\bHt$ with labeled squares is depicted far left. The remaining groups of squares indicate the action of three maps on these labeled squares. From left to right, these are the images under $\tilde \varphi_\bH$, the images under the unique nontrivial deck transformation of $\bHt \to \bH$, and images under a $G_4$-automorphism rotating the square labeled $A$ by $90^\circ$.}
\label{fig:order24}
\end{figure}

\begin{proof}
The $G_4$-automorphism group is generated by the maps described in \Cref{fig:order24}. Observe that the lift $\tilde \varphi_\bH$ of $\varphi_\bH$ is still order six and acts on labeled squares according to the permutation described in cycle notation as $(AEC)(BFD)$. By inspection, we observe $\tilde \varphi_\bH$ acts as claimed on homology. Let $\tilde \delta:\bHt \to \bHt$ be the nontrivial deck transformation of the cover $\bHt \to \bH$. By definition of the curves $\gamma_i$, we have $\tilde \delta_\ast(\gamma_i)=\gamma_{i+3}$. The action on squares as depicted in the figure can be deduced from the action on these curves. Observe that the action on squares is given by $(AD)(BE)(CF)$. Thus $\langle \tilde \varphi_\bH, \tilde \delta\rangle$ acts transitively on squares.

Let $\tilde \epsilon:\bHt \to \bHt$ be the third automorphism depicted in \Cref{fig:order24}, which rotates the square labeled $A$ by $90^\circ$. Thus, the stabilizer of square $A$ in the automorphism group is $\langle \tilde \epsilon\rangle$. Combining this with the transitive action of $\langle \tilde \varphi_\bH, \tilde \delta\rangle$ tells us that an element of $\langle \tilde \varphi_\bH, \tilde \delta, \tilde \epsilon\rangle$ extends any orientation-preserving isometry between squares as claimed. Since a $G_4$-automorphism must send squares to squares, this is the largest possible size for the $G_4$-automorphism group.

By inspection we see that $\tilde \epsilon_\ast(\gamma_i)=(-1)^i \gamma_{1-i}$. It is elementary to check that the generated group of actions on homology is as described.
\end{proof}

\subsection{The half-translation surface cover of \texorpdfstring{$\bN$}{N}}

The Necker cube surface $\bN$ is an infinite quarter-translation surface, and we'll call its half-translation surface cover $\bNt$. As a consequence of \Cref{commute-branch-cover} we obtain the following commutative diagram: 
\begin{equation}
\begin{tikzcd}
	\bNt \arrow[d, "\pi_\bN"] \arrow[r]
		& \bHt \arrow[d, "\pi_\bH"] \\
	 \bN \arrow[r]
	& \bH.
\end{tikzcd}
\label{eq:cd}
\end{equation}
Here, $\pi_\bN$ and $\pi_\bH$ denote the canonical half-translation covers of quarter-translation surfaces $\bN$ and $\bH$, respectively.

Define the following homology classes  in $H_1(\bHt;\ZZ)$:
$$\beta_j = \gamma_j + \gamma_{j+3} \quad \text{for $j \in \{0,1,2\}$}.$$
We have defined $\beta_j$ to be the sum of the homology classes of the two lifts of $\beta_j^\ast$ to $\bHt$. It follows that we have
\begin{equation}
\label{eq:intersection numbers match}
i_{\bHt}(\beta_j, \tilde \alpha) = i_{\bH}(\beta_j^\ast, \pi_\bH(\tilde \alpha)\big)
\end{equation}
for all $\tilde \alpha \in H_1(\bHt;\ZZ)$ and all $j \in \{0,1,2\}$.
Consequently:

\begin{prop}
\label{prop:lifting}
A loop $\tilde \alpha$ in $\bHt$ lifts to $\bNt$ if and only if $i(\beta_j,\tilde \alpha)=0$ for two (equivalently all) $j \in \{0,1,2\}.$
\end{prop}
\begin{proof}
First of all, suppose $\tilde \alpha$ has a lift to a loop $\tilde \eta$ in $\bNt$. 
Let $\eta=\pi_\bN(\tilde \eta)$ be the projection to $\bN$ and $\alpha=\pi_\bH(\tilde \alpha)$ be the projection to $\bH$. By \eqref{eq:cd}, $\eta$ is a lift of $\alpha$ so we know $i(\beta^\ast_j, \alpha)=0$ for each $j$. But then by \eqref{eq:intersection numbers match}, we know that $i(\beta_j,\tilde \alpha)=0$ for each $j$.

Now suppose that $i(\beta_j,\tilde \alpha)=0$ for two $j$.
Note that the covering $\bNt \to \bHt$ is unbranched because $\bN \to \bH$ is; see \Cref{commute-branch-cover}.
It follows that both lifting and algebraic intersection number are homotopy invariant, so we may assume that $\tilde \alpha$ does not pass through the singularities of $\bHt$. Let $\tilde x_0$ denote a choice of a starting point of $\tilde \alpha$, which we also use as the basepoint of $\bHs$. Since $\bHt$ is a half-translation surface, $\holhom(\tilde \alpha) \in G_2$. Choose a lift $\tilde y_0 \in \bNt$ of $\tilde x_0$, and define $y_0=\pi_\bN(\tilde y_0)$ and $x_0=\pi_\bH(\tilde x_0)$. \Cref{eq:cd} implies that $y_0$ projects to $x_0$ under $\bN \to \bH$. Let $\alpha=\pi_\bH(\tilde \alpha)$. Then we get $i(\beta_j^\ast, \alpha)=0$  for the same two $j$ from 
\eqref{eq:intersection numbers match}. These two $\beta_j^\ast$ generate the homology on the torus $\bH$, so $\alpha$ lifts to a loop $\eta$ in $\bN$ starting at $y_0$. Since the covering maps are local isometries, $\eta$ has the same holonomy as $\alpha$ (up to conjugation depending on which developing maps we use). Thus, $\holhom(\eta)\in G_2$ and so $\eta$ has a lift to $\tilde \eta$ on $\bNt$ starting at $\tilde y_0$. We claim that $\tilde \eta$ is a lift of $\tilde \alpha$. Let $\tilde \alpha'$ denote the projection of $\tilde \eta$ to $\bHt$. To prove the claim, it is enough to show that $\tilde \alpha'=\tilde \alpha$. But, $\pi_\bH(\tilde \alpha')=\pi_\bH(\tilde \alpha)=\alpha$. Since
$\tilde \alpha'$ and $\tilde \alpha$ start at the same point, $\tilde x_0$, and are both lifts of the same curve $\alpha$, and $\tilde \alpha$ does not pass through branch points of $\pi_\bH$, the curves $\tilde \alpha'$ and $\tilde \alpha$ coincide.
\end{proof}

We let $W_+ = \Span~\{\beta_0, \beta_1, \beta_2 \} \subset H_1(\bHt; \QQ)$. By \eqref{eq:gamma id}, we have
$\beta_0+\beta_1+\beta_2=0$ and $\beta_0$ and $\beta_1$ are distinct sums of classes in our symplectic basis $(\gamma_0, \gamma_3; \gamma_1, \gamma_4)$, so $W_+$ is a two dimensional subspace. As a consequence of \Cref{prop:lifting} and standard covering space theory, we have:

\begin{cor}
\label{cor:lifting}
Let $h:\bHt \to \bHt$ be a homeomorphism. Then $h$ has a lift $\tilde h:\bNt \to \bNt$ if and only if $h_\ast(W_+)=W_+$.
\end{cor}

See \cite{HW12} for a more detailed discussion of this lifting criterion.

\begin{cor}
\label{G4 auto lift}
Let $\psi:\bHt \to \bHt$ be the $G_4$-automorphism determined by $k \in \ZZ/6\ZZ$ and two signs $s,t \in \{\pm 1\}$ as in \Cref{G4 auto}.
Then $\psi$ lifts to a homeomorphism $\bNt \to \bNt$ if and only if $t=1$. If $\psi$ does not lift, then $\psi_\ast(W_+)=W_-$ where
$$W_- = \Span~\{\gamma_0-\gamma_3, \gamma_1-\gamma_4, \gamma_2-\gamma_5 \} \subset H_1(\bHt; \QQ).$$
\end{cor}
\begin{proof}
If $t=1$, then from the formula in \Cref{G4 auto}, we see that 
$$\psi_\ast(\beta_i)=\psi_\ast(\gamma_i+\gamma_{i+3})=s (\gamma_{k+i}+\gamma_{k+i+3}) \in W_+.$$
Thus in this case $\psi_\ast$ preserves $W_+$ and so $\psi$ lifts by \Cref{cor:lifting}. On the other hand, if $t=-1$ and $i$ is even, we have
$$\psi_\ast(\beta_i)=\psi_\ast(\gamma_i+\gamma_{i+3})=s (\gamma_{k-i}-\gamma_{k-i-3}) \in W_-.$$
Iterating over $i$ even, we see $\psi_\ast(W_+)=W_-$. We can see that $W_+ \neq W_-$ because $\Span(W_+ \cup W_-)=H_1(\bHt; \QQ)$, while $W_+$ is only two dimensional. Thus, $\psi$ does not lift.
\end{proof}

%
%

\subsection{Multi-twists of \texorpdfstring{$\bHt$}{H}}
\label{sect:multi-twists}
The half-translation surface $\bHt$ has a larger affine symmetry group, but we will focus on a subgroup that is generated by multi-twists. This subgroup turns out to be isomorphic to the free group of rank two, $F_2$. Since the horizontal and vertical twists will generate our subgroup, we will use $h$ and $v$ to denote the generators of $F_2$.

The surface $\bHt$ decomposes in the horizontal direction into cylinders whose circumference is twice their widths. Thus $\bHt$ admits an affine automorphism $\authom(h):\bHt \to \bHt$ which acts as a single right Dehn twist in each horizontal cylinder and whose derivative is $\pm \matrixhom(h) \in \PSL(2,\RR)$ where $\matrixhom(h)=(\begin{smallmatrix}1 & 2 \\ 0 & 1\end{smallmatrix})$. 

Similarly, considering the vertical decomposition of $\bHt$ we see that there is an affine automorphism $\authom(v):\bHt \to \bHt$ which acts as a single left Dehn twist in each vertical cylinder and has derivative $\pm \matrixhom(v) \in \PSL(2,\RR)$ where $\matrixhom(v) = (\begin{smallmatrix}1 & 0 \\ 2 & 1\end{smallmatrix})$. 

Thus, there is a homomorphism from $F_2$ to the orientation-preserving affine automorphism group
$$\authom:F_2 \to \Aff_{+1}(\bHt)$$
such that $\authom(h)$ and $\authom(v)$ are the multi-twists defined above. Also, the definitions of $M(h)$ and $M(v)$ extend to a  homomorphism 
\begin{equation}
\label{eq:matrixhom}
\matrixhom: F_2 \to \SL(2,\RR) \quad \text{such that} \quad D \circ \authom(g)=\pm \matrixhom(g) \text{ for all $g \in F_2$}
\end{equation}
where $D: \Aff_{+1}(\bHt) \to \PSL(2,\RR)$ is the derivative.

\subsection{The congruence \texorpdfstring{$2$}{2} subgroup}
\label{sect:cong2}

Let $\Gamma_2 \subset \SL(2,\RR)$ be the  principal congruence subgroup of level 2, i.e. 
given $M \in \SL(2,\RR)$, we have $M \in \Gamma_2$ if and only if $M \pmod{2}$ is the identity matrix in $\SL(2,\ZZ/2\ZZ)$. It is an elementary exercise (see \cite[pp. 34, Exercise 6]{Yoshida}) that $\Gamma_2$ is generated by $\matrixhom(h)$, $\matrixhom(v)$ and $-I$. Since $-I$ is an involution in the center of $\SL(2,\RR)$, we have
\begin{equation}
\label{eq:Gamma 2 -I}
\Gamma_2 = \matrixhom(F_2) \cup \{-N:~ N \in \matrixhom(F_2)\}.
\end{equation}
In particular, we have:

\begin{prop}
The subgroup $D \circ \authom(F_2) \subset \PSL(2,\RR)$ is the image of $\Gamma_2 \subset \SL(2,\RR)$ under the natural homomorphism $\SL(2,\RR) \to \PSL(2,\RR)$.
\end{prop}

It will be important for us to understand the action of $\Gamma_2$ on rational directions, so we make the following definitions:
\begin{defn}
Let $\cO\subset \ZZ^2$ be the set of all vectors with components that are relatively prime and both odd. Let $\cE_1$ (respectively, $\cE_2$) denote the subsets of $\ZZ^2$ consisting of relatively prime pairs whose first (respectively, second) component is even. Let $\cE=\cE_1 \cup \cE_2$.
\end{defn}

\begin{lemma}
	\label{lem:transitive}
	The sets $\cO$, $\cE_1$, $\cE_2$ are each invariant under left-multiplication by elements of $\Gamma_2$. Moreover, the group acts transitively on each set.
	\begin{proof}
The fact that the generating elements of $\Gamma_2$ preserves each set follows from an argument in modular arithmetic and an application of B\'ezout's identity. Let $v=(x,y)\in\cO$, and $u=(1,1)$. Let $A=(\begin{smallmatrix}a & b \\ c & d\end{smallmatrix})$. Since $x,y$ are relatively prime, there exist integers $m,n$ such that $mx+ny=1$. Choose $b=n(xy-1)$, $c=m(xy-1)$, $a=x-b$,  and $d=y-c$. Since $xy-1$ is even, so are $b$ and $c$. Since $x$ and $y$ are odd, so are $a$ and $d$. Therefore, $A$ modulo $2$ is the identity in $\SL(2,\ZZ/2\ZZ)$. A calculation of $\det A$ shows that it is $1$. Thus, $A$ is in $\Gamma_2$. With our choice of coefficients, we see $Au=v$. This proves that $\Gamma_2$ acts transitively on $\cO$. 
		
		Now let $v=(x,y)\in \cE_2$, and $u=(1,0) \in \cE_2$. Again, let $m,n$ be integers such that $mx+ny=1$. Without loss of generality, suppose $n$ is even. (If $n$ was odd, then $(m-y)x + (n+x) y=1$ and $n+x$ is even.) Since $n$ is even and $mx+ny=1$, $m$ must be odd. Let $A=(\begin{smallmatrix}a & b \\ c & d\end{smallmatrix})$ with $a=x$, $b=-n$, $c=y$, $d=m$. Then $\det A=1$ and $A\in \Gamma_2$. Since $Au=v$, we conclude that the action on $\cE_2$ is transitive. The proof is similar for $\cE_1$.
	\end{proof}
\end{lemma}

\subsection{Closed geodesics in \texorpdfstring{$\bHt$}{H}}

Because $\bHt$ is a square-tiled surface built from finitely many squares, a nonsingular geodesic is closed if and only if it has rational slope. In the next two lemmas, we explain how the geometry and topology of closed geodesics of $\bHt$ depend on their direction.

\begin{lemma}
\label{lem:even geodesics}
Suppose $\eta$ is a simple closed geodesic on $\bHt$ that is parallel to $(q,p) \in \cE$. Then there exists $g \in F_2$ and $i \in \{0,\ldots, 5\}$ such that $\authom(g) \circ \gamma_i$ is homologous to either $\eta$ or $-\eta$ as an element of $H_1(\bHt; \ZZ)$. Furthermore, $\eta$ has length $2 \sqrt{p^2+q^2}$. 
\end{lemma}
\begin{proof}
	By \eqref{eq:Gamma 2 -I} and \Cref{lem:transitive}, there is a $g \in F_2$ such that $$\left[\matrixhom(g)\right] \begin{pmatrix}
	q \\
	p
\end{pmatrix}	 \in \{(\pm 1,0),(0,\pm 1)\}.$$	
	Then the derivative of $\authom(g)$ is $\pm \matrixhom(g) \in \PSL(2,\RR)$. 
	From remarks in \Cref{sect:geodesics}, $\authom(g) \circ \eta$
	is a closed geodesic in the horizontal or vertical direction. Since $\bHt$ has horizontal and vertical cylinder decompositions and the homology classes of the core curves are given by the curves $\gamma_i$ we see that $\authom(g) \circ \eta$ must be homologous to some $\pm \gamma_i$. Thus, $\eta$ is homologous to $\authom(g^{-1}) \circ \gamma_i$. Now observe that for any $i$,
$$\disp(\gamma_i) = \pm 2 \matrixhom(g)\begin{pmatrix}
	q \\
	p
\end{pmatrix} \in \{\pm(2, 0), \pm(0, 2)\} \subset \RR^2/\pm I$$
By \Cref{affinemap-derivative}. Thus we have
$$\disp(\eta) = \matrixhom(g^{-1}) \disp(\gamma_i)= \pm 2 \begin{pmatrix}
	q \\
	p
\end{pmatrix} \in \RR^2/\langle -I \rangle.$$
From the discussion in \Cref{sect:geodesics}, $\eta$ has length $2 \sqrt{p^2+q^2}$.
\end{proof}

\begin{lemma}
\label{lem:odd null-homologous}
	Suppose $\zeta$ is a closed geodesic on $\bHt$ with initial direction parallel to an element $(q,p) \in \cO$. Then $\zeta$ is null-homologous and has length $6 \sqrt{p^2+q^2}$. The maximal cylinder containing $\zeta$ has area six and six singularities on each boundary counting multiplicity.
\end{lemma}
\begin{proof}
	It can be checked that in the slope one direction, $\bHt$ is decomposed into a single cylinder of circumference $6 \sqrt{2}$. This is the maximal cylinder in this direction and has area six. Observe that the core curve of this cylinder is trivial as an element of $H_1(\bHt,\ZZ)$ since it has trivial algebraic intersection number with each element in our generating set $\{\gamma_0, \ldots, \gamma_5\}$.

	Let $\zeta$ be as in the lemma. By \Cref{lem:transitive}, there is a $g \in F_2$ such that $\matrixhom(g)(q,p) \in \{\pm(1,1)\}$. Then the curve $\alpha=\authom(g) \circ \zeta$ is a closed geodesic with slope one, and is therefore null-homologous.
Since the action of $\authom(g)$ induces an automorphism of $H_1(\bHt, \ZZ)$ we conclude that $\zeta$ is also null-homologous.	
Furthermore since $\authom(g^{-1})$ has derivative $\pm \matrixhom(g^{-1})$ we see that
$$\disp(\zeta)=\matrixhom(g^{-1}) \disp(\alpha)=\pm \matrixhom(g^{-1}) \begin{pmatrix}
6 \\
6
\end{pmatrix}
=\pm\begin{pmatrix}
6q \\
6p
\end{pmatrix}.$$
Thus the length of $\zeta$ is as claimed. The maximal cylinder containing $\zeta$ is the image of the slope one maximal cylinder under $\authom(g^{-1})$ and so also has area six and six singularities on each boundary counting multiplicity.
\end{proof}

\section{Drift-periodic trajectories}
\label{sec:drift}

In this section, we prove that any geodesic on $\bN$ which is parallel to a vector in $\cE$ is drift-periodic as per statement $(2)$ of \Cref{thm:main}.
We make use of Weierstrass points and the fact that $\bN$ covers the cube.


\subsection{Weierstrass points}
A compact flat surface $\bS$ of genus $g$ is said to be {\em hyperelliptic} 
if there is an order two orientation-preserving isometry $\hyp_\bS:\bS \to \bS$ which fixes $2g+2$ points of $\bS$. This involution is called the {\em hyperelliptic involution} and it is unique if $g \geq 2$ \cite[\S III.7]{FK92}. 
The fixed points of $\hyp_\bS$ are the {\em Weierstrass points} of $\bS$. 
By Riemann-Roch, the quotient $\bS/\hyp$ is a sphere.

Observe that the surface $\bHt$ is hyperelliptic. It has genus two and the Weierstrass points are the centers of the squares making up $\bHt$. We denote the hyperelliptic involution by $\hyp_\bHt:\bHt \to \bHt$. This involution rotates every square making up $\bHt$ by $180^\circ$. See \Cref{fig:Weierstrass}.

\begin{figure}
	\includegraphics[height=2in]{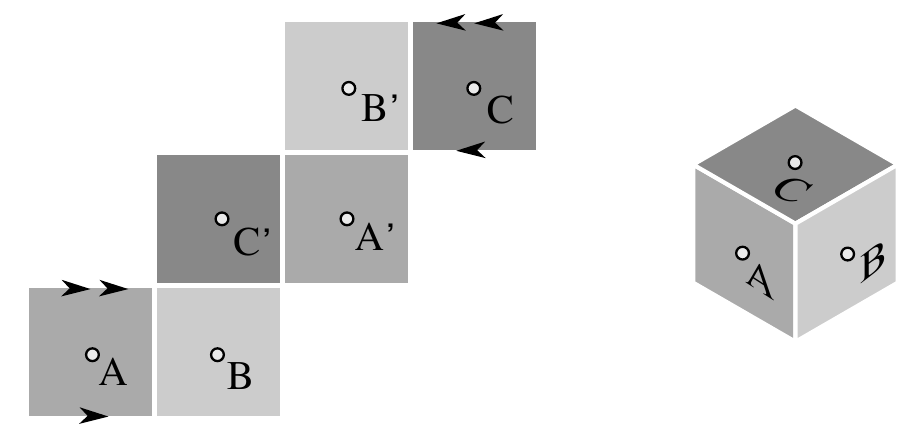}
	\caption{The surface $\bHt$ with with labeled Weierstrass points is shown on the left, while the surface $\bH$ with the images of the Weierstrass points are shown at right. Unmarked edges of $\bHt$ are glued by a translation by a vector perpendicular to the edge.}
	\label{fig:Weierstrass}
\end{figure}

We will need the following observation:

\begin{prop}
Each element of $\authom(F_2)$ permutes the Weierstrass points. 
\end{prop}

This can be proved using the uniqueness of the hyperelliptic involution mentioned earlier, but we give a sketch of a more elementary proof:

\begin{proof}
It suffices to check that the generating elements $\authom(h)$ and $\authom(v)$ permute the Weierstrass points. Using the point names from \Cref{fig:Weierstrass}, observe that
$\authom(h)$ swaps $A$ with $B$, $C'$ with $A'$, and $B'$ with $C$, while $\authom(v)$ swaps $A$ with $C$, $B$ with $C'$, and $A$ with $B'$.
\end{proof}

\Cref{lem:even geodesics} tells us that any cylinder parallel to a vector in $\cE$
is in the orbit of a horizontal or vertical cylinder under $\authom(F_2)$. As a consequence:

\begin{prop}
\label{prop:weierstrass}
Let $Y \subset \bHt$ be a maximal cylinder parallel to a vector in $\cE$. Then, $Y$ is fixed by the hyperelliptic involution and contains exactly two Weierstrass points on the central core curve.\end{prop}
\begin{proof}
In the horizontal and vertical direction, it is clear by inspection. See \Cref{fig:Weierstrass}. 

Now let $Y$ be a cylinder parallel to $\bv \in \cE$. By 
\Cref{lem:transitive} and \eqref{eq:Gamma 2 -I}, there is a 
$g \in F_2$ such that $$\matrixhom(g) \bv \in \{(\pm 1,0),(0,\pm 1)\}.$$	
Then $\authom(g)(Y)$ is a horizontal or vertical cylinder and contains two Weierstrass points $p_1$ and $p_2$ on the central core curve of $\authom(g)(Y)$ from the first paragraph. Then $Y$ contains $\authom(g^{-1})(p_1)$ and $\authom(g^{-1})(p_2)$ which are also Weierstrass points and lie on the central core curve of $Y$ since $\authom(g^{-1})$ is acting affinely. Since these two Weierstrass points are fixed and lie in the central core curve of $Y$, the hyperelliptic involution fixes $Y$.
\end{proof}

\subsection{The cube}

We make use of the following lemma:

\begin{lemma}
The Necker cube surface is a branched cover of the unit cube, $\mathbf{C}$.
\begin{proof}
We refer the reader to \Cref{fig:cubecover}, which illustrates the cover. The cube $\bC$ can be seen to be the quotient of $\bN$ under the group of isometries of $\bN$ generated by the isometries of $\RR^3$ of the form $\bv \mapsto 2 \bp - \bv$ where $\bp \in \Lambda_0$, fixing a point $\bp$ where six squares come together. (While these isometries of $\RR^3$ are orientation-reversing, they act as orientation-preserving involutions on $\bN$.)
\end{proof}
\end{lemma}

\begin{figure}
\includegraphics[width=4in]{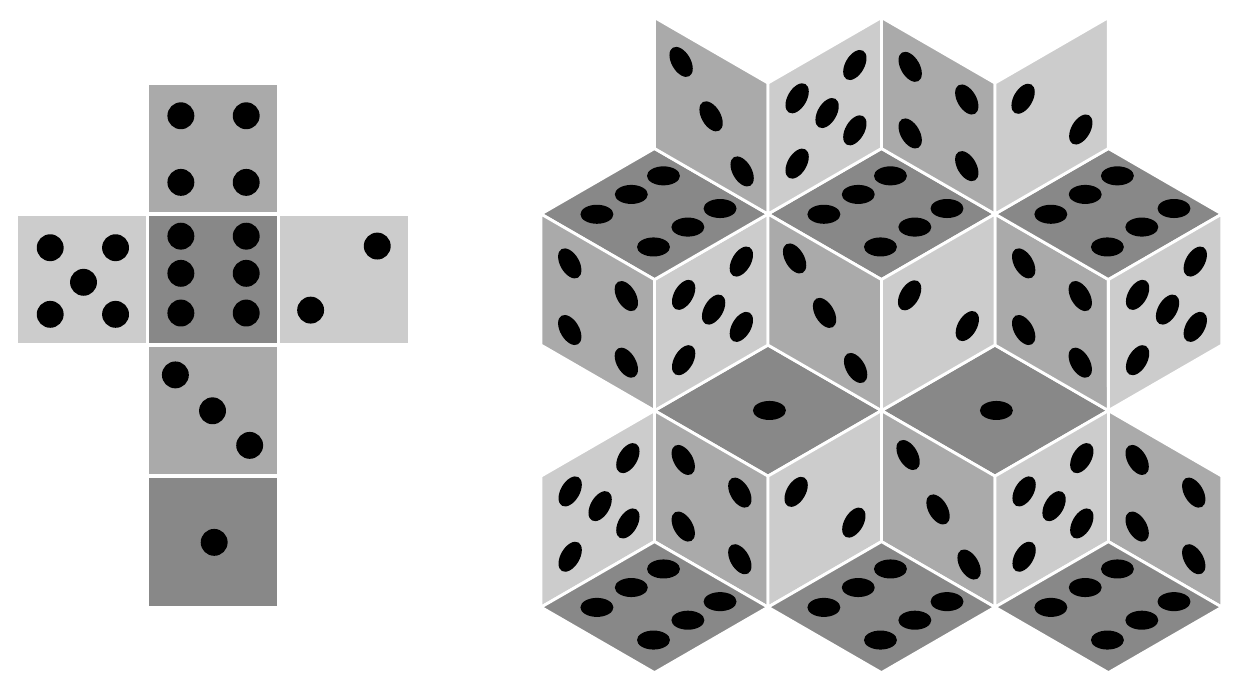}
\caption{A visual reference as to how the branched cover from $\bN$ to the unit cube is constructed.}
\label{fig:cubecover}
\end{figure}

\begin{lemma}[Cube Lemma]
\label{cube lemma}
	Let $\bC$ denote the cube and $p \in \bC$ be the center of a square face. There is no geodesic segment (not passing through a vertex) which leaves $p$ and later returns to $p$ making a $90^\circ$ angle at $p$.
\end{lemma}
\begin{proof}
	Suppose to the contrary that $\gamma:[0,L] \to \bC$ is a unit speed geodesic satisfying that $\gamma(0)=\gamma(L)=p$ and that $\gamma'(0) \perp \gamma'(L)$.
	Let $R:\bC \to \bC$ be the order four orientation-preserving isometry of $\bC$ which preserves $p$ and acts on $T_p \bC$ satisfying $R_\ast\big(\gamma'(L)\big)=\gamma'(0)$. Because $\gamma'(0)$ and $\gamma'(L)$ are orthogonal, and $R$ is acting by a rotation of order four we must have that 
	$R_\ast\big(\gamma'(0)\big)=-\gamma'(L)$.
	
	Define $\eta:[0,L] \to \bC$ by $\eta(t)=R \circ \gamma(L-t)$. Observe that 
	$$\eta'(0)=R_\ast\big(-\gamma'(L)\big)=-\gamma'(0)
	\quad \text{while} \quad
	\eta'(L)=R_\ast\big(-\gamma'(0)\big)=\gamma'(L).$$
	This is a contradiction: Since $\eta'(L)=\gamma'(L)$ these must be the same unit speed geodesics, but they are not the same geodesics since $\eta'(0) \neq \gamma'(0)$.
\end{proof}

\subsection{Drift-periodic directions}

\begin{thm}
\label{thm:even-drift-periodic}
A nonsingular geodesic in $\bN$ parallel to a vector in $\cE$ is drift-periodic.
\end{thm}
\begin{proof}
Let $\alpha$ denote a nonsingular geodesic in $\bN$ parallel to $v \in \cE$. We'll prove the theorem by contradiction. By \Cref{rational implies}, such a geodesic of rational slope is either periodic or drift-periodic. So, assume $\alpha$ is periodic.

Let $\beta$ denote the projection of $\alpha$ to a closed curve on $\bH$, and let $\tilde \beta$ be a lift to $\bHt$, which exists and is closed because the rotational holonomy around a closed geodesic is trivial. \Cref{prop:weierstrass} guarantees that $\tilde \beta$ lies in a cylinder $C \subset \bHt$ containing two distinct Weierstrass points $a$ and $b$ on its central curve. Let $\gamma:[0,L] \to C$ be a unit speed arc of the central core curve of $C$ such that $\gamma(0)=a$ and $\gamma(L)=b$.

The cylinder $C$ can be pushed down into $\bH$ and lifted to an immersed cylinder in $\bN$ that contains $\alpha$. Let $\iota:C \to \bN$ denote this immersion. 
Note that $\tilde a=\iota \circ \gamma(0)$ and $\tilde b = \iota \circ \gamma(L)$ must be centers of squares since $\gamma(0)$ and $\gamma(L)$ are Weierstrass points. 
By \Cref{isometries}, there is an order two orientation-preserving isometry $f: \bN \to \bN$ whose only fixed point is $\tilde a$ (that comes from an $180^\circ$ rotational symmetry of the rhombille tiling). Then $f$ must preserve $\iota(C)$. It follows that $f$ acts as an orientation-reversing isometry on the central core curve of $\iota(C)$. Thus $f(\tilde b)=\tilde b$, but since $f$ has only one fixed point we have $\tilde a=\tilde b$. Thus, $\iota \circ \gamma: [0,L] \to \bN$ is a closed curve. 

We claim that 
$\iota\circ \gamma$ makes a right angle at $\iota\circ \gamma(0)=\iota\circ \gamma(L)$.
Otherwise, since $\bN$ is a quarter-translation surface, this angle would have to be $180^\circ$ but in this case we can project the curve back to a closed curve in $\bH$ and lift the curve to a closed curve on $\bHt$. But this recovers $\gamma$ and by construction $\gamma$ is not closed, proving the claim. 

Because the claim holds, the image of $\iota \circ \gamma$ under the covering map $\bN \to \bC$ violates \Cref{cube lemma} giving us our desired contradiction.
\end{proof}


\section{Periodic Geodesics}
\label{sec:periodic}

In this section, we prove part of statement $(1)$ of \Cref{thm:main}, guaranteeing that a nonsingular geodesics with a rational slope that is the ratio of two odd integers is periodic. We then finish the proof of \Cref{thm:main}. Then we prove Theorems \ref{thm:order six} and \ref{thm:tiling}, demonstrating that closed geodesics have six-fold rotational symmetry and that simple closed geodesics can be used to construct tilings.

\subsection{Periodic Directions}

\begin{thm}
\label{thm:odd-odd is closed}
Let $\tilde \gamma:\RR \to \bN$ be a nonsingular geodesic whose slope is $\frac{p}{q}$ for two odd, relatively prime integers $p$ and $q$. Then $\tilde \gamma$ is closed geodesic with length $6\sqrt{p^2+q^2}$ and the maximal cylinder containing $\tilde \gamma$ has area six.
\end{thm}
\begin{proof}
Choose a lift of $\tilde \gamma$ to a geodesic $\tilde \eta$ of slope $\frac{p}{q}$ on $\bNt$. Let $\eta$ be the projection of $\tilde \eta$ to $\bHt$. By hypothesis, $\eta$ has slope $\frac{p}{q}$. \Cref{lem:odd null-homologous} guarantees that $\eta$ is null-homologous and has length $6 \sqrt{p^2+q^2}$. 
\Cref{prop:lifting} tells us that $\eta$ lifts to $\bNt$, and $\tilde \eta$ is such a lift. Thus, $\tilde \eta$ is also a periodic geodesic of length $6 \sqrt{p^2+q^2}$. Considering the double branched covering $\bNt \to \bN$, we see that $\tilde \gamma$ is the image of $\tilde \eta$ and therefore is closed. Furthermore the restriction of this covering gives a covering map of degree at most two $\tilde \eta \to \tilde \gamma$ (viewed as closed curves). We claim that the degree of $\tilde \eta \to \tilde \gamma$ is one. Suppose to the contrary that this covering map is degree two. Then the length of $\tilde \gamma$ is $3 \sqrt{p^2+q^2}$. Then this shorter curve lifts to $\bNt$ and descends to $\bHt$, giving a closed geodesic in $\bHt$ of slope $\frac{p}{q}$ or $\frac{-q}{p}$ and length $3 \sqrt{p^2+q^2}$ contradicting \Cref{lem:odd null-homologous} and proving the claim. Since $\tilde \eta \to \tilde \gamma$ has degree one, $\tilde \gamma$ also has length $6\sqrt{p^2+q^2}$. Similarly, the maximal cylinder containing $\tilde \gamma$ lifts to $\bNt$ and covers a maximal cylinder in $\bHt$ which has area six by \Cref{lem:odd null-homologous}. Since $\tilde \gamma$ has length equal to the length of the core curve of the image cylinder in $\bHt$, the covering of the cylinder in $\bHt$ must be a bijection, so the maximal cylinder containing $\tilde \gamma$ also has area six.
\end{proof}

\subsection{Proof of \texorpdfstring{\Cref{thm:main}}{Theorem \ref{thm:main}}}
\label{sect:main proof}

\begin{proof}[Proof of \Cref{thm:main}] 
Let $m$ be the slope of a nonsingular geodesic $\gamma$ in $\bN$ that is either periodic or drift-periodic. Then \Cref{periodic implies rational}, guarantees that $m \in \hat \QQ$. So $m=\frac{q}{p}$ for some relatively prime pair $(p, q)$ in either $\cO$ or $\cE$. If $(p,q) \in \cO$ then \Cref{thm:odd-odd is closed} guarantees that $\gamma$ is periodic, so both (1) and (2) are true in this case. If $(p,q) \in \cE$ then \Cref{thm:even-drift-periodic} guarantees that $\gamma$ is drift-periodic, so again both (1) and (2) are true.
\end{proof}

\subsection{Rotational symmetry of closed geodesics}

\Cref{fig:order24} depicts half-translation automorphism $\varphi_\bHt:\bHt \to \bHt$ of order six. Observe:
\begin{lemma}
\label{lem:1/6}
If $\tilde{\eta}:\RR \to \bHt$ is a periodic geodesic with least period $\ell$ and slope $\frac{q}{p}$ where $(p,q) \in \cO$, then there is a choice of sign such that 
$$\varphi_\bHt \circ \tilde \eta(x)=\tilde \eta(x \pm \tfrac{\ell}{6}) \quad \text{for every $x \in \RR/\ZZ$}.$$
\end{lemma}
\begin{proof}
From \Cref{lem:odd null-homologous}, we know there is a unique maximal cylinder $C_{p,q} \subset \bHt$ of slope $\frac{q}{p}$ that fills the surface (has full area) and is bounded by saddle connections. The automorphism $\varphi_\bHt$ sends closed geodesics of slope $\frac{q}{p}$ to closed geodesics of slope $\frac{q}{p}$ and does the same for saddle connections, and so must preserve $C_{p,q}$.
Thus $\varphi_\bHt$ acts on $C_{p,q}$ by isometry. The orientation-preserving isometry group of a cylinder either acts by translation on the cylinder (in which case all parallel closed geodesics have preserved images) or reverses the orientation of the central core curve. Note that all isometries which reverse the orientation of the central core curve have order two. Since the order of the action of $\varphi_\bHt$ on $C_{p,q}$ is six, it must be acting by translation, i.e., there is a constant $k \in \RR$ such that 
$\varphi_\bHt \circ \tilde \eta(x)=\tilde \eta(x + k)$ for every simple closed curve traveling the same direction and speed around $C_{p,q}$. The order of $\varphi_\bHt$ coincides with the order of the action of addition by $k$ on $\RR/\ell \ZZ$. Thus, $k \in \{\pm \frac{\ell}{6}\}$.
\end{proof}

We obtained $\varphi_\bHt$ by lifting the automorphism $\varphi_\bH$ of \Cref{fig:descendrotation}. Pushing down into $\bH$, we see the following:

\begin{cor}
\label{cor:1/6}
Let $\eta:\RR \to \bH$ be a periodic geodesic with least period $\ell$ and slope $\frac{q}{p}$ where $(p,q) \in \cO$.
Then there is a choice of sign such that 
\begin{equation}
\label{eq:1/6}
\varphi_\bH \circ \eta(x)=\eta(x \pm \tfrac{\ell}{6}) \quad \text{for every $x \in \RR/\ZZ$}.
\end{equation}
\end{cor}
\begin{proof}
Since $\bHt$ is the minimal half-translation cover, $\eta$ lifts to a closed geodesic $\tilde \eta:\RR/\ZZ \to \bHt$. The conclusion follows from \Cref{lem:1/6} as well as the identities $\varphi_\bH \circ \pi_\bH=\pi_\bH \circ \varphi_\bHt$ and $\eta=\pi_\bH \circ \tilde \eta.$
\end{proof}

\begin{proof}[Proof of \Cref{thm:order six}]
Let $\gamma:\RR \to \bN$ be a closed geodesic with least period $\ell$. 
By \Cref{thm:main}, we know $\gamma$ has slope $\frac{q}{p}$ for some $(p,q) \in \cO$. 
Let $p$ denote the universal covering $\bN \to \bH$ and let $\eta = p \circ \gamma:\RR \to \bH$ be the image of $\gamma$ in $\bH$. 

We claim $\eta$ also has least period $\ell$. If not, then the least period would have to be $\frac{\ell}{k}$ for some integer $k \geq 2$
and there would be a deck transformation $\delta_0:\bN \to \bN$ such that $\delta_0 \circ \gamma(t)=\gamma(t+\frac{\ell}{k})$ for all $t \in \RR$. But, since $\gamma$ has period $\ell$, $\delta_0$ would have to be periodic with period $k$ and this contradicts the fact that the deck group of $\bN \to \bH$ is isomorphic to $\ZZ^2$.

By the claim, we know \eqref{eq:1/6} is satisfied for some choice of sign. Since $\bN \to \bH$ is the universal cover, there is a lift 
$\psi:\bN \to \bN$ of $\varphi_\bH:\bH \to \bH$. Since $\varphi_\bH \circ p \circ \gamma(x)=p \circ \gamma(x \pm \frac{\ell}{6})$,
there is a deck transformation $\delta:\bN \to \bN$ such that
$$\delta \circ \psi \circ \gamma(x) = \gamma(x \pm \tfrac{\ell}{6}) \quad \text{for every $x \in \RR/\ZZ$}.$$
Since $\gamma$ has least period $\ell$, the isometry $\delta \circ \psi: \bN \to \bN$ must be least period six. By \Cref{isometries}, this isometry can be extended to an isometry of $\RR^3$ preserving $\bN$ and, by the uniqueness of this extension, this extension is also least period six.
\end{proof}

\subsection{Tilings from simple closed geodesics}

In this section we prove our \hyperref[thm:tiling]{Tiling Theorem}. Essentially, the tiling appears because of the sixfold symmetry of \Cref{thm:order six} and the following additional symmetry. 

\begin{prop}
\label{prop:center}
Let $\sigma$ be a saddle connection in $\bN$ of slope $\frac{p}{q}$ with $(p,q) \in \cO$. Then the midpoint of $\sigma$ is the center of a square $S$ of $\bN$. Thus, the orientation-preserving symmetry of $\bN$ that fixes this center and rotates $S$ by $180^\circ$ preserves $\sigma$.
\end{prop}
\begin{proof}
Let $\sigma'$ be the image of $\sigma$ in $\bH$ and let $\sigma''$ be a lift to $\bHt$. Then there is a $g \in F_2$ and a saddle connection $\gamma \subset \bHt$ of slope one such that $\authom(g)(\gamma)=\sigma''$. Because $\gamma$ is a slope one saddle connection, it is a diagonal of a square. Since the $\authom$-action permutes the centers of squares, the midpoint of $\sigma''$ must also be a center of a square. It follows that the midpoint of $\sigma$ is also the center of a square.
\end{proof}

For the remainder of this section, let $C \subset \bN$ be a maximal open cylinder whose interior closed geodesics are simple. Then $\partial C$ consists of two components. As in the introduction, the component enclosed by a simple closed geodesic in the interior of $C$ will be called the {\em interior component}, $\partial_{\textit{in}} C$. The other component is the {\em outer component} $\partial_{\textit{out}} C$.

\begin{prop}
\label{simple six}
Both boundary components of $C$ are simple closed curves passing through six distinct singularities. The singularities that $\partial_{\textit{in}} C$ passes through all have the same cone angle. The same holds for $\partial_{\textit{out}} C$.
\end{prop}
\begin{proof}
Let $C'$ be the image of $C$ under the covering map $\bN \to \bH$, and let $\tilde C'$ be a lift to $\bHt$. \Cref{lem:odd null-homologous} guarantees that there are six singularities on each boundary component of $\tilde C'$ counting multiplicity, so this holds for $C$ as well. From \Cref{thm:order six}, we know there is an order six symmetry $R:\bN \to \bN$ that preserves $C$. The six singularities on each component lie in the same $R$-orbit, so they must have the same cone angle. Also this implies that if a singularity $p$ in a boundary component is visited with multiplicity, then $p$ must be a point whose period under $R$ is smaller than six. But the action of $R$ only has a single fixed point of cone angle $3 \pi$ and every other point has least period six. So, we need to rule out that all singularities on one boundary component are the fixed point of $R$. But this is already ruled out by \Cref{prop:center}: the saddle connections in each boundary component must have distinct singularities because symmetries of $\bN$ rotating a square by $\pi$ have no fixed singularities.

It remains to show the boundary components are simple curves. Since the singularities passed through are distinct, any self-intersection must occur between the interiors of the six saddle connections making up the boundary component. Any two saddle connections either intersect transversely, are equal, or have disjoint interiors. No two of the six saddle connections can be equal because their endpoints differ. They cannot intersect transversely, because each boundary component is the limit of simple closed geodesics. Thus, the interiors of saddle connections are pairwise disjoint as desired.
\end{proof}

\begin{lemma}
\label{632 detector}
Let $A:\RR^2 \to \RR^2$ be a rotation of order six fixing the point $p_A \in \RR^2$. 
Let $B:\RR^2 \to \RR^2$ be a rotation of order two fixing a point $p_B \neq p_A$.
Then $\langle A, B\rangle$ is the crystallographic group ${\mathbf p}{\mathbf 6}$ with presentation
\begin{equation}
\label{eq:p6}
\langle A, B ~|~ A^6=B^2=(AB)^3=I\rangle.
\end{equation}
Let $D:\langle A,B \rangle \to \SO(2)$ denote the derivative mapping. If $\gamma:\RR/\ZZ \to \RR^2$ is a simple closed curve of area  (Lebesgue measure) zero through $p_B$ such that:
\begin{enumerate}
\item $A \circ \gamma(t)=\gamma(t + \frac{1}{6})$ for each $t \in \RR/\ZZ$, and
\item $B \circ \gamma(t)=\gamma(\frac{1}{6}-t)$ for each $t \in [0,\frac{1}{6}]$,
\end{enumerate}
then the region $R$ bounded by $\gamma$ is a fundamental domain for the action of $\ker D$ on $\RR^2$, and each element of $\langle A, B\rangle$ permutes the tiles in the orbit $\{g(R):~g \in \ker D\}$. Furthermore, the symmetry group of the tiling $\{g(R)\}$ is precisely $\langle A, B\rangle$ unless $R$ is a regular hexagon, in which case the tiling also has reflective symmetry and has the symmetries of the crystallographic group ${\mathbf p}{\mathbf 6}{\mathbf m}$.
\end{lemma}

It is worth pointing out that M. C. Escher used partial symmetries of boundaries of tiles similar to the symmetry described in (2)
to deform a polygonal tiling into creatures and other objects. (These are the $S$ edges described in \cite{S10}.)
A famous example in which this concept was used is the tiling by lizards featured in his {\em Reptiles} lithograph, though of course
these lizards do not have the sixfold rotational symmetry described in (1).

\begin{proof}
For the fact that $\langle A, B\rangle$ is ${\mathbf p}{\mathbf 6}$, we refer the reader to \cite[\S 4.5]{CM1972} which also gives the presentation \eqref{eq:p6}. 

Observe that the image $D(\langle A, B \rangle)$ is a subgroup of $\SO(2)$ isomorphic to $\ZZ/6\ZZ$. The derivative map altered to map to $\ZZ/6\ZZ$ is then determined by
$$A \mapsto 1 \pmod{6}
\quad \text{and} \quad
B \mapsto 3 \pmod{6}.$$
Therefore, $A^3 B \in \ker D$. Observe that $A^3 B$ is a nontrivial translation because $A$ and $B$ fix distinct points.
Since $\ker D$ is normal, $A^2BA^{-1}$ is also in $\ker D$ and because it was conjugated by a rotation of order six, the
two translations $A^3 B$ and $A^2BA^{-1}$ generate a subgroup isomorphic to $\ZZ^2$. It is a simple exercise to check that these
two elements generate $\ker D$. The quotient $Y=\RR^2/\ker D$ is therefore a torus
and using the natural rhombus as a fundamental domain, we compute that $\mathit{Area}(Y)=2 \sqrt{3} |p_A-p_B|^2$.

Now consider the region $R$ bounded by $\gamma$. We claim that the area of $R$ is also $2 \sqrt{3} |p_A-p_B|^2$. 
Let $C=B \circ A$. Then, $C$ is a rotation of order three so has a unique fixed point $p_C$.
By (1) and (2), $\gamma(0)$ is also fixed by $C$. Therefore, $p_C=\gamma(0)$. Then by (1) we have $\gamma(\frac{i}{6})=A^i(p_C)$ for $i=0, \ldots, 5$. For each such $i$, let $\eta_i$ be the closed curve $\gamma|_{[\frac{i}{6},\frac{i+1}{6}]}$ followed by the line segment from $A^{i+1}(p_C)$ back to $A^i(p_C)$. By symmetry via (2), the signed area enclosed by $\eta_i$ is zero. Therefore, the area of $R$ is equal to the area of the regular hexagon $H$ whose vertices are $A^i(p_C)$ for $i=0, \ldots, 5$. Since $A^3B$ carries the first vertex $p_C$ to the fourth vertex $A^4(p_C)$, we conclude that this diagonal of $H$ has length $2|p_0-p_1|$, from which is easy to compute that $\mathit{Area}(H)=2 \sqrt{3} |p_0-p_1|^2$,
which also equals the area of $R$ as claimed.

Based on statements (1) and (2), we have
$$A^3 \circ B \circ \gamma(t)=\gamma(\tfrac{2}{3}-t) \quad \text{for $t \in [0, \tfrac{1}{6}]$},$$
thus identifying $\gamma([0,\frac{1}{6}])$ with $\gamma([\frac{1}{2},\frac{2}{3}])$ by translation.
Similarly, $A^2BA^{-1}$ identifies $\gamma([\frac{1}{6},\frac{1}{3}])$ with $\gamma([\frac{2}{3},\frac{5}{6}])$,
and $ABA^{-2}$ identifies $\gamma([\frac{1}{3},\frac{1}{2}])$ with $\gamma([\frac{5}{6},1])$.
Let $X$ denote the quotient of $R$ by these three boundary identifications. 
Since we identify boundaries using elements of $\ker D$,
there is a natural map $f:X \to Y$ that arises from restricting the covering $\RR^2 \to Y$ to the region $R$.
Consideration of the gluing maps shows that $X$ is a closed translation surface, and that $f$ is a local homeomorphism except possibly at the two points representing the equivalence classes $\{p_C, A^2(p_C), A^4(p_C)\}$ where $\{A(p_C), A^3(p_C), A^5(p_C)\}$
where (a priori) cone points may appear. But then $f$ is a branched covering map branched over at most two points. But, since the area of $X$ equals the area of $R$ which is the same as the area of $Y$, we see that $f$ has degree one.
Therefore, $f$ is a homeomorphism, and our covering $\RR^2 \to Y$ gives a covering $\RR^2 \to X$ with deck group $\ker D$.
The orbit $\{g(R):~g \in \ker D\}$ thus covers $\RR^2$ with $R$ being the fundamental domain.
The index of $\ker D$ in $\langle A, B\rangle$ is six. Since $D(A)$ has order six and $A$ stabilizes $R$, it follows that the full group $\langle A, B\rangle$ permutes the collection of tiles $\{g(R)\}$ in the orbit.

Finally, suppose that the full symmetry group $G \subset \Isom(\RR^2)$ of the tiling $\{g(R)\}$ properly contains $\langle A, B\rangle$. We will show that $R$ must be a regular hexagon as claimed. 
The stabilizer of $R$ in $G$ must be a subgroup of order $6 [G:\langle A, B\rangle] \geq 12$. The stabilizer of $R$, ${\mathit Stab}(R)$, must fix the center of mass $p_A$ of $R$ and hence is isomorphic to a subgroup of $O(2)$. Every orientation-preserving element in the stabilizer must have order at most six (from the classification of wallpaper groups), and because $A$ lives in this stabilizer we know that the only possibility is that ${\mathit Stab}(R)$ is an order twelve dihedral group (including reflections). Observe that the six points $\{A^i(p_C):~i=0,\ldots,5\}$ are special because they are the only points in $\partial R$ that are also in the boundary of two other tiles. Thus the orientation-reversing elements of ${\mathit Stab}(R)$ are then uniquely determined by their action on $\{A^i(p_C)\}$, which must reverse their cyclic ordering. In particular, we conclude that there is an orientation-reversing $E \in {\mathit Stab}(R)$ that swaps $p_C$ with $A(p_C)$. Then the reflection $E$ preserves $\gamma([0,\frac{1}{6}])$ while swapping the endpoints. The map $B$ does the same, but acts as an orientation-preserving isometry. The composition $E \circ B$ is orientation-reversing and fixes both $p_C$ and $A(p_C)$ and therefore is a reflection in the line $\overleftrightarrow{p_C A(p_C)}$. This composition $E \circ B$ fixes $p_C$, $A(p_C)$, and the simple path $\gamma([0,\frac{1}{6}])$ which joins $p_C$ to $A(p_C)$. It is an elementary observation that the only simple path joining two points on the axis of a reflection that is preserved by the reflection is the line segment between the points, so we conclude that $\gamma([0,\frac{1}{6}])$ is a parameterization of $\overline{p_C A(p_C)}$. Using the fact that $R$ is $A$-invariant, we see that $\gamma$ is a parameterization of a regular hexagon and the tiling is the usual tiling by regular hexagons which has symmetry group ${\mathbf p}{\mathbf 6}{\mathbf m}$ \cite[\S 4.5]{CM1972}.

The result further claims that if $R$ is a regular hexagon, then its symmetry group must be ${\mathbf p}{\mathbf 6}{\mathbf m}$. To see this, note that the vertices must be the points $\{A^i(p_C)\}$, because otherwise $p_C$ would lie in the interior of an edge of $R$ and $R \cap C(R)$ would have interior contradicting statements above (since $C$ rotates by $\pm 120^\circ$). Thus if $R$ is a regular hexagon, we are in the situation where the tiling is the usual hexagonal tiling as above.
\end{proof}

The proof of \Cref{thm:tiling} uses the isometric projection $\proj:\RR^3 \to \bone^\perp$ from \Cref{sect:isom}. We need the following observation:
\begin{prop}
\label{shear}
Let $\gamma:[0,1] \to \bN$ be a geodesic segment such that $\gamma(0)$ and $\gamma(1)$ lie in the interiors of adjacent squares in $\bN$ and such that the only edge $\gamma$ crosses is the edge joining these squares. Then $\proj( \gamma ) \subset \bone^\perp$ consists of two segments that form an angle that is not equal to $180^\circ$.
\end{prop}

\begin{figure}[htb]
	\includegraphics[height=1in]{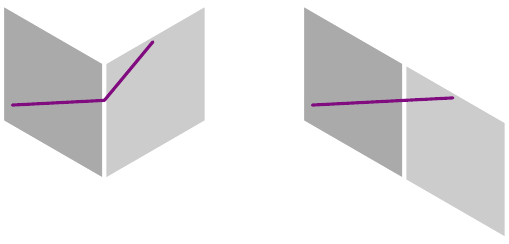}
	\caption{Left: The image under $\proj$ of a geodesic segment $\gamma$ on $\bN$. Right: The projected image of $g(\gamma)$ where $g$ rotates the lighter square by $90^\circ$. Observe that the picture at left can also be obtained from the picture at right by applying a vertical shear to the lighter projected rhombus.}
	\label{fig:shear}
\end{figure}

\begin{proof}
Let the two squares be $S_0, S_1 \subset \bN$, where $\gamma(i) \in S_i^\circ$ for $i \in \{0,1\}$. 
Let $e=S_0 \cap S_1$ be the common edge, and let $t_0 \in (0,1)$ be the point for which $\gamma(t_0) \in e$.
Let $P$ be the plane containing $S_0$. Let $f \in \Isom_+(\RR^3)$ be the isometry rotating by $90^\circ$ about the line extending $e$, chosen such that $f(S_1)$ lies adjacent to $S_0$ in $P$. Let $g:S_0 \cup S_1 \to \RR^3$ be the isometric embedding such that $g|_{S_0}$ is the identity and $g|_{S_1} \equiv f$. 
Since $g(S_0 \cup S_1)$ is contained in a plane and $g$ is an isometric embedding, $g \circ \gamma$ is a line segment.

Now we compare the projections $\proj \circ \gamma$ and $\proj \circ g \circ \gamma$ in $\bone^\perp$ as illustrated in \Cref{fig:shear}. Observe that the restrictions of these paths to $[0, t_0]$ agree, and $\proj \circ g \circ \gamma$ is contained in a line $\ell$. Our task is to show that $\proj \circ \gamma|_{[t_0,1]}$ is not contained in $\ell$. Observe that the composition
\begin{equation}
\label{eq:shear}
\proj \circ g(S_1) \xrightarrow{\proj|_{g(S_1)}^{-1}} g(S_1) \xrightarrow{~g~} S_1 \xrightarrow{\proj|_{S_1}} \proj(S_1)
\end{equation}
extends to a projective transformation of the plane $\bone^\perp$ \cite[\S 4]{Berger}. This composition is relevant to our problem
because it carries $\proj \circ g \circ \gamma|_{[t_0,1]}$ to $\proj \circ \gamma|_{[t_0,1]}$.
Observe that there is a nontrivial shear of $\bone^\perp$ that carries $\proj \circ g(S_1)$ to $\proj(S_1)$ and respecting the identities of vertices of these rhombi. A shear is also a projective transformation, 
and since it agrees at four points in general position (namely, the vertices), by the First Fundamental Theorem of Projective Geometry (\cite[\S 4.5.10]{Berger}) our shear must give the composition in \eqref{eq:shear}. The shear preserves $\proj(e)$, and therefore only preserves lines parallel to $\proj(e)$. Since $\ell$ intersects $\proj(e)$, it is not preserved by the shear and so $\proj \circ \gamma|_{[t_0,1]}$ is not contained in $\ell$.
\end{proof}

\begin{proof}[Proof of \Cref{thm:tiling}]
Let $C$ be a maximal cylinder whose interior is foliated by simple closed geodesics. Let $\psi \in \Isom(\RR^3, \bN)$ be the order-six symmetry preserving $C$. Let $\gamma:[0,1] \to \partial C \subset \bN$ be one of the closed boundary curves, parameterized at constant speed such that $\gamma(0)$ is a singularity. Let $R \subset \bN$ be the region bounded by $\gamma$. 
Then $R$ is a topological disk by \Cref{simple six} and the Jordan curve theorem. 
By the Brouwer fixed point theorem, $\psi$ has a fixed point in $R$, and as a consequence of \Cref{isometries} the fixed point $z \in \bN$ is unique and so must lie in the interior of $R$.
By \Cref{thm:order six}, after possibly reversing the orientation of $\gamma$, we have $\psi \circ \gamma(t)=\gamma(t + \frac{1}{6})$. Then each $\gamma(\frac{i}{6})$ is singular. By \Cref{simple six}, there are six saddle connections making up $\gamma$ and they must be the arcs $\gamma([\frac{i}{6}, \frac{i+1}{6}])$ for $i=0, \ldots, 5$. By \Cref{prop:center}, there is an involution $\iota \in \Isom_+(\RR^3, \bN)$ that fixes the saddle connection $\gamma([0, \frac{1}{6}])$ while swapping the endpoints. Let $x$ be the unique fixed point of $\iota$, which is the center of a square. Observe that $\gamma(\frac{1}{12})=x$.

Now consider the perpendicular projection $\proj:\RR^3 \to \bone^\perp$ of \Cref{sect:isom}. By \Cref{isometries}, $A=\psi|_{\bone^\perp}$ is an order six rotation fixing $p_A=\proj(z)$ and $B=\iota|_{\bone^\perp}$ is an order two rotation fixing $p_B=\proj(x)$. Since $\proj|_{\bN}:\bN \to \bone^\perp$ is a homeomorphism, 
and $z \in R^\circ$ and $x \in \partial R$ are distinct, we see that $p_A \neq p_B$.
Therefore \Cref{632 detector} tells us that $\langle A, B\rangle$ is a discrete subgroup of $\Isom(\bone^\perp)$ isomorphic to ${\mathbf p}{\mathbf 6}$ with presentation given in \eqref{eq:p6}. Furthermore $\proj \circ \gamma$ is a curve satisfying statements (1) and (2) of \Cref{632 detector}, so the disk $\proj(R)$ bounded by $\proj \circ \gamma$ tiles $\bone^\perp$ under the action of the kernel $K \subset \langle A, B \rangle$ of the derivative map to $\SO(2)$, and the symmetry group of the tiling contains $\langle A,B\rangle$.

Let $h:\langle \psi, \iota\rangle \to \langle A, B\rangle$ be the map sending $\phi \in \langle \psi, \iota\rangle$ to its restriction $\phi|_{\bone^\perp}$. This map is clearly surjective and is injective as a consequence of \Cref{isometries}. Let $K'=h^{-1}(K)$. Again using the fact that $\proj|_{\bN}$ is a homeomorphism,
we see that $\{\phi(R):\phi \in K'\}$ is a collection of closed disks tiling $\bN$. This collection is the same as either $\{D_\mathit{out}(C'):~C' \in {\mathcal C}\}$
or 
$\{D_\mathit{in}(C'):~C' \in {\mathcal C}\}$
where
${\mathcal C}=\{\phi(C):~\phi \in K'\},$
depending on whether $\gamma$ was the inner or outer boundary component of $C$.
This proves the tiling statement, and since $h$ is a group isomorphism, the symmetry group of the tiling contains $\langle \psi, \iota\rangle$ which is isomorphic to $\langle A, B\rangle$ and hence also ${\mathbf p}{\mathbf 6}$.

It remains to prove that if the slope of the core curves of $C$ is not $\pm 1$, then the group $\langle \psi, \iota\rangle$ is the full symmetry group of the tiling $\{\phi(R):\phi \in K'\}$. From the hypothesis on slope, we obtain that the saddle connection $\gamma([0,\frac{1}{6}])$ cannot be contained in a single square. From the symmetry (2) and the fact that $\gamma(\frac{1}{12})$ is the center of a square, we see that $\gamma\big((0,\frac{1}{6})\big)$ must cross an even number $n \geq 2$ of interiors of edges of squares. So $\gamma([0,1])$ must cross at least $6n \geq 12$ interiors of edges. Using \Cref{shear}, we see that $\proj(\gamma)$ has at least $6n$ nontrivial angles and so cannot be a regular hexagon. Therefore \Cref{632 detector} guarantees that the symmetry group of the tiling of $\bone^\perp$ by translations of $\proj(R)$ is precisely $\langle A,B\rangle$. It then follows from another application of \Cref{isometries} that the full subgroup of $\Isom(\RR^3,\bN)$ that preserve the tiling $\{\phi(R):\phi \in K'\}$ of $\bN$ is precisely $\langle \psi, \iota\rangle$.
\end{proof}

\section{Affine symmetries of \texorpdfstring{$\bNt$}{the half-translation cover of the Necker cube surface}}
\label{sect:affine group}

In this section, we gain an understanding of the Veech group $V(\bNt)$, consisting of elements of $\PSL(2,\RR)$
that stabilize $\bNt$. In the first subsection, we list elements known to be in $V(\bNt)$. In \Cref{sect: complete veech}, we explain that these known elements generate the Veech group and describe the quotient $\HH/V(\bNt)$, which is significant because its tangent space represents the orbit of $\bNt$ under $\PSL(2,\RR)$ equipped with its natural Teichm\"uller metric. The remainder of this section is devoted to proving this result and \Cref{thm:symmetries}.

\subsection{Elements of the Veech group}

\begin{prop}
\label{rotation by 90}
The unique nontrivial deck transformation of the cover $\bNt \to \bN$ is a quarter-translation automorphism of $\bNt$ with derivative 
$\pm \begin{pmatrix}
0 & -1 \\
1 & 0
\end{pmatrix}$. In particular, this element of $\PSL(2,\RR)$ is in the Veech group of $\bNt$.
\end{prop}
\begin{proof}
This is a direct consequence of \Cref{deck group and derivative}.
\end{proof}

We have the following consequence of our theorem about closed geodesics.

\begin{cor}
\label{cor:odd twist}
For each $(p,q) \in \cO$, there is an affine automorphism $\psi_{p,q}:\bNt \to \bNt$ that performs a single right Dehn twist in each maximal cylinder with slope $\frac{q}{p}$. The derivative of $\psi_{p,q}$ is
$$
P_{p,q}=\pm \begin{pmatrix}
1-6pq & 6p^2 \\
-6q^2 & 1+6pq
\end{pmatrix} \in \GL(2,\ZZ)/\pm 1.
$$
\end{cor}
\begin{proof}
\Cref{thm:odd-odd is closed} guarantees that for each $(p,q) \in \cO$, there is a cylinder decomposition of $\bNt$ such that each cylinder has circumference $6 \sqrt{p^2+q^2}$ and area six. The inverse modulus of each cylinder is therefore $6p^2+6q^2$. From comments in \Cref{sect:cylinders}, we know there is an affine automorphism $\psi_{p,q}:\bNt \to \bNt$ that performs a single right Dehn twist in every cylinder of this slope and its derivative is the image of $R \circ Q \circ R^{-1}$ in $\PSL(2,\RR)$ where
$$R=\frac{1}{\sqrt{p^2+q^2}}\begin{pmatrix}
p & -q \\
q & p
\end{pmatrix} 
\quad \text{and} \quad
Q = \begin{pmatrix}
1 & 6p^2+6q^2 \\
0 & 1
\end{pmatrix}.$$
A computation shows $R \circ Q \circ R^{-1}$ is as stated.
\end{proof}

\subsection{The complete Veech group}
\label{sect: complete veech}

\begin{thm}
\label{Veech group}
The matrices listed in \Cref{rotation by 90} and \Cref{cor:odd twist} generate $V(\bNt)$.
The quotient $\HH/V(\bNt)$ is homeomorphic to an open topological disk with a countable closed discrete subset of points removed.
The quotient $\HH/V(\bNt)$ has exactly one cone singularity, and the cone angle at this point is $\pi$.
\end{thm}

We remark that \Cref{Veech group} determines the topological type of $\HH/V(\bNt)$; it is a {\em flute surface}, a surface of genus zero whose end space is homeomorphic to $\NN \cup \{+\infty\}$. The argument we use allows for a precise description of the geometry of the quotient as well, in terms of a decomposition into ideal triangles, though we will not state this precisely in a theorem.

Recall that a subgroup of $\PSL(2,\RR)$ is called {\em Fuchsian of the first-kind} if its M\"obius action on $\HH$ has a limit set that is dense in the boundary $\partial \HH$.

\begin{cor}
The Veech group $V(\bNt)$ is an infinitely generated Fuchsian of the first-kind
\end{cor}
\begin{proof}
The Veech group is isomorphic to the (orbifold) fundamental group of $\HH/V(\bNt)$, and so is infinitely generated because the quotient has infinitely many punctures. The M\"obius action of $P_{p,q}$ preserves $\frac{p}{q}$ and so $\frac{p}{q}$ lies in the limit set.
Since this statement holds for all $(p,q) \in \cO$ and the limit set is closed, the limit set is all of $\partial \HH$.
\end{proof}

The theorem above compliments \Cref{thm:symmetries} from the introduction, which makes a statement about the affine group of $\bNt$.

\subsection{Affine automorphisms descend}

To prove \Cref{Veech group} and \Cref{thm:symmetries}, we first explain that $V(\bNt)$ is the group of derivatives of elements of $\Aff_+(\bHt)$ that lift to $\bNt$. To show this, we will first argue that the quotient $\bHt$ of $\bNt$ is affinely natural.

\begin{prop}
\label{isomorphic automorphisms}
Every quarter-translation automorphism of $\bN$ has a unique lift that is a half-translation automorphism of $\bNt$, and the induced map $\Aut_{G_4}(\bN) \to \Aut_{G_2}(\bNt)$ is a group isomorphism.
\end{prop}
\begin{proof}
The fact that this lifting map $L:\Aut_{G_4}(\bN) \to \Aut_{G_2}(\bNt)$ is a homomorphism is \Cref{commute-branch-cover}. This map is injective because the original map can be recovered from the lift. Assume to the contrary that the map is not surjective. Then there is a $G_2$-automorphism $\psi:\bNt \to \bNt$ that does not descend to a well-defined map on $\bN$. Let $S \in \bN$ be any square, and let $\tilde S_0 \cup \tilde S_1$ be its preimage under $\bNt \to \bN$, which consists of two squares. By \Cref{isometries and squares}, the isometry group of $\bN$ acts transitively on squares, so by postcomposing $\psi$ with an element of $L\big(\Aut_{G_4}(\bN)\big)$, we can assume without loss of generality that $\psi(\tilde S_0) \in \{\tilde S_0, \tilde S_1\}$. Let $\tilde \delta:\bNt \to \bNt$ denote the unique nontrivial deck transformation of \Cref{rotation by 90}. Then, there is an $i \in \{0, 1\}$ such that $\tilde \delta^i \circ \psi$ fixes $\tilde S_0$. Since $\tilde \delta^i \circ \psi$ is not in $L\big(\Aut_{G_4}(\bN)\big)$ or in its image under the deck group, $\psi$ must rotate $\tilde S_0$ by either $90^\circ$ or $270^\circ$. By possibly replacing $\tilde \delta^i \circ \psi$ with its third power if necessary, we have produced a $G_4$-automorphism that rotates $\tilde S_0$ by $90^\circ$.

To obtain a contradiction, we can therefore show that the isometry $j$ of $\tilde S_0 \subset \bNt$ that rotates it by $90^\circ$ does not extend to an isometry $\bar j:\bNt \to \bNt$. Let $\tilde v_0, \tilde v_2$ be the two opposite vertices of $\tilde S_0$ that map to cone singularities with cone angle $\frac{3 \pi}{2}$ under $\bNt \to \bN$. Then moving three squares counterclockwise from $\tilde S_0$ around either $\tilde v_0$ or $\tilde v_2$ brings us to $\tilde S_1=\tilde \delta(\tilde S_0)$. If $j$ extends, then $\bar j(\tilde S_1)$ must be obtained by moving three squares counterclockwise around the other two vertices $v_1$ and $v_3$ of $\tilde S_0$. But, because these other two vertices descend in $\bN$ to cone singularities with cone angle $3 \pi$, the two squares obtained in this way differ even in their image in $\bN$. This is our contradiction.
\end{proof}

\begin{cor}
\label{cor:descends}
Any element $\tilde f \in \Aff(\bNt)$ descends to a well-defined $f \in \Aff(\bHt)$.
\end{cor}
\begin{proof}
A map $\tilde f$ descends if and only if it normalizes the deck group $\tilde \Delta$ of $\bNt \to \bHt$. Note that $\tilde \Delta \subset \Aut_{G_2}(\bNt)$. Let $L:\Aut_{G_4}(\bN) \to \Aut_{G_2}(\bNt)$ be the isomorphism of \Cref{isomorphic automorphisms}, and let $\Delta$ denote the deck group of $\bN \to \bH$. By \Cref{commute-branch-cover}, $\tilde \Delta=L(\Delta)$. Now observe from \Cref{isometries} that $\Delta$ is the subgroup of $\Aut_{G_4}(\bN)$ consisting of the identity and all infinite-order elements of $\Aut_{G_4}(\bN)$. 
Since $L$ is an isomorphism, $\tilde \Delta$ consists of the identity and all infinite-order elements of $\Aut_{G_2}(\bNt)$. It follows that $\tilde \Delta$ is a characteristic subgroup of $\Aut_{G_2}(\bNt)$. Observe that $\tilde f$ acts by conjugation on $\Aut_{G_2}(\bNt)$, inducing an automorphism of $\Aut_{G_2}(\bNt)$. Since $\tilde \Delta$ is characteristic, it must be preserved by this action of $\tilde f$, so $\tilde f$ descends as desired.
\end{proof}

\subsection{Lifting affine automorphisms}

Since every element of $\Aff(\bNt)$ descends to $\Aff(\bHt)$, $V(\bNt)$ is the collection of all derivatives of elements of $\Aff(\bHt)$ that lift. The action $\authom:F_2 \to \Aff(\bHt)$ induces an action $\authom_\ast:F_2 \to \Aut\big(H_1(\bHt; \QQ)\big)$. 
By \Cref{cor:lifting}, $\authom(g)$ lifts if and only if $\authom_\ast(g)(W_+)=W_+$. \Cref{W stabilization} below then gives a criterion for when $\authom(g)$ lifts.

The action $\authom_\ast$ turns out to be closely related to the representation $\rho:F_2 \to \SL(2,\ZZ)$ given by 
$$\rho(h) = \begin{pmatrix}
1 & 3 \\
0 & 1 \\
\end{pmatrix} 
\quad \text{and} \quad
\rho(v) = \begin{pmatrix}
1 & 0 \\
 1 & 1 \\
\end{pmatrix}.
$$
The following two lemmas explain the relationship between $\authom$ and $\rho$, and exploit this relationship to understand how the action $\authom_\ast$ moves the subspace $W_+$.

\begin{lemma}
\label{rho conjugation}
The action $\authom_\ast$ is conjugate via a linear map $H_1(\bHt; \QQ) \to \QQ^4$ to the direct sum representation $\rho \oplus \rho$.
\end{lemma}

The action $\authom$ of $F_2$ on $\bHt$ does not give all automorphisms of $\bHt$. In particular, the $G_4$-automorphism group of $\bHt$ also acts, and some of these elements interchange $W_+$ with $W_-$, see \Cref{G4 auto lift}. So, we also want to rule out the possibility that $\authom_\ast(g)(W_+)=W_-$.

\begin{lemma}
\label{W stabilization}
Let $g \in F_2$. The following statements are equivalent:
\begin{enumerate}
\item $\authom_\ast(g)(W_+) \in \{W_+, W_-\}$.
\item $\rho(g)=I$.
\item $\authom_\ast(g)$ acts trivially on $H_1(\bHt; \QQ)$.
\end{enumerate}
\end{lemma}

We devote the remainder of this subsection to the proofs of these lemmas.
To begin see why \Cref{rho conjugation} holds, consider the six subspaces 
$$U_i=\Span \{\gamma_i, \gamma_{i+1}-\gamma_{i-1}\} \subset H_1(\bH; \QQ) \quad \text{for $i \in \ZZ/6\ZZ$.}$$

\begin{prop}
\label{U invariant}
Each $U_i$ is invariant under the action of every element $\authom_\ast(g)$ for $g \in F_2$. Furthermore, if $i$ is odd, then 
the matrix representation associated to the restriction $\authom_\ast |_{U_i}: U_i \to U_i$ in the basis $\{\gamma_i, \gamma_{i+1}-\gamma_{i-1}\}$ is given by $\rho$.
\end{prop}
\begin{proof}
Let $i$ be odd.
Recall that $\authom(h)$ performs a right Dehn twist in cylinders whose core curves are in $\{\gamma_1, \gamma_3, \gamma_5\}$.
Thus $\authom_\ast(h)(\gamma_i)=\gamma_i$. Using \eqref{eq:Dehn action} and intersection numbers from \eqref{eq:gamma intersect}, we see
\begin{align*}
\authom_\ast(h)(\gamma_{i+1}-\gamma_{i-1}) &=
\authom_\ast(h)(\gamma_{i+1})- \authom_\ast(h)(\gamma_{i-1}) \\
&= (\gamma_{i+1} - \gamma_{i+2} + \gamma_i) - (\gamma_{i-1} - \gamma_{i} + \gamma_{i-2}) \\
&= 3 \gamma_i + (\gamma_{i+1}-\gamma_{i-1}),
\end{align*}
with the last equality using the fact that $\gamma_{i-2}+\gamma_i+\gamma_{i+2}=0$.
Similarly, $\authom(v)$ performs a left twist in $\{\gamma_0, \gamma_2, \gamma_4\}$.
Since $\gamma_{i+1}-\gamma_{i-1}$ is a sum of vertical curves, it is fixed by $\authom_\ast(v)$. We have
$$\authom_\ast(v)(\gamma_{i}) = \gamma_i + \gamma_{i+1} - \gamma_{i-1}.$$
This gives the matrix description, and proves the invariance in the odd case. 
If $j$ is even, then an automorphism $\varphi$ of $\bHt$ whose derivative is a $90^\circ$ rotation (which exists by \Cref{G4 auto}) will satisfy $\varphi_\ast(U_j)=U_i$ for some $i$ odd. 
Note that $\authom(h)$ is a right Dehn twist, while $\authom(v)$ is a left Dehn twist. Therefore, conjugation by $\varphi_\ast$ swaps $\authom(h)$ with $\authom(v^{-1})$. Letting $\iota:F_2 \to F_2$ denote the automorphism such that $\iota(h)=v^{-1}$ and $\iota(v)=h^{-1}$, we have $\varphi_\ast \circ \authom_\ast(g)=\authom_\ast(\iota(g)) \circ \varphi_\ast$ for all $g \in F_2$ so the action $\authom_\ast$ must preserve $U_j$ as well.
\end{proof}

\begin{prop}
If $i$ is odd, then $\Span (U_i \cup U_{i+2})=H_1(\bHt; \QQ)$.
\end{prop}
\begin{proof}
Recall from \Cref{sect:cover of H} that $(\gamma_0, \gamma_1; \gamma_3, \gamma_4)$ is a basis for $H_1(\bHt; \QQ)$.
Consider $\Span (U_1 \cup U_{3})$. Writing in the standard basis using the relations in \eqref{eq:gamma id}, we have
$$U_1=\Span \{\gamma_1, - 2\gamma_0 - \gamma_4\} \quad \text{and} \quad 
U_3=\Span \{\gamma_3, 2 \gamma_4-\gamma_0\}.$$
A simple linear algebra calculation checks that the four elements listed above are linearly independent, proving that the span has dimension four, equal to $\dim H_1(\bHt; \QQ)$. This completes the proof in the case $i=1$.
For any other odd $i$, there is a half-translation automorphism (obtained via \Cref{G4 auto}) that sends $U_i$ to $U_1$ and $U_{i+2}$ to $U_3$, so the conclusion holds here too.
\end{proof}

Since $U_1$ and $U_3$ are two dimensional subspaces that span $H_1(\bHt; \QQ)$, a union of bases of these two subspaces gives a basis
for $H_1(\bHt; \QQ)$. We choose bases for $U_1$ and $U_3$ from \Cref{U invariant}, obtaining the basis
$\cB = \{\gamma_1, \gamma_2-\gamma_0, \gamma_3, \gamma_4-\gamma_2\}$.

\begin{proof}[Proof of \Cref{rho conjugation}]
By \Cref{U invariant}, the action $\authom_\ast$ in the basis $\cB$ has the matrix representation $\rho \oplus \rho$.
\end{proof}

We will now explain how we understand the subspaces $W_+$ and $W_-$ in the basis $\cB$. Given $w \in H_1(\bHt; \QQ)$, write $\bw_\cB \in \QQ^4$ to denote the vector representation of $w$ in the basis $\cB$. Write $\bw_1 \in \QQ^2$ for the first two entries of $\bw_\cB$,
and write $\bw_3 \in \QQ^2$ for the last two entries. (Then $\bw_1$ represents the projection of $\bw_\cB$ to $U_1$ along $U_3$, and vice versa.)

\begin{prop}
\label{W is transverse}
Let $s \in \{\pm 1\}$.
For any $w \in W_s$, we have $\bw_3=N^s \bw_1$ where $N=\frac{1}{2}\begin{pmatrix} 1 & -3 \\ 1 & 1 \end{pmatrix}$ and $N^{-1}=\frac{1}{2}\begin{pmatrix} 1 & 3 \\ -1 & 1 \end{pmatrix}$.
\end{prop}
\begin{proof}
By linearity, it suffices to check the statement on a basis for $W_s$. First consider $w^1=\gamma_3 + s \gamma_0 \in W_s$. It can be checked that
$$\bw^1_\cB=\begin{pmatrix}
0 \\ \frac{-2s}{3} \\ 1 \\ \frac{-s}{3}
\end{pmatrix}
\quad \text{and} \quad
N^s \begin{pmatrix}
0 \\ \frac{-2s}{3}
\end{pmatrix} = 
\begin{pmatrix}
1 \\ \frac{-s}{3}
\end{pmatrix}.$$
Similarly, when $w^2=\gamma_1+s\gamma_4 \in W_s$, we have 
$$\bw^2_\cB=\begin{pmatrix}
1 \\ \frac{s}{3} \\ 0 \\ \frac{2s}{3}
\end{pmatrix}
\quad \text{and} \quad
N^s \begin{pmatrix}
1 \\ \frac{s}{3}
\end{pmatrix} = 
\begin{pmatrix}
0 \\ \frac{2s}{3}
\end{pmatrix}.$$
\end{proof}

To prove \Cref{W stabilization}, we need one more ingredient: a basic understanding of the representation $\rho$.
To gain this, recall that matrices in $\SL(2,\RR)$ act on the upper half-plane via M\"obius transformations according to the formula that
$$\begin{pmatrix}
a & b \\
c & d
\end{pmatrix}(z)=\frac{az+b}{cz+d}.$$

Recall that the {\em Farey triangulation} of $\HH^2$ is the ideal triangulation whose vertices are the points in $\hat \QQ=\QQ \cup \{\infty\} \subset \partial \HH^2$ where there is an edge from $\frac{a}{b}$ to $\frac{c}{d}$ (expressed as ratios of relatively prime integers) if and only if $ad-bc \in \{\pm 1\}$. This triangulation is depicted in \Cref{fig:Farey}. It is well known that $\PSL(2,\ZZ)$ is the orientation-preserving automorphism group of this triangulated space. The image of $\Gamma_2$ in $\PSL(2,\ZZ)$ is the group of orientation-preserving automorphisms that permute the equivalence classes of relatively prime fractions taken modulo two. Thus, the quotient $\HH/\Gamma_2$ is a hyperbolic three punctured sphere.

\begin{figure}[thb]
\labellist
\small\hair 2pt
 \pinlabel {$-1$} [ ] at 280 -50
 \pinlabel {$\frac{-1}{2}$} [ ] at 890 -70
 \pinlabel {$0$} [ ] at 1500 -50
 \pinlabel {$\frac{1}{2}$} [ ] at 2110 -70
 \pinlabel {$1$} [ ] at 2700 -50
\endlabellist
\centering
	\includegraphics[width=\textwidth]{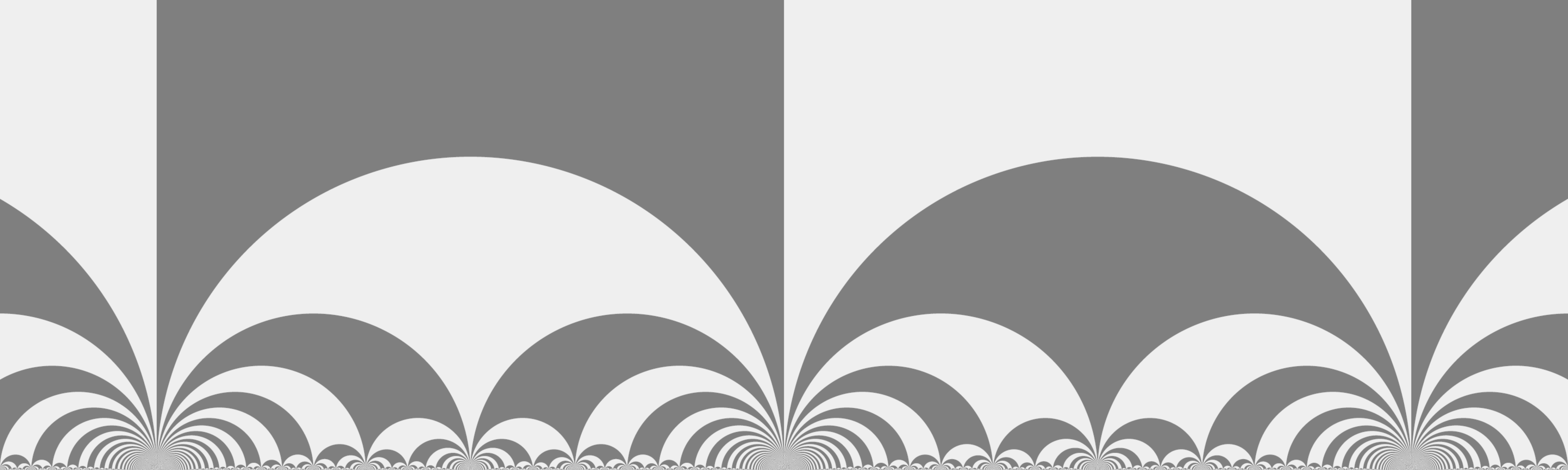}
	\vspace{-0.5em}
	\caption{A portion of the Farey triangulation of the upper half-plane.}
	\label{fig:Farey}
\end{figure}
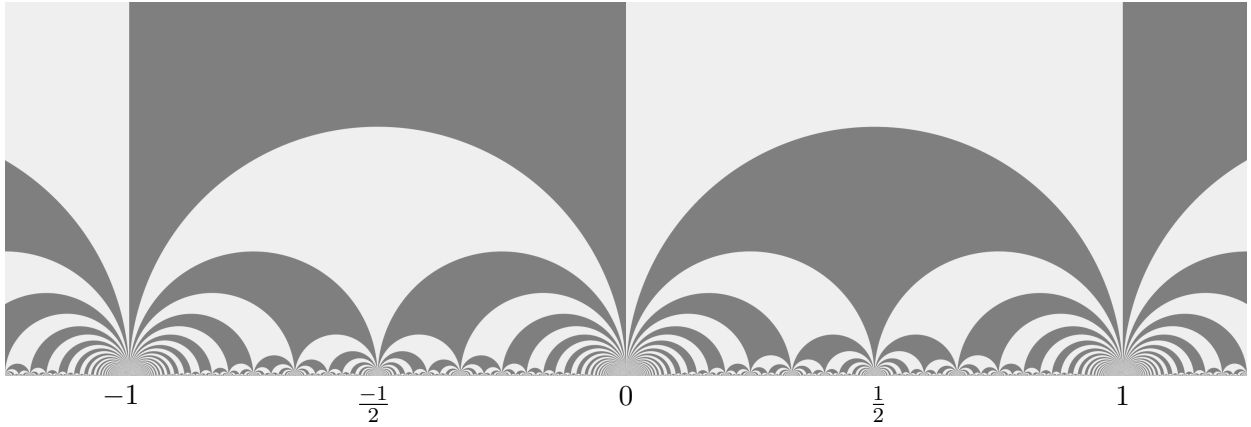

\begin{figure}[bht]
\labellist
\small\hair 2pt
 \pinlabel {$-2$} [ ] at 146 -50
 \pinlabel {$\frac{-3}{2}$} [ ] at 500 -70
 \pinlabel {$-1$} [ ] at 813 -50
 \pinlabel {$0$} [ ] at 1500 -50
 \pinlabel {$1$} [ ] at 2166 -50
 \pinlabel {$\frac{3}{2}$} [ ] at 2500 -70
 \pinlabel {$2$} [ ] at 2833 -50
\endlabellist
\centering
	\includegraphics[width=\textwidth]{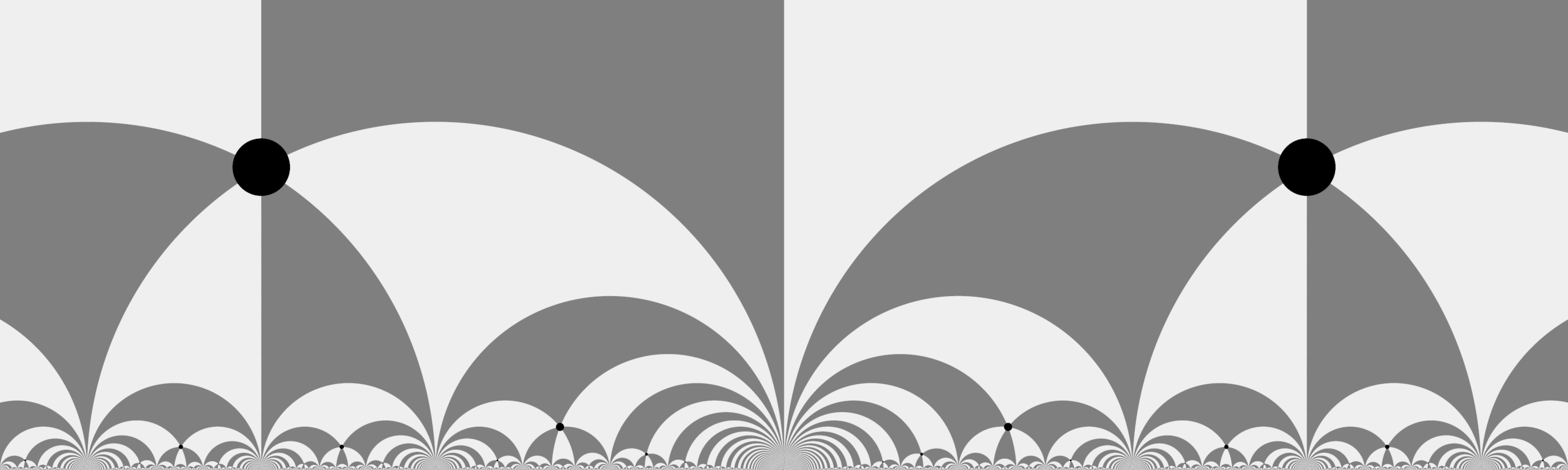}
	\vspace{-0.5em}
	\caption{The triangulation of the upper half-plane associated to the $(3, \infty, \infty)$ triangle group, $\rho(\Gamma_2)$, 
	punctured at the interior vertices.}
	\label{fig:veech}
\end{figure}

We can similarly understand the action of $\rho(F_2)$:

\begin{prop}
\label{rho action}
The action of $\rho$ on $\HH$ via M\"obius transformations is a $(3,\infty, \infty)$ triangle group. That is, $\HH/\rho(F_2)$ is a sphere with two punctures and one cone singularity with cone angle $\frac{2 \pi}{3}$. 
The kernel of this map to the M\"obius group is the smallest normal subgroup of $F_2$ containing $(vh^{-1})^3$. 
There is no $g \in F_2$ such that $\rho(g)=-I$.
\end{prop}
\begin{proof}
Under the M\"obius action, $\rho(h):z \mapsto z+3$ and $\rho(v):z \mapsto \frac{z}{1+z}$. Thus a fundamental domain is given by 
$$\{z:~\text{$\frac{-3}{2} \leq \Re z \leq \frac{3}{2}$ and $|z\pm 1| \geq 1$ for each sign}\},$$
the union of the two triangles in the top center of \Cref{fig:veech}. The quotient $Q = \HH/\rho(F_2)$ then is a hyperbolic cone structure on the twice punctured sphere with two punctures and a cone singularity with cone angle $\frac{2 \pi}{3}$.
The two punctures are images of neighborhoods of $\infty$ (which is stabilized by the M\"obius action of $h$) and $0$ (stabilized by $v$).
The M\"obius action of $\rho(v h^{-1})$ rotates by $\frac{2 \pi}{3}$ about the point $\frac{3 + i\sqrt{3}}{2}$, so the M\"obius action of $(vh^{-1})^3$ is trivial. The action on the fundamental domain gives a presentation for the M\"obius action, namely 
$\langle h, v|~(vh^{-1})^3\rangle$. This is a standard presentation for the $(3, \infty,\infty)$ triangle group, and geometrically the fundamental domain can be cut into two triangles as in \Cref{fig:veech} that tile the plane under the M\"obius action. Our presentation for the M\"obius action also tells us that the kernel is the smallest normal subgroup containing $(vh^{-1})^3$.

To see the last statement, suppose to the contrary that $\rho(g)=-I$. Note that $-I$ acts trivially on $\HH$, so $g$ lies in the kernel defined above. From work above, we know $g$ can be written as a product of conjugates of $(vh^{-1})^3$ and its inverse.
But then because $\rho\big(vh^{-1})^3\big)=I$, we must also have $\rho(g)=I$, contradicting our supposition.
\end{proof}

\begin{proof}[Proof of \Cref{W stabilization}]
That (2) and (3) are equivalent follows directly from \Cref{rho conjugation}. Also, certainly (3) implies (1), so it suffices to prove (1) implies (2). 

To prove this statement, we use \Cref{W is transverse} to describe $\authom_\ast(g)(W_+)$ in terms of $\rho(g)$. 
We take the point of view that \Cref{W is transverse} explains that $W_+$ written in the basis $\cB$ is the graph of $N:\QQ^2 \to \QQ^2$. Here the domain is representing our basis coordinates on $U_1$ and the codomain represents coordinates on $U_3$. By \Cref{U invariant}, we see that if $w \in W_+$ and $x=\authom_\ast(g)(w)$ then in the notation of \Cref{W is transverse}
we have $\bx_1=\rho(g)(\bw_1)$ and $\bx_3=\rho(g)(\bw_3)$. By \Cref{W is transverse}, we have $\bw_3=N \bw_1$. It follows that
$\bx_3 = N_g \bx_1$ where 
$$N_g = \rho(g) \cdot N \cdot \rho(g^{-1}).$$
Thus, the vector representation in the basis $\cB$ of the subspace $\authom_\ast(g)(W_+)$ is given by 
$$\{(\bx_1, \bx_3) \in \QQ^4:~\text{$\bx_1, \bx_3 \in \QQ^2$ and $\bx_3 = N_g \bx_1$}\}.$$
This expresses $\authom_\ast(g)(W_+)$ in a similar manner to $W_+$: as a graph of $N_g$. We conclude that $\authom_\ast(g)(W_+)=W_s$ if and only if $N_g=N^s$.

From our definition of $N_g$, $\authom_\ast(g)(W_+)=W_s$ if and only if $\rho(g)$ preserves the set $\{\bv_\pm\}$ of eigenvectors
of $N$. (If $\rho(g)$ preserves the eigenvectors then $\authom_\ast(g)$ preserves $W_+$ and if it swaps them, then it sends $W_+$ to $W_-$.) The eigenvectors are $\bv_\pm = (1, \pm \frac{i \sqrt{3}}{3})$.

Now we can show that (1) implies (2). Suppose $\authom_\ast(g)(W_+)=W_s$.
From the above, we must have $\rho(g)(\bv_+)=\lambda \bv_s$ for some $\lambda \in \CC$. Then we also have $\rho(g)(\bv_-)=\bar \lambda \bv_{-s}$ since complex conjugation swaps $\bv_+$ with $\bv_-$. Noting that $(1,0)=\frac{1}{2}(\bv_+ + \bv_-)$ and $(0,1)=\frac{-i\sqrt{3}}{2}(\bv_+ - \bv_-)$, we have
$$\rho(g)\begin{pmatrix} 1 \\ 0 \end{pmatrix} =  \frac{1}{2}(\lambda \bv_s + \bar \lambda \bv_{-s})=\Re(\lambda \bv_s)
\quad \text{and} \quad
\rho(g)\begin{pmatrix} 0 \\ 1 \end{pmatrix} =  \frac{-i\sqrt{3}}{2}(\lambda \bv_s - \bar \lambda \bv_{-s})=\sqrt{3} \Im(\lambda \bv_s).$$
It follows that
$$\rho(g)=\begin{pmatrix}
a & 3b \\
-sb & sa
\end{pmatrix}
\quad \text{where} \quad a=\Re \lambda \quad \text{and} \quad b = \frac{\sqrt{3}}{3} \Im \lambda.$$
Since $\rho$ is a representation whose image is contained in $\SL(2,\ZZ)$, we know $a,b \in \ZZ$. Also the determinant must equal one,
so $s(a^2 +3b^2)=1$. But the only integer solutions to this equation are $a= \pm 1$, $b=0$ and $s=1$. By \Cref{rho action}, there is no $g \in F_2$ such that $\rho(g)=-I$, so we must have $\rho(g)=I$. Thus (2) holds.
\end{proof}

\subsection{Geometry of the Teichm\"uller disk}

Recall the group isomorphism $M:F_2 \to \Gamma_2$ from \eqref{eq:matrixhom}.
We define 
$$\Gamma_\cO=\langle P_{p,q}:~(p,q) \in \cO\rangle \subset \Gamma_2$$
where the matrices $P_{p,q}$ are taken from \Cref{cor:odd twist}.  By this corollary $\Gamma_\cO \subset V(\bNt)$.
Set $F_\cO = M^{-1}(\Gamma_\cO) \subset F_2$.

\begin{lemma}
\label{odd kernel}
We have that $F_\cO=\ker \rho$. The surface $\HH/\Gamma_\cO$ is a topological disk with countably many punctures located at the images of $\frac{p}{q}$ with $(p,q) \in \cO$. As $\Gamma_\cO \subset \SL(2,\ZZ)$, the $\Gamma_\cO$ action respects the Farey triangulation and the image triangles come together six to a vertex at images of $\frac{p}{q}$ for $(p,q) \in \cO$ and infinitely many triangles come together at the other vertices which all lie in the boundary of the disk.
\end{lemma}

The proof shows that the triangulation of $\HH/\Gamma_\cO$ is combinatorially the same as the $(3,\infty,\infty)$ triangulation stabilized by $\rho$ and depicted in \Cref{fig:veech}, but the geometry is different: In $\HH/\Gamma_\cO$ all the triangles are ideal triangles.

\begin{proof} \notepat{I'm frustrated by the double use of $M$ in the second sentence.}
Recall the identification $M:F_2 \to \Gamma_2$ from \eqref{eq:matrixhom}. By a computation, we can check that $P_{1,1}$ of \Cref{cor:odd twist} equals $M\big((v h^{-1})^3\big)$. Thus we have $M^{-1}(P_{1,1}) \in \ker \rho$. Recall that the image of $\Gamma_2$ in $\PGL(2,\RR)$ acts transitively on ratios of odd integers. The matrix $P_{p,q}$ acts on the Farey triangulation, fixing $\frac{p}{q}$ and shifting triangles with vertex $\frac{p}{q}$ six triangles counterclockwise. It follows that for every $(p,q) \in \cO$, we have $P_{p,q}=Q P_{1,1} Q^{-1}$ for some $Q \in \Gamma_2$. Since $\ker \rho$ is normal, it follows that $M^{-1}(P_{p,q}) \in \ker \rho$ for all $(p,q) \in \cO$.
Thus $F_\cO \subset \ker \rho$. Conversely, every conjugate $Q P_{1,1} Q^{-1}$ with $Q \in \Gamma_2$ equals $P_{p,q}$ where $\frac{p}{q}$ is the image of $1$ under the M\"obius action of $Q$. It also follows from this that $F_\cO$ is the smallest normal subgroup of $F_2$ containing $M^{-1}(P_{1,1})=(v h^{-1})^3$. From \Cref{rho action}, we see that $F_\cO=\ker \rho$. 

To see the remaining statements, let $P$ denote the collection of interior vertices of the $(3,\infty,\infty)$ triangle preserved
by $\rho$ (and depicted in \Cref{fig:veech}). There is a unique orientation-preserving covering map $f:\HH \to \HH \setminus P$ that respects the triangulations and sends $0 \mapsto 0$ and $\infty \mapsto \infty$. We can extend this covering map to a sequence of covers
$$\HH \xrightarrow{f} \HH \setminus P \to (\HH \setminus P)/\rho(F_2),$$
where the last space is a three punctured sphere whose fundamental group is naturally identified with $F_2=\langle h,v\rangle$.
The fundamental group of $\HH \setminus P$ is then naturally isomorphic to $F_\sigma$, and the deck group of the cover $f$ is $\Gamma_\cO$. It follows that $f$ descends to a triangle-respecting homeomorphism from $\HH/\Gamma_\cO$ to $\HH \setminus P$ as desired.
\end{proof}

\begin{proof}[Proof of \Cref{Veech group}]
Let $R=\pm \begin{pmatrix}
0 & 1 \\
-1 & 0
\end{pmatrix} \in \PSL(2, \RR)$ be the element mentioned in \Cref{rotation by 90}.
The matrices listed in \Cref{cor:odd twist} generate $\Gamma_\cO$. 
We first need to show that $V(\bNt)$ is generated by $\{R\} \cup \Gamma_\cO$. 

\compat{WORKING HERE... Check that new map names are consistent.}
Let $\tilde \psi_0 \in \Aff(\bNt)$. We must show that $D \tilde \psi_0$ lies in the group generated by $R$ and $\Gamma_\cO$. 
By \Cref{SL2Z}, $D \tilde \psi_0 \in \PSL(2,\ZZ)$. Therefore, $D \tilde \psi_0$ must permute the three collections of relatively prime vectors $\cO$, $\cE_1$ and $\cE_2$. By \Cref{thm:main}, $\cO$ must be fixed, therefore if this permutation action is nontrivial, it must swap $\cE_1$ with $\cE_2$. \Cref{rotation by 90} then guarantees that there is a deck transformation $\tilde \delta:\bNt \to \bNt$ of the covering $\bNt \to \bN$ (possibly trivial) such that the derivative of $\tilde \psi_1=\tilde \delta \circ \tilde \psi_0$ preserves each of the three classes $\cE_1$, $\cE_2$ and $\cO$. Since the action of $\tilde \psi_1$ preserves these classes, its derivative lies in the congruence two subgroup of $\PSL(2,\ZZ)$, i.e., there is a $g \in F_2$ such that $D \tilde \psi_1=M(g)$ where $M$ is defined as in \eqref{eq:matrixhom}.

By \Cref{cor:descends}, $\tilde \psi_1$ descends to an automorphism $\psi_1:\bHt \to \bHt$ also with derivative $M(g)$. Then $\psi_1=\authom(g)\circ \iota $ where $\iota:\bHt \to \bHt$ is a half-translation automorphism. By \Cref{G4 auto lift}, we have $\iota_\ast(W_+)=W_s$ for some sign $s \in \{\pm 1\}$. Since $\psi_1$ lifts, we must have that $\psi_1(W_+)=W_+$. Therefore $\authom_\ast(g)(W_s)=W_+$. But then $\authom_\ast(g^{-1})(W_+)=W_s$. \Cref{W stabilization} then tells us that $\rho(g^{-1})=I$. Thus, $g \in \ker \rho$,
and via \Cref{odd kernel} we conclude that $g \in F_\cO$. Then 
$$D \tilde \psi_1 = D \psi_1 = D \authom(g)=M(g) \in \Gamma_\cO.$$
We have $\tilde \psi_0=\tilde \delta^{-1} \circ \tilde \psi_1$, and $D \tilde \delta$ is either the identity or $R$; see \Cref{rotation by 90}. We conclude that $D \tilde \psi_0 = (D \tilde \delta)^{-1} \cdot (D \tilde \psi_1)$ is in the group generated by $R$ and $\Gamma_\cO$ as desired.

The remainder of the statement concerns the geometry of $\HH/V(\bNt)$. Observe that $R$ preserves $\cO$. It follows that the M\"obius action of $R$ preserves the vertices of the Farey triangulation consisting of ratios of odd integers. From the geometric description of how elements $P_{p,q}$ from \Cref{cor:odd twist} act on the Farey triangulation, we conclude that $R \Gamma_\cO R^{-1}=\Gamma_\cO$.
Therefore, there is a covering map $\HH / \Gamma_\cO \to \HH / \langle \Gamma_\cO \cup \{R\}\rangle$ of degree two (equal to the order of $R$). From \Cref{odd kernel}, we saw that $\HH/\Gamma_\cO$ can be combinatorially identified with $H\setminus P$ where $P$ represents the interior vertices of the $(3, \infty, \infty)$ triangulation. The M\"obius action of $R$ swaps $0$ and $\infty$ and acts as an orientation-preserving automorphism of the Farey triangulation. Thus, this action descends to an order two orientation-preserving automorphism of $\HH \setminus P$ that preserves an edge joining two vertices of infinite valence. This introduces a cone point with cone angle $\pi$ at the fixed point on the edge. The quotient of a disk punctured at a discrete countable set by an involution is still a disk with a discrete countable set of punctures.
\end{proof}

\begin{proof}[Proof of \Cref{thm:symmetries}]
Let $\psi \in \Aff(\bNt)$. Here we need to argue that $\psi$ is the composition of sequence of affine multitwists preserving cylinder decompositions of odd over odd slope and an isometry. \Cref{Veech group} guarantees that $D\psi=R^i \cdot Q$ where $i \in \{0,1\}$,
$R$ is as in the previous proof, and $Q \in \Gamma_\cO$. Since $Q \in \Gamma_\cO$, there is a concatenation of multitwists $\eta$ whose derivative is $Q$. Then $D(\psi \circ \eta^{-1})=R^i$. Since $R^i$ is a rotation matrix, the composition $\psi \circ \eta^{-1}$ is an isometry.
\end{proof}

\section{Dynamical Results}
\label{sec:additional}

\subsection{Conjugacy to straight-line flow on a translation surface}
\label{sect:setup}
We will explain here that under mild conditions, the geodesic flow on a fixed direction in quarter-translation surface $\bS$ is defined for all time almost everywhere, and is conjugate to a straight-line flow on its minimal translation cover $\hat \bS$. This is a special case of the unfolding construction due to Fox-Kershner; see \cite{FK36} and the discussion in \cite{AAH}.

Suppose $\bS$ is a genuine quarter-translation surface. Recall that $\bSs$ denotes $\bS$ with its singularities removed. Then, the metric on $\bSs$ is a flat Riemannian metric, and we may consider the unit speed geodesic flow $F^t:T^1 \bSs \to T^1 \bSs$. In local coordinates coming from the charts defining the structure, 
$F^t(\bp, \bu)=(\bp+t\bu, \bu)$
for $t$ small, where $(\bp, \bu)$ denotes the unit tangent vector based at $\bp \in \RR^2$ in direction $\bu \in \SS^1$. Recall from the introduction and \Cref{sect:geodesics} that 
the direction $\dir(\bv) \in \SS^1/C_4$ 
of $\bv \in T^1 \bSs$
is invariant under the geodesic flow. For $[\bu] \in \SS^1/C_4$, we define $T^1_{[\bu]} \bSs = \dir^{-1}([\bu])$
and let $F^t_{[\bu]}:T^1_{[\bu]} \bSs \to T^1_{[\bu]} \bSs$
denote the restricted geodesic flow.

Note that $\bSs$ will not typically be complete, so trajectories may not be defined for all time. In the case of a closed surface obtainable by gluing together finitely many polygons, it is an elementary fact that for $t>0$, $F^t_{[\bu]}(\bv)$ only fails to be defined if there is a $t_0 \in (0,t]$ such that $\lim_{s \to t_0^-} F^s_{[\bu]}(\bv)$ is a singularity. That is, trajectories only fail to be defined through singularities. The set of all such $\bv$ has codimension one in $T^1_{[\bu]} \bSs$, so in this case $F^t_{[\bu]}$ is defined almost everywhere for all time. Note that this also holds on any cover of $\bSs$.

Let $\hat \bS$ denote the canonical translation surface cover of $\bS$, and let $\hbSs$ denote the preimage of $\bSs$. On a translation surface, the direction of a geodesic is an element $\bu \in \SS^1$ and is invariant under the geodesic flow $\hat F^t:T^1 \hbSs \to T^1 \hbSs$, which we define in local coordinates as before. Let $T^1_\bu \hbSs \subset T^1 \hbSs$ denote the unit tangent vectors with direction $\bu$. Observe that the projection to the base $T^1 \hbSs \to \hbSs$ restricts to a homeomorphism $\beta_\bu: T^1_\bu \hbSs \to \hbSs$. The {\em straight-line flow} in direction $\bu$ is
\begin{equation}
\label{eq:conjugacy 1}
\hat F_\bu^t: \hbSs \to \hbSs; \quad p \mapsto \beta_\bu \circ \hat F^t \circ \beta_\bu^{-1}(p)
\end{equation}
which can alternately be defined by the fact that, in local coordinates, $\hat F_\bu^t$ sends $\bp$ to $\bp+t\bu.$

Let $\hat \pi_\bS:\hat \bS \to \bS$ be the branched covering map. This covering also induces a covering map of degree four from $T^1 \hbSs$ to $T^1 \bSs$, which we will also denote by $\hat \pi_\bS$. From local considerations, we can see that $\hat \pi_\bS$ conjugates the geodesic flows on these surfaces. Since $\bS$ is a genuine quarter-translation surface, 
\Cref{deck group and derivative} guarantees the derivative map restricts to an isomorphism from the deck group to $C_4 \subset \SO(2)$. 
For $[\bu] \in \SS^1/C_4$, the preimage $\hat \pi_\bS^{-1}\big(T^1_{[\bu]} \bSs)$ consists of vectors in $T^1 \hbSs$ whose directions lie in $[\bu]$. This preimage has four connected components, namely $T^1_\bu \hbSs$ for $\bu \in [\bu]$. It follows that for each $\bu \in \SS^1$, the restriction 
$$\hat \pi_\bS|_{T^1_\bu \hbSs}: T^1_\bu \hbSs \to T^1_{[\bu]} \bSs$$
is a homeomorphism that conjugates the geodesic flows restricted to these subsets.
Combining this with \eqref{eq:conjugacy 1}, we have that:

\begin{prop}
\label{prop:conjugacy}
For each $\bu \in \SS^1$, the composition 
$$\hat \pi_\bS \circ \beta_u^{-1}:\hbSs \to T^1_{[\bu]} \bSs$$
is a topological conjugacy from the straight-line flow $F_\bu^t:\hbSs \to \hbSs$ to the restricted geodesic flow $F_{[\bu]}^t: T^1_{[\bu]} \bSs \to T^1_{[\bu]} \bSs$.
\end{prop}

\subsection{Recurrence via cylinders and hyperbolic geometry}
\label{sect:recurrence}
We introduce a recurrence criterion coming from consideration of cylinders in a translation surface closely related to the approach in \cite[\S 5]{HLT11}. The main idea is that cylinders enable close returns:

\begin{prop}
\label{prop:cylinder}
Let $\bS$ be a translation surface, let $C \subset \bS$ be a cylinder, and let $\bv \in \RR^2$ be the holonomy of a core curve of $C$.
If $\bu \in \SS^1$ is a unit vector, then the statement defined for $x \in X$,
\begin{equation}
\label{eq:inequality}
d\big(x,F_\bu^{|\bu \cdot \bv|}(x)\big) \leq |\bu \wedge \bv|
\end{equation}
holds for a subcylinder of $C$ whose complement has measure at most $|(\bu \cdot \bv)(\bu \wedge \bv)|$. Here $\wedge$ denotes the usual wedge product on $\RR^2$, $(a,b) \wedge (c,d)=ad-bc$.
\end{prop}


\begin{proof}
First we eliminate the case when $\bu$ is orthogonal to $\bv$. In this case statement \eqref{eq:inequality} becomes $d(x, x) \leq \|\bv\|$ which is trivially true all $x \in C$. For the remainder of the proof, assume that $\bu$ is not orthogonal to $\bv$.

We will show \eqref{eq:inequality} holds for $x \in C$ unless the orbit segment of $F_\bu^t(x)$ with $t \in \big[0,|\bu \cdot \bv|\big]$ intersects $\partial C$.

Observe that there is a $180^\circ$ rotation $C \to C$ that conjugates the flow $F_\bu$ to $F_{-\bu}$. Also if $\bv$ is a holonomy of $C$, then so is $-\bv$. Using these symmetries and by applying a linear isometry (perhaps orientation-reversing) to all the objects under discussion, we can assume that $\bu=(0,1)$ and that $\bv$ lies in the first quadrant. 

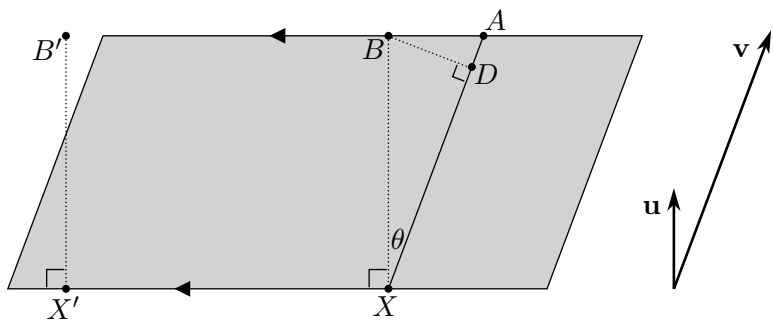
\begin{figure}[htb]
\labellist
\small\hair 2pt
 \pinlabel {$X'$} [t] at 32 3
 \pinlabel {$X$} [t] at 215 3
 \pinlabel {$A$} [bl] at 269 150
 \pinlabel {$B$} [ ] at 208 141
 \pinlabel {$B'$} [ ] at 23 141
 \pinlabel {$D$} [ ] at 272 127
 \pinlabel {$\theta$} [ ] at 222 33
 \pinlabel {$\bu$} [ ] at 366 51
 \pinlabel {$\bv$} [ ] at 417 140
\endlabellist
\centering
\includegraphics[width=4in]{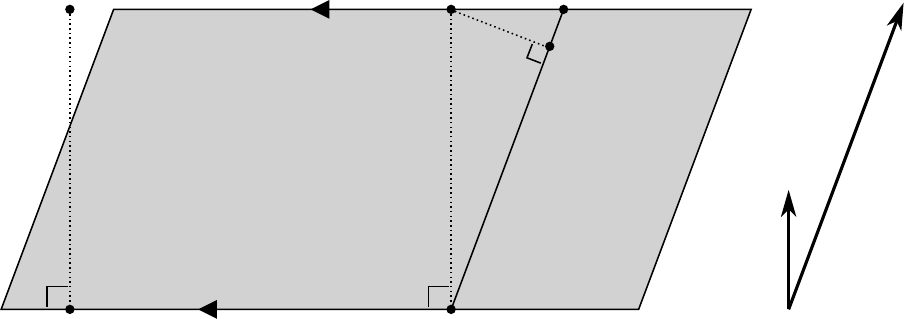}
\caption{The parallelogram $P$ utilized in the proof of \Cref{prop:cylinder}.}
\label{fig:strip}
\end{figure}

We can form a cylinder that is translation-equivalent to $C$ by identifying opposite sides of a parallelogram $P$ with one pair of horizontal sides, and the others pair of sides given by $\bv$. See \Cref{fig:strip}. Given any $x \in C$, we can define a translation-equivalence $f:P \to C$ such that a point $X$ on the bottom side of $P$ satisfies $f(X)=x$.
Let $A$ be the point on the top edge of $P$ such that $\overrightarrow{XA}=\bv$. Then we also have $f(A)=x$. Let $B$ be the point that has the same $x$-coordinate as $X$ and the same $y$-coordinate as $A$. Note that $B$ may or may not be in $P$. If $B \in P$, then because $f(\overline{AB})$ is a path joining $x$ to its image under vertical straight-line flow, we have
$$f(B)=F_{\bu}^{|XB|}(x)
\quad \text{and so} \quad
\dist\big(x, F_{\bu}^{|XB|}(x)\big) \leq |AB|$$
where 
$|XB|$ and $|AB|$ denote distances between two points in the plane.
Let $\theta$ be the angle between $\bu$ and $\bv$. Then we have $|\bu \wedge \bv|=\|\bv\| \sin \theta$ and
$|\bu \cdot \bv|=\|\bv\| \cos \theta$. Since $\angle XBA$ is a right angle, we have
$$|XB|=\|\bv\| \cos \theta=|\bu \cdot \bv| 
\quad \text{and} \quad
|AB|=\|\bv\| \sin \theta=|\bu \wedge \bv|.$$
We conclude that if $B \in P$, then \eqref{eq:inequality} holds. On the other hand, if $B \not \in P$, then the vertical segment $\overline{XB}$ intersects the left (as illustrated by $\overline{X'B'}$). In this case, the perpendicular distance from $X$ to the line extending the left side of $P$ is less than the length $|BD|$, where $D$ is the perpendicular projection of $B$ onto $\overline{XA}$. We have
$$|BD|=|XB| \sin \theta = |\bu \cdot \bv|\frac{|\bu \wedge \bv|}{\|\bv\|}.$$
Thus the set of points for which \eqref{eq:inequality} fails is contained in the collection of points contained within distance $|BD|$ from one boundary component. This is a subcylinder whose area is
$|BD| \|\bv\| = |(\bu \cdot \bv)(\bu \wedge \bv)|$ as claimed.
\end{proof}

We will connect \Cref{prop:cylinder} to hyperbolic geometry. The collection of images of a translation surface $\bS$ under area- and orientation-preserving affine maps up to rotation is naturally identified with 
\begin{equation}
\label{eq:affine images}
\SO(2) \backslash \SL(2,\RR)/D\big(\Aff(\bS)\big),
\end{equation}
where $D\big(\Aff(\bS)\big)$ is the Veech group of $\bS$ (i.e., the set of matrices whose actions fix the translation surface $\bS$).
This is a naturally a hyperbolic orbifold.

Let $\HH$ denote the upper half plane in $\CC$, which consists of complex numbers whose imaginary parts are positive. The boundary of $\HH$ in the Riemann sphere $\hat \CC=\CC \cup \{\infty\}$ is $\hat \RR=\RR \cup \{\infty\}$, which can naturally be identified with the collection of slopes of lines in the plane. The orientation-preserving isometry group of the hyperbolic plane is $\PSL(2,\RR)=\SL(2,\RR)/\langle -I \rangle$, and so we get an action of $\SL(2,\RR)$ on $\HH$ by isometry in which $-I$ acts trivially. For compatibility with \eqref{eq:affine images}, we need $\SL(2,\RR)$ to act on $\HH$ on the right. Therefore, we define the anti-homomorphism to the M\"obius group:
\begin{equation}
\label{eq:tau}
\tau: \SL(2,\RR) \to \Aut(\CC); \quad \begin{pmatrix}
a & b \\
c & d
\end{pmatrix} \mapsto \left(z \mapsto \frac{az-c}{d-bz}\right).
\end{equation}
We have defined $\tau$ to ensure that $\tau(M)(z)=z'$ implies that if $\bv \in \CC^2$ is a vector of slope $z$ then $M^{-1} \bv$ has slope $z'$. Since this is an anti-homomorphism, we have $\tau(M) \circ \tau(N) = \tau(NM)$ for all $M,N \in \SL(2,\RR)$.

The curve in $\bH$ given by $t \mapsto e^t i$ is a geodesic and we take it to be at unit speed. Observe that we can obtain this geodesic as 
\begin{equation}
\label{eq:vertical}
t \mapsto \tau(g_t)(i) \quad \text{where} \quad
G_t = \begin{pmatrix}
e^{t/2} & 0 \\
0 & e^{-t/2}
\end{pmatrix}.
\end{equation}
The isometry group of $\HH$ acts transitively on unit-speed parameterized geodesics. Thus, every unit-speed geodesic has the form 
$$\gamma:t \mapsto \tau(M)\big(\tau(G_t)(i)\big)=\tau(G_t M)(i).$$

The {\em Busemann function} associated to a unit-speed geodesic $\gamma:\RR \to \HH$ is $1$-Lipschitz function
$$B_\gamma: \HH \to \RR; \quad z \mapsto \lim_{t \to +\infty} \left[ d\big(\gamma(t),z\big)-t \right].$$
Note that if we are given a Busemann function $B_\gamma$, any sequence of points $z_n \in \HH$ such that $\lim B_\gamma(z_n)=-\infty$ must satisfy $\lim z_n=\gamma(+\infty) \in \hat \RR$. We call $\gamma(+\infty)$ the {\em boundary point associated to $B_\gamma$}. Note that not every sequence approaching the boundary point satisfies $\lim B_\gamma(z_n)=-\infty$.

There is a natural bijection between nonzero vectors in $\RR^2$ modulo the action of $-I$ and Busemann functions:

\begin{prop}
\label{prop:Busemann}
For any Busemann function $B_\gamma$, there is a nonzero $\bv \in \RR^2$ such that
\begin{equation}
\label{eq:Busemann}
B_\gamma\big(\tau(M)(i)\big)=2 \log \left\|M\bv\right\|
\quad \text{for all $M \in \SL(2,\RR)$}.
\end{equation}
Moreover the slope of $\bv$ is the endpoint $\gamma(+\infty)$. Conversely, given any nonzero $\bv \in \RR^2$, there is a Busemann function satisfying \eqref{eq:Busemann}.
\end{prop}
\begin{proof}
It is an elementary exercise to observe that if $\eta$ is a vertical geodesic of the form $t \mapsto a + e^t i$ for some $a \in \RR$, then the corresponding Busemann function is $B_\eta(z)=-\log \Im z$. We can then compute that 
$$B_\eta\big(\tau(M)(i)\big)= \log (b^2 + d^2) = \log \left\|M \begin{pmatrix}
0 \\
1 
\end{pmatrix}\right\|^2
\quad 
\text{for }
M = \begin{pmatrix}
a & b \\
c & d
\end{pmatrix} \in \SL(2,\RR).$$
Now if $\gamma$ is a different unit-speed geodesic in $\HH$, there is an $N \in \SL(2,\RR)$
such that $[\tau(N) \circ \gamma](t)=\eta(t)$. By naturality of the definition of the Busemann function, we have $B_\gamma(z)=B_{\tau(N) \circ \gamma}\big(\tau(N)(z)\big)$. Thus,
$$
B_\gamma\big(\tau(M)(i)\big) = 
B_\eta\big(\tau(MN)(i)\big) 
 =\log \left\|MN \begin{pmatrix}
0 \\
1 
\end{pmatrix}\right\|^2=2\log \left\|M\bv\right\|
$$
where $\bv=N\begin{pmatrix}
0 \\
1 
\end{pmatrix}.$ Write $N=\begin{pmatrix}
a' & b' \\
c' & d'
\end{pmatrix}.$ Then $\bv=\begin{pmatrix}
b' \\
d'
\end{pmatrix}.$
Since the endpoint of the geodesic $\eta$ is $\eta(+\infty)=\infty$, we have
$\gamma(+\infty)=\tau(N^{-1})(\infty)=\frac{d'}{b'}$, which agrees with the slope of $\bv$.
To see the converse, it suffices to observe that any nonzero $\bv$ can be realized as the second column of a matrix in $N \in \SL(2,\RR)$. Then defining $\gamma$ by $\gamma(t)=[\tau(N^{-1}) \circ \eta](t)$ gives a geodesic satisfying \eqref{eq:Busemann} by the computations above.
\end{proof}

\begin{lemma}
\label{min busemann}
Let $\bu \in \RR^2$ be a unit vector. Let $\bv \in \RR^2$ be a nonzero vector, and let $B$ be the Busemann function given satisfying \eqref{eq:Busemann} for this $\bv$. Let $\gamma:\RR \to \HH$ be the geodesic such that $\gamma(0)=i$ and $\gamma(+\infty)$ is the slope of $\bu$.
Then the function $\RR \to [0, +\infty]$ given by $t \mapsto \exp B\big(\gamma(t)\big)$ extends to a convex continuous function $[-\infty,+\infty] \to [0,+\infty]$ that attains its minimum value of $2\big|(\bu \cdot \bv)(\bu \wedge \bv)\big|$ only at 
$t=\log |\bu \cdot \bv| - \log |\bu \wedge \bv|$.
\end{lemma}
\begin{proof}
Write $\bu=(u_1, u_2)$. Let $R$ be the rotation matrix such that $R(\bu)=(0,1)$. Then,
$$R = \begin{pmatrix}
u_2 & -u_1 \\
u_1 & u_2  \\
\end{pmatrix}.$$
Define $\gamma(t)=[\tau(R) \circ \tau(G_t)](i)$. Then $\gamma$ is the image under $\tau(R)$ of the geodesic in \eqref{eq:vertical}. Therefore, we have $\gamma(+\infty)=\tau(R)(\infty)=\frac{u_2}{u_1}$, the slope of $\bu$, by \eqref{eq:tau}. Using our definition of $B$, we see that
$$\exp~B\big(\gamma(t)\big) = \exp~B\big(\tau(G_t R)(i)\big)=\|G_t R \bv\|.$$
We have
$$G_t R \bv = G_t \begin{pmatrix}
-\bu \wedge \bv \\
\bu \cdot \bv
\end{pmatrix} = \begin{pmatrix}
-e^{t/2}(\bu \wedge \bv) \\
e^{-t/2} (\bu \cdot \bv)
\end{pmatrix}.$$
Thus we see that
$$\exp B\big(\gamma(t)\big) = e^{t}(\bu \wedge \bv)^2 + e^{-t}(\bu \cdot \bv)^2,$$
which is clearly convex. It is a standard calculus exercise that this function, which is continuous and well-defined for $t \in [-\infty, +\infty]$ is minimized when $e^t=|\frac{\bu \cdot \bv}{\bu \wedge \bv}|$. Thus, we have
$$\inf\,\Big\{\exp B\big(\gamma(t)\big):~t \in \RR\Big\} = 2 \big|(\bu \cdot \bv)(\bu \wedge \bv)\big|$$
as desired.
\end{proof}

We turn our attention to the structure arising from completely periodic directions on a normal cover $\tbS$ of the closed translation surface $\bS$, possibly branched over the singularities of $\bS$. Let $\Delta$ denote the deck group. By normality of the cover, $\bS=\tbS/\Delta$.

\newcommand{\mcr}{\mathrm{mcr}}%
\newcommand{\mcs}{\mathrm{mcs}}%

Let $\bv$ be unit vector parallel to a cylinder decomposition on $\tbS$. We can use the deck group to construct a fundamental domain $D \subset \tbS$ for the action of $\Delta$. First, observe that the deck group permutes cylinders in the direction $\tbS$. Thus, we can choose a cylinder from every $\Delta$-orbit of a maximal cylinder. The image of each such cylinder is a different cylinder on $\bS$, so we end up choosing finitely many cylinders $C_1, \ldots, C_n \subset \tbS$. Let $c_i$ denote the circumference of $C_i$. 
The {\em maximal circumference ratio} for the completely periodic direction $\bv$ is 
\begin{equation}
\label{eq:mcr}
\mcr(\bv)=\max~\left\{\frac{c_i}{c_j}:~i,j \in \{1, \ldots, n\}\right\}.
\end{equation}
The stabilizer of a single cylinder $C_i$ must act on $C_i$ as a translation automorphism. Therefore each $C_i$ is stabilized by cyclic subgroup of $\Delta$ with finite order say $r_i \geq 1$. The maximal cylinder stabilizer in direction $\bv$ is 
\begin{equation}
\label{eq:mcs}
\mcs(\bv)=\max~\left\{r_i:~i \in \{1, \ldots, n\}\right\}.
\end{equation}
A rectangular fundamental domain $R_i$ for the action of the stabilizer in $\Delta$ on $C_i$ can be be obtained by cutting along two segments perpendicular to the core curve of $C_i$. (If $r_i=1$, we define $R_i=C_i$.) The area of $R_i$ is $\frac{1}{r_i}$ times the area of $C_i$. The union $D=\bigcup R_i$ is a fundamental domain for the $\Delta$ action on $\tbS$.

\begin{cor}
\label{cor:cylinder decomp}
Continue the notation of the paragraph above for working with a completely periodic direction described by a unit vector $\bv$. Assume that the cylinders are listed in order of increasing circumferences: $c_1 \leq c_2 \leq \ldots \leq c_n$.
Then for any unit vector $\bu$,
$$d(x,F^t_{\bu}(x)\big) \leq c_n |\bu \wedge \bv|$$
holds for some $t>c_1|\bu \cdot \bv|$ except for a subset of the fundamental domain $D$ whose measure is at most 
$$|(\bu \cdot \bv^\ast)(\bu \wedge \bv^\ast)|
\quad \text{where} \quad 
\bv^\ast = {\textstyle \sqrt{\sum_{i=1}^n \frac{c_i^2}{r_i}}}\,\bv.$$
\end{cor}
\begin{proof}
Fix $i$. \Cref{prop:cylinder} gives a return under $F_{[\bu]}$ at time $t_i=c_i |\bu \cdot \bv| \geq |\bu \cdot \bv_i\|$ to a distance within $c_i \|\bu \wedge \bv\| \leq c_n |\bu \wedge \bv\|$ for all $x \in C_i$ except a subcylinder $C_i'$ whose area is at most $c_i^2 |(\bu \cdot \bv)(\bu \wedge \bv)|$. Since $R_i \subset C_i$ is a rectangle formed by segments perpendicular to the boundary whose area is $\frac{1}{r_i}$ times the area of $C_i$, we have 
$R_i \cap C_i' = \frac{c_i^2}{r_i} |(\bu \cdot \bv)(\bu \wedge \bv)|$. Thus, the statement holds for all $x \in D$ except for those in $\bigcup_i (R_i \cap C_i')$, which has area $|(\bu \cdot \bv^\ast)(\bu \wedge \bv^\ast)|$ with $\bv^\ast$ as in the statement.
\end{proof}

Let $P \subset \SS^1$ denote a collection of completely periodic directions on $\tbS$.
Then for each $\bv \in P$, we can define $\bv^\ast$ as in \Cref{cor:cylinder decomp}. Via \Cref{prop:Busemann}, we can define the Busemann function associated to the direction $\bv$ via
\begin{equation}
\label{eq:Busemann of v}
B_\bv\big(\tau(M)(i)\big)=2 \log \|M \bv^\ast\|.
\end{equation}
For each $m \geq 1$, define 
\begin{equation}
\label{eq:P directions}
P_m=\{\bv \in P:~ \text{$\mcr(\bv) \leq c$ and $\mcs(\bv) \leq c$}\}.
\end{equation}

For a constant $m \geq 1$, we define
\begin{equation}
\label{eq:busemann inf}
B^{\inf}_m: \HH \to [-\infty, +\infty]; \quad B^{\inf}(z) = \inf \{B_\bv(z):~\bv \in P_m\}.
\end{equation}
Note that because it is the infimum of $1$-Lipschitz functions, $B^{\inf}_m$ is also $1$-Lipschitz. (If $P_m$ is empty, we take $B^{\inf}_m \equiv +\infty$.)

The following result is our main method for detecting recurrent directions:
\begin{lemma}
\label{lem:recurrence}
Let $\bu \in \RR^2$ be a unit vector. 
Let $\gamma:\RR \to \HH$ be the geodesic such that $\gamma(0)=i$ and $\gamma(+\infty)$ is the slope of $\bu$. If there is a $m \geq 1$ such that
$$\inf \big\{B^{\inf}_m\big(\gamma(t)\big):~ t \geq 0\}=-\infty,$$
then the flow $F_\bu$ is recurrent.
\end{lemma}
\begin{proof}
First note that we are really taking an infimum over both $\bv \in P$ and $t \in \RR$. That is, our hypothesis is equivalent to 
$$\inf \big\{B_\bv\big(\gamma(t)\big):~\text{$\bv \in P_m$ and $t\geq 0$}\}=-\infty.$$
Suppose this equation is true.

We eliminate two special cases.
For the first special case, it is possible that there is at least one $\bv \in P_m$ such that
$$\inf \{B_\bv\big(\gamma(t)\big):~\text{$t \geq 0$}\}=-\infty.$$
In this case, \Cref{min busemann} guarantees that after continuously extending the function $t \mapsto B_\bv\big(\gamma(t)\big),$ the minimum is attained at $t=+\infty$. Further, since the minimum is attained when $e^t=\big|\frac{\bu \cdot \bv^\ast}{\bu \wedge \bv^\ast}\big|$, $\bu$ must be parallel to $\bv$. Since $\bv$ is the direction of a cylinder decomposition of $\tbS$, this is a recurrent direction.

For the second special case, it is possible that there is a nonrepeating sequence $\bv_n \in P_m$ such that
$y_n = B_{\bv_n} \big(\gamma(0)\big) \to -\infty$.
We will show this is actually impossible. Since $\gamma(0)=i=\tau(I)(i)$, via our definition of $B_\bv$ we have
$y_n = 2 \log \|\bv^\ast_n\|.$
We conclude that $\|\bv^\ast_n\| \to 0$. Note that in each direction $\bv_n$, the cylinders have circumference less than $m\|\bv^\ast_n\|$. The image of such a cylinder in $\bS$ is a cylinder of (nonstrictly) smaller circumference. So, $\bS$ has a sequence of cylinders whose circumference tends to zero. This contradicts the fact that there are only finitely many holonomy vectors of cylinders in a bounded set (which follows from the main results in \cite{Masur88}).

For the remainder of the proof, suppose we not in the first special case above. Then there is a nonrepeating sequence $\bv_n \in P_m$ such that
$$x_n = \inf \{B_{\bv_n}\big(\gamma(t)\big):~\text{$t \geq 0$}\} \to -\infty \quad \text{as $n \to +\infty$}$$
\Cref{min busemann} allows us to evaluate these infima; We have
$$\frac{1}{2} \exp x_n = \begin{cases}
\frac{1}{2} \exp B_{\bv_n}\big(\gamma(0)\big) & \text{if $|\bu \cdot \bv_n^\ast| \leq |\bu \wedge \bv_n^\ast|$} \\
\big|(\bu \cdot \bv_n^\ast)(\bu \wedge \bv_n^\ast)\big| & \text{otherwise.}
\end{cases}$$
Since the second special case above is impossible, we can assume by passing to a subsequence that $|\bu \cdot \bv_n^\ast| \geq |\bu \wedge \bv_n^\ast|$ for all $n$ and therefore
$$\tfrac{1}{2} \exp x_n = \big|(\bu \cdot \bv_n^\ast)(\bu \wedge \bv_n^\ast)\big| \to 0 \quad \text{as $n \to +\infty$}.$$

We now claim that $\lim_{n \to +\infty} \|\bv_n^\ast\| = +\infty$. This follows from the fact that the set of all $\bv$ that satisfy the condition 
$$K \geq |\bu \cdot \bv| \geq |\bu \wedge \bv|$$
is compact and star-shaped. By linearity of the dot and wedge product, this region is star-shaped in the sense that if $\bv$ satisfies this condition, and $y \in [0,1]$, then $y \bv$ satisfies the statement. The region is compact because it is closed and by the Pythagorean theorem it is bounded:
$$\|\bv\|^2 = |\bu \cdot \bv|^2 + |\bu \wedge \bv|^2.$$
Again, by definition of $\bv_n^\ast$, for each $n$, there is a cylinder on $\tbS$ (and hence also on $\bS$) whose holonomy has the form $y m \bv_n^\ast$ for some $y \in [0,1]$. So, existence of a bounded subsequence of the $\bv_n^\ast$ would again lead to a contradiction of the fact that given any bound, the number of holonomy vectors of cylinders on $\bS$ with norm less than the bound is finite.

Since we have $\|\bv_n^\ast\| \to +\infty$, $|\bu \cdot \bv_n^\ast| \geq |\bu \wedge \bv_n^\ast|$, and $\big|(\bu \cdot \bv_n^\ast)(\bu \wedge \bv_n^\ast)\big| \to 0$, we can conclude that
$$|\bu \cdot \bv_n^\ast| \to +\infty \quad \text{and} \quad |\bu \wedge \bv_n^\ast| \to 0.$$
Let $t_n$ denote the lower bound on the close return times provided by \Cref{cor:cylinder decomp} for the direction $\bv_n$ (namely, $c_1 \|\bu \cdot \bv_n\|$ in the notation of the corollary). Again, minimal cylinder circumferences in a periodic direction $\bv_n$ on $\bS$ must tend to $+\infty$ as $n \to +\infty$, so we have $t_n \to +\infty$. Let $d_n$ be the maximal return distance guaranteed by the Corollary ($c_n \|\bu \wedge \bv_n\|$ in the notation of the corollary). Noting that cylinder circumferences in direction $\bv_n$ are bounded by $m \|\bv_n^\ast\|$, we see that
$$d_n \leq m \|\bv_n^\ast\| |\bu \wedge \bv_n| \leq m |\bu \cdot \bv_n^\ast| |\bu \wedge \bv_n| \leq m\big|(\bu \cdot \bv_n^\ast)(\bu \wedge \bv_n^\ast)\big|,$$
for any $n$ such that $\|\bv_n^\ast\| \geq 1$. The last expression in the inequality above is $m$ times something we know tends to zero, so $d_n \to 0$. The subset of $A_n$ of a fundamental domain satisfying the return statement of \Cref{cor:cylinder decomp} has a compliment in the fundamental domain with measure at most $\big|(\bu \cdot \bv_n^\ast)(\bu \wedge \bv_n^\ast)\big| \to 0$. These compliments have measure tending to zero, so for almost every point $x \in \tbS$, there are infinitely many $n$ such that there is a deck transformation $\delta \in \Delta$ such that $\delta(x) \in A_n$. Any such point is positively recurrent under $F_{\bu}$. The same reasoning holds for $F_{-\bu}$ and a positively recurrent point for $F_{-\bu}$ is negatively recurrent for $F_{\bu}$. Thus, $F_\bu$ is recurrent.
\end{proof}

Now we can prove our recurrence theorem.

\begin{proof}[Proof of \Cref{thm:recurrence}]
Let $\bNh$ denote the translation cover of $\bN$, and let $\bHh$ denote the translation cover of $\bH$. Then $\bNh$ is a normal cover of $\bHh$ with deck group isomorphic to $\ZZ^2$. Note that $\bHh$ has Veech group $\Gamma_2 \subset \SL(2,\RR)$ where $\Gamma_2$ is as in \eqref{eq:Gamma 2 -I}.

Let $\bv$ be a unit vector whose slope is $\frac{p}{q}$ for two odd, relatively prime integers $p$ and $q$. As a consequence of \Cref{thm:odd-odd is closed}, the straight-line flow in this direction is completely periodic and each cylinder in this direction has circumference $6 \sqrt{p^2+q^2}$ and area six. This single cylinder on $\bHt$ has two lifts to $\bHh$, forming a cylinder decomposition of $\bHh$ with slope $\frac{p}{q}$ that consists of two cylinders of circumference $6 \sqrt{p^2+q^2}$ and area six.
Thus a fundamental domain for the deck group action on $\bNh$ consists of two cylinders of slope $\frac{p}{q}$, circumference $6 \sqrt{p^2+q^2}$, and area six. It follows that in the notation of \Cref{cor:cylinder decomp}, for $\bv$ as above, we have
$$\bv^\ast = 6\sqrt{2(p^2+q^2)} \bv=6 \sqrt{2}\begin{pmatrix}
q \\ p \end{pmatrix}.$$

We let $P$ be the collection of unit vectors with rational, odd over odd slope. From the above, this is a collection of completely periodic directions for $\bNh$ and we have $P_1=P$ where $P_1$ is defined as in \eqref{eq:P directions}. Note that the collection of vectors
$P^\ast=\{\bv^\ast:~\bv \in P\}$
is invariant under the action of $\Gamma_2$ by \Cref{lem:transitive}. For $\bv \in P$, we define the associated Busemann function $B_\bv$ as in \eqref{eq:Busemann of v}. This collection of Busemann functions is $\Gamma_2$ invariant in the sense that
$$B_\bw(z) = [B_\bv \circ \tau(N)](z) \quad \text{for $N \in \Gamma_2$ and $\bv, \bw \in P$ such that $\bw^\ast=N\bv^\ast$.}$$
Therefore, the function $B_1^{\inf}:\HH \to [-\infty, +\infty)$ as defined in \eqref{eq:busemann inf} is $\Gamma_2$-invariant and therefore descends to a well-defined function $F:\Gamma_2 \backslash \HH \to [-\infty, +\infty]$ which is also $1$-Lipschitz because $B_1^{\inf}$ is. Now fix a $\bu \in \SS^1$ and let $\gamma_\bu:[0,+\infty) \to \HH$ be the geodesic ray with $\gamma_\bu(0)=i$ and $\gamma_\bu(+\infty)$ the slope of $\bu$. We have
\begin{equation}
\label{eq:B equals F}
B_1^{\inf} \circ \gamma_\bu(t) = F(\Gamma_2 \cdot \gamma_\bu(t)\big).
\end{equation}

Note that $\Gamma_2 \backslash \HH$ is a finite-volume complete hyperbolic surface. Therefore, the geodesic flow on $\Gamma_2 \backslash \HH$ is ergodic. So almost every geodesic ray equidistributes with respect to the Harr measure on the quotient. Noting that two geodesic rays with the same endpoint either both equidistribute or neither does, we see that for almost every $\bu$, $\gamma_\bu$ equidistributes. Since $F$ is a continuous function whose infimum is $-\infty$, it follows that if $\bu$ is a unit vector such that $\gamma_\bu$ equidistributes, we have $\inf \{F(\Gamma_2 \cdot \gamma_\bu(t)\big)=-\infty$. Therefore, by
\eqref{eq:B equals F} and \Cref{lem:recurrence}, any such $\bu$ is a recurrent direction for the straight-line flow on $\bNh$. From \Cref{prop:conjugacy}, we conclude that for almost every $[\bu] \in \SS^1/C_4$, the flow $F_{[\bu]}^t$ is recurrent.
\end{proof}

\subsection{Ergodicity}

\begin{proof}[Proof of \Cref{thm:ergodicity}]
Recall that $\bN$ is a normal $\ZZ^2$-cover of $\bH$ as quarter-translation surfaces. It follows from \Cref{commute-branch-cover} that this covering lifts
to a covering between the minimal translation covers, $\bNh \to \bHh$, and \Cref{deck groups} guarantees that this is a normal cover with deck group isomorphic to $\ZZ^2$.

By \Cref{cor:odd twist}, $\bNt$ has twistable cylinder decompositions in the dense set of directions whose slopes are the ratio of two odd integers. Since $\ZZ^2$ is a nilpotent group, Corollary G.4 of \cite{H15} gives ergodicity of the straight-line flow on $\bNt$ in a dense set of directions with Hausdorff dimension larger than $\frac{1}{2}$. (In the notation of Appendix G of \cite{H15}, note that Lebesgue measure on $\bNt$ is the unique $\chi$-Maharam measure up to scale where $\chi:\ZZ^2 \to \RR$ is the trivial homomorphism $\chi \equiv 0$ \cite[Appendix G.2]{H15}.)
\end{proof}

\begin{remark}[Measure classification]
Corollary G.4 of \cite{H15} gives a classification of the locally finite ergodic straight-line flow invariant measures in the set of directions appearing in the proof; they are the Maharam measures and are in bijection with group homomorphisms $\ZZ^2 \to \RR$.
\end{remark}

\subsection{Divergence rate}
In \cite{DHL14}, the authors study a $\GL(2,\RR)$-invariant sublocus $\cG$ of the stratum $\cH(2^4)$ of genus five translation surfaces with four singularities with $6 \pi$ cone angles. (We follow the convention of \cite[\S 2]{DHL14} that the singularities be marked and come with a choice of a horizontal separatrix emanating from each singularity.) The sublocus $\cG$ is a collection of surfaces in $\cH(2^4)$ arising from a covering construction from a genus two surface with one singularity. These covers will be normal with the Klein four-group $(\ZZ/2\ZZ)^2$ as the Deck group.

By definition, $\cG$ is the $\GL(2,\RR)$-orbit closure of branched covers of $L$-shaped translation surfaces with genus two. We will verify that the translation cover $\bHh$ of the half-cube $\bH$ lies in the locus $\cG$, because this will enable use to utilize the analysis in \cite{DHL14}. Therefore, in this section we will follow to the notation in \cite{DHL14} as much as possible. The analysis of drift on $\bNh$ is almost identical to that of unfolded Ehrenfest wind-tree models associated to the simplest lattice examples in $\cG$, and all the hard work is done in \cite{DHL14}. This section summarizes their approach, and explains the relationship with the Necker cube surface.

First, we will explain that $\bHh$ is in $\cG$.
Take four square tori with central $\frac{1}{2} \times \frac{1}{2}$ sub-squares removed from each square. Let $T_{i,j}$ denote the resulting tori with squares removed for $i,j \in \ZZ/2\ZZ=\{0,1\}$. Then we identify each vertical edge of the removed square of $T_{i,j}$ to the opposite vertical edge of $T_{i+1,j}$, and we identify each horizontal edge of the removed square of $T_{i,j}$ to the opposite horizontal edge of $T_{i,j+1}$. A surface describable in this manner is shown on the right side of \Cref{fig:diffusion}. This figure also provides a new presentation of $\bH$ as a union of polygons with edge identifications and makes it clear that $\bHh$ is the translation cover of $\bH$. The surface $\bHh$ is also one of the $L$-shaped surface covers described in \cite[\S 3]{DHL14}. In particular, we conclude that $\bHh \in \cG$. The translation automorphism group of $\bHh$ contains an action of $(\ZZ/2\ZZ)^2$; Each $(k,l) \in (\ZZ/2\ZZ)^2$ acts by sending each $T_{i,j}$ to $T_{i+k,j+l}$ by translation. The quotient of $\bHh$ modulo this group action is the genus two surface covered by $\bHh$.

\begin{figure}[htb]
\begin{minipage}{5.7in}
\labellist
\small\hair 2pt
 \pinlabel {$x$} [l] at 150 79
 \pinlabel {$y$} [b] at 74 152
\endlabellist
\centering
\raisebox{-0.5\height}{\includegraphics[width=1.5in]{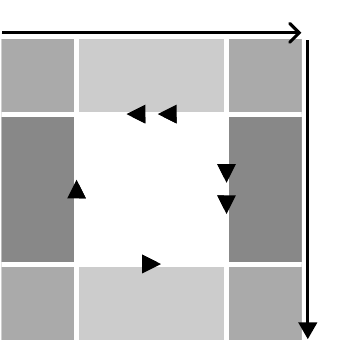}}
\labellist
\small\hair 2pt
 \pinlabel {$\hat x$} [l] at 150 275
 \pinlabel {$\hat y$} [b] at 71 348
 \pinlabel {$\hat x$} [b] at 272 348
 \pinlabel {$\hat y$} [r] at 195 273
 \pinlabel {$\hat x$} [t] at 69 14
 \pinlabel {$\hat y$} [l] at 150 91
 \pinlabel {$\hat x$} [r] at 195 93
 \pinlabel {$\hat y$} [t] at 267 14
 \pinlabel {$T_{0,0}$} [ ] at 71 91
 \pinlabel {$T_{1,0}$} [ ] at 270 93
 \pinlabel {$T_{0,1}$} [ ] at 71 273
 \pinlabel {$T_{1,1}$} [ ] at 272 271
\endlabellist
\hspace{0.5in}
\raisebox{-0.5\height}{\includegraphics[width=3.5in]{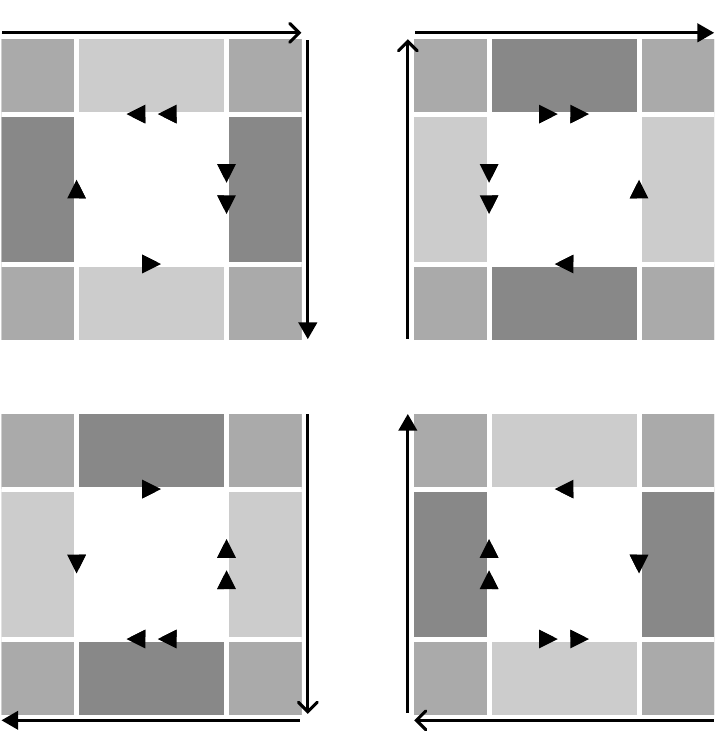}}
\end{minipage}
\caption{Left: A representation of the half-cube $\bH$ as the square torus with an inner square removed with inner edges identified. Right: The translation cover $\bHh$ of $\bH$ can be constructed from four such tori with inner squares removed together with edge identifications. This time inner edges are identified if the markings agree up to translation. In this presentation, the covering map $\bHh \to \bH$ given by orientation-preserving isometries from the tori with squares removed $T_{i,j}$ to the tori with square removed in the presentation for $\bH$ which respect the markings on the inner squares.}
\label{fig:diffusion}
\end{figure}
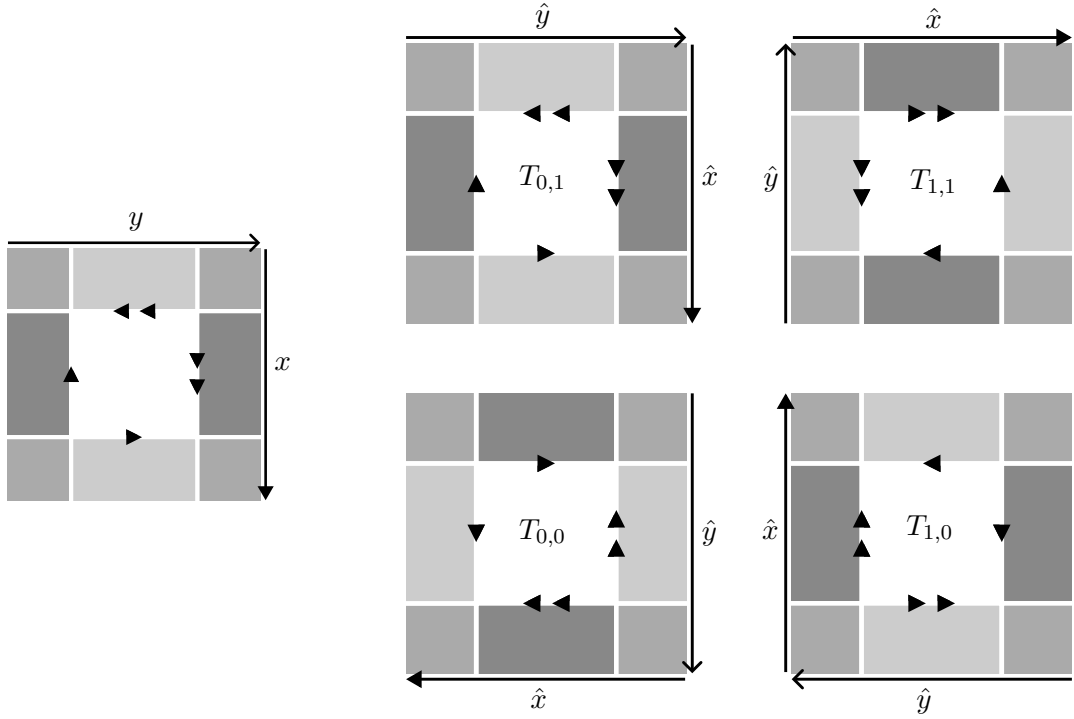

We will need to name some cohomology classes on the surface $\bHh$. For each pair $i,j \in \ZZ/2\ZZ$, let $h_{i,j} \in H^1(\bHh; \ZZ)$ denote the cohomology class such that for $\gamma \in H_1(\bHh;\ZZ)$, $h_{ij}(\gamma)$ returns the (signed) number of times $\gamma$ moves upward across a closed horizontal leaf in $T_{i,j}$. Similarly, let $v_{ij}$ denote the cohomology class keeping track of the number of times $\gamma$ moves rightward across a closed vertical leaf in $T_{i,j}$. The article \cite{DHL14} also uses the cohomology classes $h_{ij}$ and $v_{ij}$. 

\Cref{fig:diffusion} illustrates the Poincar\'e duals to two additional cohomology classes:
$$
\hat x = h_{00}-h_{11}+v_{01}-v_{10}
\quad \text{and} \quad
\hat y = h_{10}-h_{01}+v_{00}-v_{11}.$$
Let $\hat g=(\hat x, \hat y) \in H^1(\bHh; \ZZ^2)$. Observe:

\begin{prop}
The translation cover $\bNh$ of $\bN$ is isomorphic to the cover of $\bHh$ associated to $\hat g$, i.e., a curve $\hat \gamma$ in $\bHh$ lifts to $\bNh$ if and only if $\langle \hat x, \hat \gamma \rangle=\langle \hat y, \hat \gamma \rangle=0.$
\end{prop}
\begin{proof}
By \Cref{commute-branch-cover}, we obtain the following commutative diagram: 
$$
\begin{tikzcd}
	\bNh \arrow[r, "\hat \pi_\bN"] \arrow[d]
		& \bN \arrow[d] \\
	 \bHh \arrow[r, "\hat \pi_\bH"]
	& \bH
\end{tikzcd}
$$
We claim that a loop $\hat \gamma$ in $\bHh$ lifts to $\bNh$ if and only if $\gamma = \hat \pi_\bH(\hat \gamma)$ in $\bH$ lifts to $\bN$. It is clear by commutativity that if $\hat \gamma$ lifts to a loop $\hat \eta$ on $\bNh$, then also $\eta =\hat \pi_\bN(\hat \eta)$ is a lift of $\gamma$. In the other direction, suppose $\gamma$ in $\bH$ lifts to a loop $\eta$ in $\bN$. Note that because $\hat \gamma$ is a loop in a translation surface, it has trivial rotational holonomy. Since $\gamma = \hat \pi_\bH(\hat \gamma)$, it must also have trivial rotational holonomy. Since $\eta$ is a lift of $\gamma$, the loop $\eta$ must also have trivial rotational holonomy. Since $\bNh$ is the translation cover of $\bN$ and $\eta$ has trivial rotational holonomy, $\eta$ lifts to a loop $\hat \eta$ on $\hat \bN$. This is the desired lift of $\eta$, proving the claim.

Note $x$ and $y$ generate $H^1(\bH; \ZZ)$.
Thus a loop $\hat \gamma$ in $\bHh$ lifts to $\bNh$ if and only if $\langle x,\gamma\rangle=0$ and $\langle y, \gamma\rangle=0$ where $\gamma=\hat \pi_\bH(\hat \gamma)$. The classes $\hat x, \hat y \in H^1(\bHh;\ZZ)$ were chosen to ensure that 
$\langle \hat x,\hat \gamma \rangle=\langle x,\gamma\rangle$ and $\langle \hat y, \hat \gamma\rangle=\langle y, \gamma\rangle$
for all $\hat \gamma \in H_1(\bHh; \ZZ)$. Thus $\hat \gamma$ lifts to $\bNh$ if and only if $\langle \hat x,\hat \gamma\rangle=\langle \hat y,\hat \gamma\rangle=0$ as desired.
\end{proof}

Let $(X_1, d_1)$ and $(X_2, d_2)$ be two metric spaces.
A {\em quasi-isometry} $f:X_1 \to X_2$ is a map such that there are constants $a \geq 1$, $b \geq 0$, and $c \geq 0$ such that the two following statements hold:
\begin{enumerate}
\item For every $p,q \in X_1$, we have 
$\frac{1}{a} \cdot d_1(p,q) - b\leq d_2\big(f(p), f(q)\big) \leq a \cdot d_1(p,q) + b$.
\item For every $r \in X_2$, there is a $p \in X_1$ such that $d_2\big(r, f(p)\big)\leq c$.
\end{enumerate}
We will call a surjective map satisfying (1) an {\em $(a,b)$-quasi-isometry}. It is well known that the composition of quasi-isometries is a quasi-isometry. Also, the following standard result guarantees that the exponential growth rates between trajectories are quasi-isometry invariant.
\begin{prop}
\label{div and qi}
Suppose $p_n, q_n \in X_1$ are sequences and $t_n \in \RR$ is a sequence such that $t_n \to +\infty$. If $f:X_1 \to X_2$ is a quasi-isometry, then 
$$\lim_{n \to \infty} \frac{\log d_1(p_n,q_n)}{\log t_n} \quad \text{exists if and only if} \quad
\lim_{n \to \infty} \frac{\log d_2\big(f(p_n),f(q_n)\big)}{\log t_n} \quad \text{exists}$$
and if they exist they are equal.
\end{prop}

Consider $\bH$ as depicted on the left side of \Cref{fig:diffusion}. Removing the curves labeled $x$ and $y$ from $\bH$ produces a simply connected domain $D \subset \bH$. Let $\tilde D_{0,0}$ be a lift of $D$ to $\bN$. Let $(m,n) \mapsto \delta_{m,n}$ be the natural isomorphism from $\ZZ^2$ to the deck group of the covering $\bN \to \bH$ such that if $\gamma:[0,1] \to \bH$ is a loop and $\eta:[0,1] \to \bN$ is a lift we have
$$\delta_{m,n}\big(\eta(0)\big)=\eta(1) \quad \text{where} \quad
m=\langle x,\gamma\rangle \quad \text{and} \quad n=\langle y, \gamma\rangle.$$
Let $\tilde D_{m,n}=\delta_{m,n}(\tilde D_{0,0})$. Then we can define
$$\iota:\bN \to \ZZ^2 \quad \text{by} \quad
\iota(p)=(m,n) \quad \text{if $p \in \tilde D_{m,n}$,}$$
where we also just make a choice for points on the common boundary of two domains.

We endow $\ZZ^2$ with its Euclidean metric to make it into a metric space.
We have:
\begin{prop}
\label{quasi-isometry}
The map $\iota:\bN \to \ZZ^2$ is a quasi-isometry.
\end{prop}
\begin{proof}
Choose a point $q_{0,0} \in \tilde D_{0,0}$. Define $q_{m,n}=\delta_{m,n}(q_{0,0})$. Because the deck group is a group of translations of $\RR^3$, the map $q_{m,n} \mapsto (m,n) \in \ZZ^2$ is affine. It follows that $q_{m,n} \mapsto (m,n)$ is an $(a,0)$-quasi-isometry where $a$ is the operator norm of the derivative of this affine map. The map sending all of $\tilde D_{m,n}$ to $q_{m,n}$ for each $m,n \in \ZZ$ is a $(0,b)$-quasi-isometry where $b$ is the diameter of the compact set $\tilde D_{0,0}$. The map $\iota$ is the composition of the map sending $\tilde D_{m,n}$ to $q_{m,n}$ followed by the map sending $q_{m,n}$ to $(m,n)$, and therefore $\iota$ is a quasi-isometry.
\end{proof}

As in \cite[Lemma 4]{DHL14}, the action of the Klein-four group allows for us to split absolute cohomology into natural pieces for surfaces with presentations such as $\bHh$. We have following the direct sum:
\begin{equation}
\label{eq:splitting}
H_1(\bHh; \QQ)=E^{++} \oplus E^{+-} \oplus E^{-+} \oplus E^{--}.
\end{equation}
Here the signs are related to the manner in which generators for the Klein four-group act on cohomology. \cite{DHL14} computes that $E^{+-}$ and $E^{-+}$ are two dimensional subspaces with bases given by:
\begin{align*}
E^{+-} &= \Span \{h_{00} - h_{01} + h_{10} - h_{11}, v_{00}-v_{01}+v_{10}-v_{11}\}, \quad \text{and} \\
E^{-+} &= \Span \{h_{00} + h_{01} - h_{10} - h_{11}, v_{00}+v_{01}-v_{10}-v_{11}\}.
\end{align*}
The following is a trivial observation:

\begin{prop}
We have $\hat x, \hat y \in E^{+-} \oplus E^{-+}$.
\end{prop}

In section 4 of \cite{DHL14}, it is explained that the splitting \eqref{eq:splitting} extends to an $\SL(2,\RR)$-equivariant splitting of the Hodge bundle over $\cG$. In particular, we get vector bundles over $\cG$ for each of the four subspaces. Further, the Lyapunov exponents are computed for any $\SL(2,\RR)$-invariant probability measure on $\cG$. In our case, $\bHh$ is a lattice surface as it is preserved by $\Gamma_2 \subset \SL(2,\RR)$ and Harr measure on the quotient $\Gamma_2 \backslash \SL(2,\RR)$ induces a finite measure on the orbit $\SL(2,\RR) \cdot \bHh$. But regardless of the measure, the set of Lyapunov exponents of the Kontsevich-Zorich cocycle restricted to the bundles associated to $E^{+-}$ (or $E^{-+}$) is computed to be $\{\pm \frac{2}{3}\}$.

Now consider a unit vector $\bu \in \SS^1 \subset \RR^2$ and consider the straight-line flow $\hat F^t_\bu$ on $\bHh$. For each $\bu$ and each $T>0$ and each $p \in \bHh$, $\gamma_{\bu, T}$ to be the path 
$$
[0,T] \to \bHh; \quad t \mapsto \hat F^t_\bu(p).
$$
closed to form a loop by an arc joining $p$ to $\hat F^t_\bu(p)$ by a path which intersects neither $\hat x$ nor $\hat y$, viewed as curves depicted in \Cref{fig:diffusion}. By making compatible boundary choices as with the definition of $\iota$, we can arrange that
\begin{equation}
\label{pairing meaning}
\langle \hat g, \gamma_{\bu, T}(p) \rangle = \iota \circ \hat \pi_\bN \circ \hat F^T_\bu(p) - \iota \circ \hat \pi_\bN(p)
\end{equation}
for all $p$ and $T>0$ such that $F^T_\bu(p)$ is defined.

We have the following as a consequence of Theorem 2 of \cite{DHL14} as well as further reasoning in that paper:

\begin{lemma}
\label{divergence main}
For Lebesgue almost every $\bu \in \SS^1$, for every $p \in \bHh$ with an infinite $\hat F_\bu$-trajectory, we have
$$\limsup_{T \to \infty} \frac{\log \| \langle \hat g, \gamma_{\bu, T}(p)\rangle\|_2}{\log T}=\frac{2}{3}.$$
\end{lemma}
\begin{proof}
Let $\cX \subset \cG$ denote the $\SL(2,\RR)$ orbit of $\bHh$, which is endowed with its Harr measure $\mu$. 
Let $\bS \in \cX$. Then by definition of $\cX$, there is an $A \in \SL(2, \RR)$ be such that $A(\bS)=\bHh$.
Also, there is an affine homeomorphism $\psi_A:\bS \to \bHh$ with derivative $A$.
Let $\beta^A_T(p)$ denote the vertical segment of length $T$ starting at $p \in \bS=A^{-1}(\bHh)$, closed while avoiding the images of curves $\hat x_A=\psi_A^{-1}(\hat x)$ and $\hat y_A = \psi_A^{-1}(\hat y)$ in $\bS$. Thus $\beta^A_T(p)$ is analogous to the definition of $\gamma_{\bu,T}$. We view $\hat x_A$ and $\hat y_A$ as cohomology classes in $H^1(\bS; \ZZ)$. These classes depend on $A$, and $\psi_A$ is uniquely determined by $A$ from the convention that singularities in our stratum are marked points and come with a choice of a horizontal prong emanating from each singularity.


We claim that for almost every $\bS$, and any $A$ such that $A(\bS)=\bHh$, the class $\hat x_A$ lies in $E^{+-} \oplus E^{-+}$ and not in the subspace associated to $\frac{-2}{3}$.
The fact that $\hat x_A \in E^{+-} \oplus E^{-+}$ is just $\SL(2,\RR)$-invariance of the splitting; see \cite[\S 4.1]{DHL14}. 
As in \cite[Proof of Lemma 11]{DHL14}, $\hat x_A$ cannot lie in the subspace associated to Lyapunov exponent $\frac{-2}{3}$ because it is a nonzero integer cohomology class, so the norm of its orbit under the Kontsevich-Zorich cocycle cannot tend to zero. The same arguments hold for $\hat y_A$.

For any of the $\mu$-almost every $\bS \in \cX$ satisfying the hypotheses of Theorem 2 of \cite{DHL14} (guaranteeing generic recurrence and Oceledets genericity), we have
$$\limsup_{T \to \infty} \frac{\log |\langle \hat z, \beta^A_T(p)\rangle|}{\log T}=\frac{2}{3}$$
for any $A$ such that $A(\bS)=\bHh$ and any class $\hat z$ in $E^{+-} \oplus E^{-+} \subset H^1(\bS; \RR)$ that does not lie in the subspace for the Oseledets decomposition associated to the 
Lyapunov exponent $\frac{-2}{3}$ coming for the Kontsevich-Zorich cocycle over the Teichm\"uller flow.

We have shown that for $\mu$-almost every $\bS$, for any $A$ such that $A(\bHh)=\bS$, we have
$$\limsup_{T \to \infty} \frac{\log |\langle \hat x_A, \beta^A_T(p)\rangle|}{\log T}=\limsup_{T \to \infty} \frac{\log |\langle \hat y_A, \beta^A_T(p)\rangle|}{\log T}=\frac{2}{3}.$$
Since for $(a,b) \in \RR^2$, we have $\max \{|a|,|b|\} \leq \|(a,b)\|_2 \leq |a|+|b|$, we conclude under the same hypotheses that
\begin{equation}
\label{eq:limsup1}
\limsup_{T \to \infty} \frac{\log \|\langle \hat g_A, \beta^A_T(p)\rangle\|_2}{\log T}=\frac{2}{3},
\quad
\text{where $\hat g_A=(\hat x_A, \hat y_A) \in H^1(\bS; \ZZ^2)$}.
\end{equation}

Now we claim that the condition in \eqref{eq:limsup1} is invariant in the sense that if $A$ makes \eqref{eq:limsup1} true
and $B \in \SL(2,\RR)$ preserves the vertical direction in the sense that $B \begin{pmatrix}0 \\ 1\end{pmatrix}=e^s \begin{pmatrix}0 \\ 1\end{pmatrix}$ for some $s \in \RR$, then we have
\begin{equation}
\label{eq:limsup2}
\limsup_{T \to \infty} \frac{\log \|\langle \hat g_{B \cdot A}, \beta^{B \cdot A}_T(p)\rangle\|_2}{\log T}=\frac{2}{3}.
\end{equation}
With the stated hypotheses on $B$ and $s$, let $\bT \in \cX$ be such that $B(\bT)=\bS$. Then there is an affine homeomorphism $\phi_B:\bT \to \bS$ with derivative $B$. We have $\phi_B\big(\beta^{B \cdot A}_{e^{-s} T}(q)\big)=\beta^A_T\big(\phi_B(q)\big)$ for all $q \in \bT$ and all $T>0$, and $(\phi_B)_\ast(\hat g_{B \cdot A})=\hat g_A$.
Using these observations, we see that
$$\langle \hat g_{B \cdot A}, \beta^{B \cdot A}_{e^{-s}T}(p)\rangle=\langle \hat g_{A}, \beta^{A}_T(p)\rangle.
$$
It follows that statements \eqref{eq:limsup1} and \eqref{eq:limsup2} are equivalent when $B$ preserves the vertical direction.

Let $R_\theta \in \SL(2,\RR)$ denote the rotation by $\theta \in \RR/2\pi \ZZ$.
The KAN decomposition of $\SL(2,\RR)$ tells us that for every $A \in \SL(2, \RR)$, there is a rotation matrix $R_\theta$ and a $B$ preserving the vertical direction in the sense above such that $A=B \cdot R_\theta$. Set $X_\theta=\{B \cdot R_\theta\}$ where $B \in \SL(2,\RR)$ varies over the matrices preserving the vertical direction. Then we have the partition $\SL(2,\RR)=\bigsqcup_\theta X_\theta$. 
Let $Y \subset \SL(2,\RR)$ be the set of matrices $A \in \SL(2, \RR)$ that satisfy \eqref{eq:limsup1}. From remarks above we know the compliment of $Y$ has zero Harr measure. The previous paragraph implies that for any $\theta$ either $X_\theta \subset Y$ or $X_\theta \cap Y=\emptyset$. 
From the expression for Harr measure in terms of the KAN decomposition, we see that for almost every $\theta$, $X_\theta \subset Y$.
In particular, it follows that for almost every $\theta$, \eqref{eq:limsup1} holds with $A=R_\theta$. This is equivalent to the statement in the Lemma
when we have $\bu=R_\theta \begin{pmatrix}0 \\ 1 \end{pmatrix}$.
\end{proof}

\begin{proof}[Proof of \Cref{thm:divergence}]
Let $\bu$ be one of the almost-every directions such that \Cref{divergence main} applies. Recall from \Cref{prop:conjugacy} that $L_\bu=(\pi_\bS \circ \beta_\bu^{-1})^{-1}$ conjugates $F^t_{[\bu]}$ to $\hat F^t_\bu$. Furthermore, if $p \in \bN$ and $\bv \in \dir^{-1}(\bu)$ has basepoint $p$, we have 
$\hat \pi_\bN \circ L_\bu(\bv)=p$. Thus letting $q_T \in \bN$ denote the basepoint of $F^T_{[\bu]}(\bv)$, assuming $q_T$ is defined, we have
$$
\iota(q_T)=
\iota \circ \hat \pi_\bN \circ \hat F^T_{\bu} \circ L_\bu(\bv).$$
By \eqref{pairing meaning}, we further have
$$\iota(q)-\iota(p) = \iota \circ \hat \pi_\bN \circ \hat F^T_{\bu} \circ L_\bu(\bv)-
\iota \circ \hat \pi_\bN \circ L_\bu(\bv) = \big\langle \hat g, \gamma_{\bu, T}\big(L_\bu(\bv)\big) \big\rangle.$$
From \Cref{divergence main} we obtain
$$\limsup_{T \to +\infty} \frac{\log\|\iota(q_T)-\iota(p)\|_2}{\log T}=\frac{2}{3}.$$
By \Cref{quasi-isometry} we know $\iota$ is a quasi-isometry and so we get the same divergence rate for $F_{[\bu]}$ as stated in the theorem because of \Cref{div and qi}.
\end{proof}

\subsection*{Acknowledgments}

W. P. Hooper was supported by a grant from the Simons
Foundation and by a PSC-CUNY Award, jointly funded by The Professional Staff Congress and The City
University of New York. Work of Pavel Javornik on this project was supported by the Dr. Barnett and Jean Hollander Rich Mathematics Summer Internship Program at City College of New York.

\renewcommand{\UrlFont}{\ttfamily\tiny}

\bibliographystyle{amsalpha}
\bibliography{thebibliography}

\providecommand{\bysame}{\leavevmode\hbox to3em{\hrulefill}\thinspace}
\providecommand{\MR}{\relax\ifhmode\unskip\space\fi MR }
\providecommand{\MRhref}[2]{%
  \href{http://www.ams.org/mathscinet-getitem?mr=#1}{#2}
}
\providecommand{\href}[2]{#2}
\begin{thebibliography}{DGJdL23}

\bibitem[AAH22]{AAH}
Jayadev~S. Athreya, David Aulicino, and W.~Patrick Hooper, \emph{Platonic
  solids and high genus covers of lattice surfaces}, Experimental Mathematics
  \textbf{31} (2022), no.~3, 847--877, with an appendix by Anja Randecker.

\bibitem[AH17]{AH17}
Artur Avila and Pascal Hubert, \emph{Recurrence for the wind-tree model},
  Annales de l'Institut Henri Poincare (C) Non Linear Analysis, Elsevier, 2017.

\bibitem[AL21]{AL21}
Jayadev Athreya and Dami Lee, \emph{Translation covers of some triply periodic
  platonic surfaces}, Conformal Geometry and Dynamics of the American
  Mathematical Society \textbf{25} (2021), no.~2, 34--50.

\bibitem[Art17]{A17}
Mauro Artigiani, \emph{Exceptional ergodic directions in {E}aton lenses},
  Israel Journal of Mathematics \textbf{220} (2017), no.~1, 29--56.

\bibitem[Ber09]{Berger}
M.~Berger, \emph{Geometry {I}}, Universitext, Springer Berlin Heidelberg, 2009.

\bibitem[CM72]{CM1972}
H.~S.~M. Coxeter and W.~O.~J. Moser, \emph{Abstract crystallography},
  pp.~32--52, Springer Berlin Heidelberg, Berlin, Heidelberg, 1972.

\bibitem[Con92]{Conway}
J.H. Conway, \emph{The orbifold notation for surface groups}, London
  Mathematical Society Lecture Note Series, p.~438–447, Cambridge University
  Press, 1992.

\bibitem[Del13]{D13}
Vincent Delecroix, \emph{Divergent trajectories in the periodic wind-tree
  model}, Journal of Modern Dynamics \textbf{7} (2013), no.~1, 1--29.

\bibitem[DGJdL23]{DGJL}
Kishan~Kumar Dayaram, Amartya Goswami, Zurab Janelidze, and Tim~Van der Linden,
  \emph{Duality in non-abelian algebra v. homological diagram lemmas}, 2023.

\bibitem[DH18]{DH18}
Diana Davis and W.~Patrick Hooper, \emph{Periodicity and ergodicity in the
  trihexagonal tiling}, Comment. Math. Helv. \textbf{93} (2018), no.~4,
  661--707.

\bibitem[DHL14]{DHL14}
Vincent Delecroix, Pascal Hubert, and Samuel Leli\`evre, \emph{Diffusion for
  the periodic wind-tree model}, Annales Scientifiques de l'{\'E}cole Normale
  Sup{\'e}rieure, vol.~47, 2014, pp.~1085--1110.

\bibitem[Dun99]{dunbabin1999mosaics}
K.M.D. Dunbabin, \emph{Mosaics of the greek and roman world}, Cambridge
  University Press, 1999.

\bibitem[FK36]{FK36}
Ralph~H. Fox and Richard~B. Kershner, \emph{{Concerning the transitive
  properties of geodesics on a rational polyhedron}}, Duke Mathematical Journal
  \textbf{2} (1936), no.~1, 147 -- 150.

\bibitem[FK92]{FK92}
H.~M. Farkas and I.~Kra, \emph{Riemann surfaces}, second ed., Graduate Texts in
  Mathematics, vol.~71, Springer-Verlag, New York, 1992. \MR{1139765}

\bibitem[FS14]{FS14}
Krzysztof Frączek and Martin Schmoll, \emph{Directional localization of light
  rays in a periodic array of retro-reflector lenses}, Nonlinearity \textbf{27}
  (2014), no.~7, 1689.

\bibitem[FS18]{FS18}
\bysame, \emph{On ergodicity of foliations on $\mathbb{Z}^d$-covers of
  half-translation surfaces and some applications to periodic systems of
  {E}aton lenses}, Communications in Mathematical Physics \textbf{362} (2018),
  no.~2, 609--657.

\bibitem[FSU18]{FSU15}
Krzysztof Frączek, Ronggang Shi, and Corinna Ulcigrai, \emph{Genericity on
  curves and applications: pseudo-integrable billiards, eaton lenses and gap
  distributions}, Journal of Modern Dynamics \textbf{12} (2018), 55 -- 122.

\bibitem[FU14]{FU14}
Krzysztof Fr{\c{a}}czek and Corinna Ulcigrai, \emph{Non-ergodic
  $\mathbb{Z}$-periodic billiards and infinite translation surfaces},
  Inventiones mathematicae \textbf{197} (2014), no.~2, 241--298 (English).

\bibitem[GJ00]{GJ00}
Eugene Gutkin and Chris Judge, \emph{Affine mappings of translation surfaces:
  geometry and arithmetic}, Duke Math. J. \textbf{103} (2000), no.~2, 191--213.
  \MR{1760625}

\bibitem[HLT11]{HLT11}
Pascal Hubert, Samuel Lelievre, and Serge Troubetzkoy, \emph{The {E}hrenfest
  wind-tree model: periodic directions, recurrence, diffusion}, Journal f{\"u}r
  die reine und angewandte Mathematik (Crelles Journal) \textbf{2011} (2011),
  no.~656, 223--244.

\bibitem[Hoo15]{H15}
W.~Patrick Hooper, \emph{The invariant measures of some infinite interval
  exchange maps}, Geometry \& Topology \textbf{19} (2015), no.~4, 1895--2038.

\bibitem[Hoo23]{Hooper_Necker_App}
\bysame, \emph{Necker tiler}, \url{http://wphooper.com/visual/necker/}, Aug
  2023, Accessed: 3 Aug 2023.

\bibitem[HS06]{HubertSchmidt}
Pascal Hubert and Thomas~A Schmidt, \emph{An introduction to {V}eech surfaces},
  2006, pp.~501--526.

\bibitem[HW12]{HW12}
W.~Patrick Hooper and Barak Weiss, \emph{Generalized staircases: recurrence and
  symmetry}, Annales de l'institut Fourier \textbf{62} (2012), no.~4,
  1581--1600.

\bibitem[HW13]{HW13}
Pascal Hubert and Barak Weiss, \emph{Ergodicity for infinite periodic
  translation surfaces}, Compositio Mathematica \textbf{149} (2013), no.~8,
  1364--1380.

\bibitem[Jas99]{JJ99}
Joseph Jastrow, \emph{The mind's eye}, The Popular Science Monthly \textbf{54}
  (1899), no.~6, 299--318.

\bibitem[Mas88]{Masur88}
Howard Masur, \emph{Lower bounds for the number of saddle connections and
  closed trajectories of a quadratic differential}, Holomorphic Functions and
  Moduli I: Proceedings of a Workshop held March 13--19, 1986, Springer, 1988,
  pp.~215--228.

\bibitem[NE32]{LN32}
L.A. Necker~Esq., \emph{{LXI.} {O}bservations on some remarkable optical
  phænomena seen in {S}witzerland; and on an optical phænomenon which occurs
  on viewing a figure of a crystal or geometrical solid}, The London,
  Edinburgh, and Dublin Philosophical Magazine and Journal of Science
  \textbf{1} (1832), no.~5, 329--337.

\bibitem[ORSS21]{ORSS}
Andre Oliveira, Felipe Ramírez, Chandrika Sadanand, and Sunrose~T. Shrestha,
  \emph{Periodicity and symmetry in the mucube}, preprint available from
  \url{https://sites.google.com/view/sunroseshrestha/research}, 2021.

\bibitem[Sch10]{S10}
Doris Schattschneider, \emph{The mathematical side of {M}. {C}. {E}scher},
  Notices of the AMS \textbf{57} (2010), no.~6, 706--718.

\bibitem[Tab05]{Tabachnikov}
Serge Tabachnikov, \emph{Geometry and billiards}, vol.~30, American
  Mathematical Soc., 2005.

\bibitem[Thu88]{T88}
William~P. Thurston, \emph{On the geometry and dynamics of diffeomorphisms of
  surfaces}, Bull. Amer. Math. Soc. (N.S.) \textbf{19} (1988), no.~2, 417--431.
  \MR{956596 (89k:57023)}

\bibitem[Thu97]{Thurston}
\bysame, \emph{Three-dimensional geometry and topology. {V}ol. 1}, Princeton
  Mathematical Series, vol.~35, Princeton University Press, Princeton, NJ,
  1997, Edited by Silvio Levy. \MR{1435975 (97m:57016)}

\bibitem[Tro07]{Troyanov07}
Marc Troyanov, \emph{On the moduli space of singular euclidean surfaces},
  Handbook of Teichm{\"u}ller Theory, Volume I (2007), 507--540.

\bibitem[Yos97]{Yoshida}
Masaaki Yoshida, \emph{Hypergeometric functions, my love}, Aspects of
  Mathematics, E32, Friedr. Vieweg \& Sohn, Braunschweig, 1997, Modular
  interpretations of configuration spaces. \MR{1453580}

\end{thebibliography}

\end{document}